\documentclass{article}
\usepackage{etex}
\usepackage[utf8]{inputenc}
\usepackage[english]{babel}
\usepackage{latexsym,amsfonts,amssymb,exscale,enumerate,amsthm}
\usepackage{tikz,tikz-cd}
\usepackage{flafter}
\usepackage{float}
\usetikzlibrary{intersections}
\usepackage[centertags,sumlimits,intlimits,namelimits,reqno]{amsmath}
\usepackage{amsfonts}
\usepackage{amssymb}
\usepackage{bondgraph}
\usepackage[all,2cell]{xy} \UseAllTwocells
\usepackage[hyphens,spaces,obeyspaces]{url}

\usetikzlibrary{arrows,positioning,fit,matrix,shapes.geometric,external}

\newcommand*\pgfdeclareanchoralias[3]{%
  \expandafter\def\csname pgf@anchor@#1@#3\expandafter\endcsname
     \expandafter{\csname pgf@anchor@#1@#2\endcsname}}

\tikzset{every picture/.style={line width=1.5pt}}

\newtheorem{theorem}{Theorem}[section]
\newtheorem{proposition}[theorem]{Proposition}   
\newtheorem{conjecture}[theorem]{Conjecture}

\newtheorem{corollary}[theorem]{Corollary}  
\theoremstyle{definition}
\newtheorem{definition}[theorem]{Definition}   
\newtheorem{example}[theorem]{Example}
\newtheorem{remark}[theorem]{Remark}

\newcommand{\define}[1]{{\bf \boldmath #1}}

\usetikzlibrary{intersections,decorations.markings}
\usetikzlibrary{arrows,positioning,fit,matrix,shapes.geometric,external}
\usetikzlibrary{backgrounds,circuits,circuits.ee.IEC,shapes,fit,matrix}

\tikzset{
circnode/.style={
  circle, draw=red, very thin, outer sep=0.025em, minimum size=2em,
  fill=red, text centered},
integral/.style={
  regular polygon, regular polygon sides=3, shape border rotate=180, draw=black, very thick,
  outer sep=0.025em, inner sep=0, minimum size=2em, fill=blue!5, text centered},
multiply/.style={
  regular polygon, regular polygon sides=3, shape border rotate=180, draw=black, very thick,
  outer sep=0.025em, inner sep=0, minimum size=2em, fill=blue!5, text centered},
upmultiply/.style={
  regular polygon, regular polygon sides=3, draw=black, very thick,
  outer sep=0.025em, inner sep=0, minimum size=2em, fill=blue!5, text centered},
zero/.style={
  circle, draw=black, very thick, minimum size=0.15cm, fill=black,
  inner sep=0, outer sep=0},
hole/.style={
  circle, draw=white, very thick, minimum size=0.25cm, fill=white,
  inner sep=0, outer sep=0},
bang/.style={
  circle, draw=black, very thick, minimum size=0.15cm, fill=green!10,
  inner sep=0, outer sep=0},
delta/.style={
  regular polygon, regular polygon sides=3, minimum size=0.4cm, inner
  sep=0, outer sep=0.025em, draw=black, very thick, fill=green!10},
codelta/.style={
  regular polygon, regular polygon sides=3, shape border rotate=180, minimum size=0.4cm,
  inner sep=0, outer sep=0.025em, draw=black, very thick, fill=green!10},
plus/.style={
  regular polygon, regular polygon sides=3, shape border rotate=180, minimum size=0.4cm,
  inner sep = 0, outer sep=0.025em, draw=black, very thick, fill=black},
coplus/.style={
  regular polygon, regular polygon sides=3, minimum size=0.4cm,
  inner sep = 0, outer sep=0.025em, draw=black, very thick, fill=black},
sqnode/.style={
  regular polygon, regular polygon sides=4, minimum size=2.6em,
  draw=black, very thick, inner sep=0.2em, outer sep=0.025em,
  fill=yellow!10, text centered},
blackbox/.style={
  regular polygon, regular polygon sides=4, minimum size=2.6em,
  draw=black, very thick, inner sep=0.2em, outer sep=0.025em, fill=black},
bigcirc/.style={
  circle, draw=black, very thick, text width=1.6em, outer sep=0.025em,
  minimum height=1.6em, fill=blue!5, text centered}
 }
\pgfdeclareanchoralias{regular polygon}{corner 1}{io}
\pgfdeclareanchoralias{regular polygon}{corner 2}{left out}
\pgfdeclareanchoralias{regular polygon}{corner 3}{right out}
\pgfdeclareanchoralias{regular polygon}{corner 2}{left copy}
\pgfdeclareanchoralias{regular polygon}{corner 3}{right copy}
\pgfdeclareanchoralias{regular polygon}{corner 3}{addend}
\pgfdeclareanchoralias{regular polygon}{corner 2}{summand}
\pgfdeclareanchoralias{regular polygon}{corner 3}{left in}
\pgfdeclareanchoralias{regular polygon}{corner 2}{right in}

\usepackage{color}

\newcommand{\op}{\mathrm{op}}
\newcommand{\tensor}{\otimes}
\newcommand{\id}{\mathrm{id}}
\newcommand{\maps}{\colon}
\newcommand{\Ob}{\mathrm{Ob}}
\newcommand{\colim}{\mathrm{colim}}

\newcommand{\asrelto}{\nrightarrow}

\newcommand{\Rel}{\mathrm{Rel}}
\newcommand{\BGraph}{\mathrm{BGraph}}
\newcommand{\BondGraph}{\mathrm{BondGraph}}
\newcommand{\BG}{\mathrm{BG}}
\newcommand\Span{\mathrm{Span}}
\newcommand{\Corel}{\mathrm{Corel}}
\newcommand{\Cospan}{\mathrm{Cospan}}
\newcommand{\Fin}{\mathrm{Fin}}
\newcommand{\Set}{\mathrm{Set}}
\newcommand{\Cob}{\mathrm{Cob}}
\newcommand{\SigFlow}{\mathrm{SigFlow}}
\newcommand{\SigCirc}{\mathrm{SigCirc}}
\newcommand{\Lag}{\mathrm{Lag}}
\newcommand{\Aff}{\mathrm{Aff}}
\newcommand{\PROP}{\mathrm{PROP}}
\newcommand{\Symm}{\mathrm{Symm}}
\newcommand{\Mon}{\mathrm{Mon}}
\newcommand{\Cat}{\mathrm{Cat}}
\newcommand{\Vect}{\mathrm{Vect}}
\newcommand{\Circ}{{\mathrm{Circ}}}
\newcommand{\LCirc}{{\mathit{L}\mathrm{Circ}}}
\newcommand{\RLCCirc}{{\mathit{RLC}\mathrm{Circ}}}
\newcommand{\Prod}{\mathrm{Prod}}
\newcommand{\Th}{\mathrm{Th}}
\newcommand{\Mod}{\mathrm{Mod}}
\newcommand{\CMC}{\mathrm{CMC}}
\newcommand{\Nat}{\mathrm{Nat}}

\newcommand{\A}{A}
\newcommand{\C}{C}
\newcommand{\D}{D}
\newcommand{\T}{T}
\newcommand{\NN}{N}
\newcommand{\N}{\mathbb{N}}
\newcommand{\R}{\mathbb{R}}

\usepackage[hang,small,bf]{caption}
\definecolor{myurlcolor}{rgb}{0.6,0,0}
\definecolor{mycitecolor}{rgb}{0,0,0.8}
\definecolor{myrefcolor}{rgb}{0,0,0.8}
\usepackage[pagebackref]{hyperref}
\hypersetup{colorlinks, linkcolor=myrefcolor, citecolor=mycitecolor, urlcolor=myurlcolor}
\renewcommand*{\backref}[1]{(Referred to on page #1.)}

\tikzset{
    overdraw/.style={preaction={draw,white,line width=#1}},
    overdraw/.default=5pt
}

\usepackage{tikz,tikz-cd}

\usetikzlibrary{intersections,decorations.markings}
\usetikzlibrary{arrows,positioning,fit,matrix,shapes.geometric,external}
\usetikzlibrary{backgrounds,circuits,circuits.ee.IEC,shapes,fit,matrix}

\pgfdeclarelayer{edgelayer}
\pgfdeclarelayer{nodelayer}
\pgfsetlayers{edgelayer,nodelayer,main}

\tikzset{font=\footnotesize}
\tikzset{->-/.style={decoration={
  markings,
  mark=at position #1 with {\arrow{>}}},postaction={decorate}}}

\tikzstyle{none}=[inner sep=0pt]
\tikzstyle{circ}=[circle,fill=black,draw,inner sep=3pt]

\newcommand{\corelation}[1]
{
\begin{aligned}
\begin{tikzpicture}
[circuit ee IEC, set resistor graphic=var resistor IEC
      graphic,scale=.5]
\scalebox{1}{
		\node [style=none] (0) at (-1, -0) {};
		\node  [circle,draw,inner sep=1pt,fill]   (1) at (-0.125, -0) {};
		\node [style=none] (2) at (1, 0.5) {};
		\node [style=none] (3) at (1, -0.5) {};
		\node [style=none] (4) at (-1, -1.5) {};
		\node [style=none] (5) at (-1, -2.5) {};
		\node [style=none] (6) at (1, -1.5) {};
		\node [style=none] (7) at (1, -2.5) {};
		\node  [circle,draw,inner sep=1pt,fill]   (8) at (0, -2) {};
		\draw[line width = 1.5pt] (0.center) to (1.center);
		\draw[line width = 1.5pt] [in=180, out=60, looseness=1.20] (1.center) to (2.center);
		\draw[line width = 1.5pt] [in=180, out=-60, looseness=1.20] (1.center) to (3.center);
		\draw[line width = 1.5pt] [in=180, out=60, looseness=1.20] (8.center) to (6.center);
		\draw[line width = 1.5pt] [in=180, out=-60, looseness=1.20] (8.center) to (7.center);
		\draw[line width = 1.5pt] [in=0, out=120, looseness=1.20] (8.center) to (4.center);
		\draw[line width = 1.5pt] [in=0, out=-120, looseness=1.20] (8.center) to (5.center);
}
      \end{tikzpicture}
\end{aligned}
}

\newcommand{\mult}[1]
{
\begin{aligned}
\begin{tikzpicture}
[circuit ee IEC, set resistor graphic=var resistor IEC
      graphic,scale=.5]
\scalebox{1}{
		\node [style=none] (0) at (1, -0) {};
		\node  [circle,draw,inner sep=1pt,fill]   (1) at (0.125, -0) {};
		\node [style=none] (2) at (-1, 0.5) {};
		\node [style=none] (3) at (-1, -0.5) {};
		\draw[line width = 1.5pt] (0.center) to (1.center);
		\draw[line width = 1.5pt] [in=0, out=120, looseness=1.20] (1.center) to (2.center);
		\draw[line width = 1.5pt] [in=0, out=-120, looseness=1.20] (1.center) to (3.center);
}
      \end{tikzpicture}
\end{aligned}
}

\newcommand{\unit}[1]
{
  \begin{aligned}
\begin{tikzpicture}
[circuit ee IEC, set resistor graphic=var resistor IEC
      graphic,scale=.5]
\scalebox{1}{
		\node [style=none] (0) at (1, -0) {};
		\node [style=none] (1) at (-1, -0) {};
		\node  [circle,draw,inner sep=1pt,fill]   (2) at (0, -0) {};
		\draw[line width = 1.5pt] (0.center) to (2);
}
      \end{tikzpicture}
  \end{aligned}
}

\newcommand{\comult}[1]
{
\begin{aligned}
\begin{tikzpicture}
[circuit ee IEC, set resistor graphic=var resistor IEC
      graphic,scale=.5]
\scalebox{1}{
		\node [style=none] (0) at (-1, -0) {};
		\node  [circle,draw,inner sep=1pt,fill]   (1) at (-0.125, -0) {};
		\node [style=none] (2) at (1, 0.5) {};
		\node [style=none] (3) at (1, -0.5) {};
		\draw[line width = 1.5pt] (0.center) to (1.center);
		\draw[line width = 1.5pt] [in=180, out=60, looseness=1.20] (1.center) to (2.center);
		\draw[line width = 1.5pt] [in=180, out=-60, looseness=1.20] (1.center) to (3.center);
}
      \end{tikzpicture}
\end{aligned}
}

\newcommand{\counit}[1]
{
\begin{aligned}
\begin{tikzpicture}
[circuit ee IEC, set resistor graphic=var resistor IEC
      graphic,scale=.5]
\scalebox{1}{
		\node [style=none] (0) at (-1, -0) {};
		\node [style=none] (1) at (1, -0) {};
		\node  [circle,draw,inner sep=1pt,fill]   (2) at (0, -0) {};
		\draw[line width = 1.5pt] (0.center) to (2);
}
      \end{tikzpicture}
  \end{aligned}
}

\newcommand{\edge}[1]
{
  \begin{aligned}
\begin{tikzpicture}
[circuit ee IEC, set resistor graphic=var resistor IEC
      graphic,scale=.5]
\scalebox{1}{
		\node [style=none] (0) at (-1, -0) {};
		\node [style=none] (1) at (1, -0) {};
		\draw[line width = 1.5pt] (0.center) to (1)  node [above, pos=0.25] {\Huge $\ell$};
}
      \end{tikzpicture}
  \end{aligned}
}

\newcommand{\idone}[1]
{
  \begin{aligned}
\begin{tikzpicture}
[circuit ee IEC, set resistor graphic=var resistor IEC
      graphic,scale=.5]
\scalebox{1}{
		\node [style=none] (0) at (-1, -0) {};
		\node [style=none] (1) at (1, -0) {};
		\node [style=none] (2) at (0, 0.5) {};
		\node [style=none] (3) at (0, -0.5) {};
		\draw[line width = 1.5pt] (1.center) to (0.center);
}
\end{tikzpicture} 
  \end{aligned}
}

\newcommand{\idonezero}[1]
{
  \begin{aligned}
\begin{tikzpicture}
[circuit ee IEC, set resistor graphic=var resistor IEC
      graphic,scale=.5]
\scalebox{1}{
		\node [style=none] (0) at (-1, -0) {};
		\node [style=none] (1) at (1, -0) {};
		\node [style=none] (2) at (0, 0.5) {};
		\node [style=none] (3) at (0, -0.5) {};
}
\end{tikzpicture} 
  \end{aligned}
}

\newcommand{\identitytwo}[1]
{
  \begin{aligned}
\begin{tikzpicture}
[circuit ee IEC, set resistor graphic=var resistor IEC
      graphic,scale=.5]
\scalebox{1}{
		\node [style=none] (0) at (-.75, -.25) {};
		\node [style=none] (1) at (.75, -.25) {};
		\node [style=none] (2) at (-.75, 0.25) {};
		\node [style=none] (3) at (.75, 0.25) {};
		\draw[line width=1.5pt] (1.center) to (0.center);
		\draw[line width=1.5pt] (2.center) to (3.center);
}
\end{tikzpicture} 
  \end{aligned}
}

\newcommand{\two}[1]
{
  \begin{aligned}
\begin{tikzpicture}
[circuit ee IEC, set resistor graphic=var resistor IEC
      graphic,scale=.5]
\scalebox{1}{
		\node  [circle,draw,inner sep=1pt,fill]   (0) at (0,-.5) {};
		\node  [circle,draw,inner sep=1pt,fill]   (1) at (0, .5) {};
}
\end{tikzpicture}
  \end{aligned}
}

\newcommand{\singlegen}[1]
{
  \begin{aligned}
\begin{tikzpicture}
[circuit ee IEC, set resistor graphic=var resistor IEC
      graphic,scale=.5]
\scalebox{1}{
		\node [style=none] (0) at (-1, -0) {};
		\node [style=none] (1) at (1, -0) {};
		\node [style=none] (2) at (0, 0.5) {};
		\node [style=none] (3) at (0, -0.5) {};
		\draw[line width = 1.5pt] (1.center) to (0.center) node[midway,above] {\Huge $\ell$};
}
\end{tikzpicture} 
  \end{aligned}
}

\newcommand{\swap}[1]
{
  \begin{aligned}
    \resizebox{#1}{!}{
\begin{tikzpicture}
	\begin{pgfonlayer}{nodelayer}
		\node [style=none] (2) at (-0.5, -0.5) {};
		\node [style=none] (3) at (-2, 0.5) {};
		\node [style=none] (4) at (-0.5, 0.5) {};
		\node [style=none] (5) at (-2, -0.5) {};
	\end{pgfonlayer}
	\begin{pgfonlayer}{edgelayer}
		\draw[line width = 1.5pt] [in=180, out=0, looseness=1.00] (3.center) to (2.center);
		\draw[line width = 1.5pt] [in=0, out=180, looseness=1.00] (4.center) to (5.center);
	\end{pgfonlayer}
\end{tikzpicture}
    }
  \end{aligned}
}

\newcommand{\parmult}[1]
{
\begin{aligned}
\begin{tikzpicture}
[circuit ee IEC, set resistor graphic=var resistor IEC
      graphic,scale=.5]
\scalebox{1}{
		\node [style=none] (0) at (1, -0) {};
		\node  [circle,draw,inner sep=1pt,fill]   (1) at (0.125, -0) {};
		\node [style=none] (2) at (-1, 0.5) {};
		\node [style=none] (3) at (-1, -0.5) {};
		\node [style=none] (4) at (1, -.5) {};
		\node  [circle,draw,inner sep=1pt,fill]   (5) at (0.125, -.5) {};
		\node [style=none] (6) at (-1, -0) {};
		\node [style=none] (7) at (-1, -1) {};
		\draw[line width = 1.5pt] (0.center) to (1.center);
		\draw[line width = 1.5pt] [in=0, out=120, looseness=1.20] (1.center) to (2.center);
		\draw[line width = 1.5pt] [in=0, out=-120, looseness=1.20] (1.center) to (3.center);
		\draw[line width = 1.5pt] (4.center) to (5.center);
		\draw[line width = 1.5pt] [in=0, out=120, looseness=1.20] (5.center) to (6.center);
		\draw[line width = 1.5pt] [in=0, out=-120, looseness=1.20] (5.center) to (7.center);
}
      \end{tikzpicture}
\end{aligned}
}

\newcommand{\weirdmult}[1]
{
\begin{aligned}
\begin{tikzpicture}
[circuit ee IEC, set resistor graphic=var resistor IEC
      graphic,scale=.5]
\scalebox{1}{
		\node [style=none] (0) at (1, -0) {};
		\node  [circle,draw,inner sep=1pt,fill]   (1) at (0.125, -0) {};
		\node [style=none] (2) at (-1, 0.5) {};
		\node [style=none] (3) at (-1, -0.5) {};
		\node [style=none] (4) at (1, -.75) {};
		\node  [circle,draw,inner sep=1pt,fill]   (5) at (0.125, -.75) {};
		\node [style=none] (6) at (-1, -0) {};
		\node [style=none] (7) at (-1, -1.25) {};
		\draw[line width = 1.5pt] (0.center) to (1.center);
		\draw[line width = 1.5pt] [in=0, out=120, looseness=1.20] (1.center) to (2.center);
		\draw[line width = 1.5pt] [in=0, out=-120, looseness=1.20] (1.center) to (3.center);
		\draw[line width = 1.5pt] (4.center) to (5.center);
		\draw[line width = 1.5pt] [in=0, out=120, looseness=1.50] (5.center) to (7.center);
		\draw[line width = 1.5pt] [in=0, out=-120, looseness=1.50] (5.center) to (6.center);
}
      \end{tikzpicture}
\end{aligned}
}

\newcommand{\parcomult}[1]
{
\begin{aligned}
\begin{tikzpicture}[circuit ee IEC, set resistor graphic=var resistor IEC
      graphic,scale=.5]
\scalebox{1}{
\begin{scope}[rotate=180]
		\node [style=none] (0) at (1, -0) {};
		\node  [circle,draw,inner sep=1pt,fill]   (1) at (0.125, -0) {};
		\node [style=none] (2) at (-1, 0.5) {};
		\node [style=none] (3) at (-1, -0.5) {};
		\node [style=none] (4) at (1, -.5) {};
		\node  [circle,draw,inner sep=1pt,fill]   (5) at (0.125, -.5) {};
		\node [style=none] (6) at (-1, -0) {};
		\node [style=none] (7) at (-1, -1) {};
		\draw[line width = 1.5pt] (0.center) to (1.center);
		\draw[line width = 1.5pt] [in=0, out=120, looseness=1.20] (1.center) to (2.center);
		\draw[line width = 1.5pt] [in=0, out=-120, looseness=1.20] (1.center) to (3.center);
		\draw[line width = 1.5pt] (4.center) to (5.center);
		\draw[line width = 1.5pt] [in=0, out=120, looseness=1.20] (5.center) to (6.center);
		\draw[line width = 1.5pt] [in=0, out=-120, looseness=1.20] (5.center) to (7.center);
\end{scope}
}
      \end{tikzpicture}
\end{aligned}
}

\newcommand{\parunit}[1]
{
  \begin{aligned}
\begin{tikzpicture}[circuit ee IEC, set resistor graphic=var resistor IEC
      graphic,scale=.5]
\scalebox{1}{
		\node [style=none] (0) at (1, -0) {};
		\node [style=none] (1) at (-1, -0) {};
		\node  [circle,draw,inner sep=1pt,fill]   (2) at (0, -0) {};
		\node [style=none] (3) at (1, -.5) {};
		\node [style=none] (4) at (-1, -.5) {};
		\node  [circle,draw,inner sep=1pt,fill]   (5) at (0, -.5) {};
		\draw[line width = 1.5pt] (0.center) to (2);
		\draw[line width = 1.5pt] (3.center) to (5);
}
      \end{tikzpicture}
  \end{aligned}
}

\newcommand{\parcounit}[1]
{
  \begin{aligned}
\begin{tikzpicture}[circuit ee IEC, set resistor graphic=var resistor IEC
      graphic,scale=.5]
\scalebox{1}{
\begin{scope}[rotate=180]
		\node [style=none] (0) at (1, -0) {};
		\node [style=none] (1) at (-1, -0) {};
		\node  [circle,draw,inner sep=1pt,fill]   (2) at (0, -0) {};
		\node [style=none] (3) at (1, -.5) {};
		\node [style=none] (4) at (-1, -.5) {};
		\node  [circle,draw,inner sep=1pt,fill]   (5) at (0, -.5) {};
		\draw[line width = 1.5pt] (0.center) to (2);
		\draw[line width = 1.5pt] (3.center) to (5);
\end{scope}
}
      \end{tikzpicture}
  \end{aligned}
}

\newcommand{\assocl}[1]
{
  \begin{aligned}
\begin{tikzpicture}[circuit ee IEC, set resistor graphic=var resistor IEC
      graphic,scale=.5]
\scalebox{1}{
		\node  [circle,draw,inner sep=1pt,fill]   (0) at (0.125, -0) {};
		\node [style=none] (1) at (-1, 0.5) {};
		\node [style=none] (2) at (-1, -0.5) {};
		\node [style=none] (3) at (0, -1) {};
		\node [style=none] (4) at (2.25, -0.5) {};
		\node [style=none] (5) at (0.25, -0) {};
		\node  [circle,draw,inner sep=1pt,fill]   (6) at (1.25, -0.5) {};
		\node [style=none] (7) at (-1, -1) {};
		\draw[line width = 1.5pt] [in=0, out=120, looseness=1.20] (0.center) to (1.center);
		\draw[line width = 1.5pt] [in=0, out=-120, looseness=1.20] (0.center) to (2.center);
		\draw[line width = 1.5pt] (4.center) to (6);
		\draw[line width = 1.5pt] [in=0, out=120, looseness=1.20] (6) to (5.center);
		\draw[line width = 1.5pt] [in=0, out=-120, looseness=1.20] (6) to (3.center);
		\draw[line width = 1.5pt] (3.center) to (7.center);
}
      \end{tikzpicture}
  \end{aligned}
}

\newcommand{\assocr}[1]
{
  \begin{aligned}
\begin{tikzpicture}[circuit ee IEC, set resistor graphic=var resistor IEC
      graphic,scale=.5]
\scalebox{1}{
		\node  [circle,draw,inner sep=1pt,fill]   (0) at (0.125, -0.5) {};
		\node [style=none] (1) at (-1, -1) {};
		\node [style=none] (2) at (-1, 0) {};
		\node [style=none] (3) at (0, 0.5) {};
		\node [style=none] (4) at (2.25, 0) {};
		\node [style=none] (5) at (0.25, -0.5) {};
		\node  [circle,draw,inner sep=1pt,fill]   (6) at (1.25, 0) {};
		\node [style=none] (7) at (-1, 0.5) {};
		\draw[line width = 1.5pt] [in=0, out=-120, looseness=1.20] (0.center) to (1.center);
		\draw[line width = 1.5pt] [in=0, out=120, looseness=1.20] (0.center) to (2.center);
		\draw[line width = 1.5pt] (4.center) to (6);
		\draw[line width = 1.5pt] [in=0, out=-120, looseness=1.20] (6) to (5.center);
		\draw[line width = 1.5pt] [in=0, out=120, looseness=1.20] (6) to (3.center);
		\draw[line width = 1.5pt] (3.center) to (7.center);
}
      \end{tikzpicture}
  \end{aligned}
}

\newcommand{\coassocl}[1]
{
  \begin{aligned}
\begin{tikzpicture}[circuit ee IEC, set resistor graphic=var resistor IEC
      graphic,scale=.5]
\scalebox{1}{
		\node  [circle,draw,inner sep=1pt,fill]   (0) at (1.125, -0.5) {};
		\node [style=none] (1) at (2.25, -1) {};
		\node [style=none] (2) at (2.25, 0) {};
		\node [style=none] (3) at (1.25, 0.5) {};
		\node [style=none] (4) at (-1, 0) {};
		\node [style=none] (5) at (1, -0.5) {};
		\node  [circle,draw,inner sep=1pt,fill]   (6) at (0, 0) {};
		\node [style=none] (7) at (2.25, 0.5) {};
		\draw[line width = 1.5pt] [in=180, out=-60, looseness=1.20] (0.center) to (1.center);
		\draw[line width = 1.5pt] [in=180, out=60, looseness=1.20] (0.center) to (2.center);
		\draw[line width = 1.5pt] (4.center) to (6);
		\draw[line width = 1.5pt] [in=180, out=-60, looseness=1.20] (6) to (5.center);
		\draw[line width = 1.5pt] [in=180, out=60, looseness=1.20] (6) to (3.center);
		\draw[line width = 1.5pt] (3.center) to (7.center);
}
      \end{tikzpicture}
  \end{aligned}
}

\newcommand{\parassocl}[1]
{
\begin{aligned}
\begin{tikzpicture}[circuit ee IEC, set resistor graphic=var resistor IEC
      graphic,scale=.5]
\scalebox{1}{
		\node [style=none] (0) at (1, -0) {};
		\node  [circle,draw,inner sep=1pt,fill]   (1) at (0.125, -0) {};
		\node [style=none] (2) at (-1, 0.5) {}; %
		\node [style=none] (3) at (-1, -0.5) {};
		\node [style=none] (4) at (1, -.5) {};
		\node  [circle,draw,inner sep=1pt,fill]   (5) at (0.125, -.5) {};
		\node [style=none] (6) at (-1, -0) {}; %
		\node [style=none] (7) at (-1, -1) {};
		\node [style=none] (8) at (-3, .5) {}; 
		\node [style=none] (9) at (-3, 1) {};
		\draw[line width = 1.5pt] (0.center) to (1.center);
		\draw[line width = 1.5pt] [in=0, out=120, looseness=1.20] (1.center) to (2.west);
		\draw[line width = 1.5pt] [in=0, out=-120, looseness=1.20] (1.center) to (3.west);
		\draw[line width = 1.5pt] (4.center) to (5.center);
		\draw[line width = 1.5pt] [in=0, out=120, looseness=1.20] (5.center) to (6.west);
		\draw[line width = 1.5pt] [in=0, out=-120, looseness=1.20] (5.center) to (7.west);
		\draw[line width = 1.5pt] [in=0, out=180, looseness=1.5] (2.center) to (9.center);
		\draw[line width = 1.5pt] [in=0, out=180, looseness=1.5] (6.center) to (8.center);
\begin{scope}[shift={(-2,-.5)}]
		\node [style=none] (0) at (1, -0) {};
		\node  [circle,draw,inner sep=1pt,fill]   (1) at (0.125, -0) {};
		\node [style=none] (2) at (-1, 0.5) {};
		\node [style=none] (3) at (-1, -0.5) {};
		\node [style=none] (4) at (1, -.5) {};
		\node  [circle,draw,inner sep=1pt,fill]   (5) at (0.125, -.5) {};
		\node [style=none] (6) at (-1, -0) {};
		\node [style=none] (7) at (-1, -1) {};
		\draw[line width = 1.5pt] (0.center) to (1.center);
		\draw[line width = 1.5pt] [in=0, out=120, looseness=1.20] (1.center) to (2.center);
		\draw[line width = 1.5pt] [in=0, out=-120, looseness=1.20] (1.center) to (3.center);
		\draw[line width = 1.5pt] (4.center) to (5.center);
		\draw[line width = 1.5pt] [in=0, out=120, looseness=1.20] (5.center) to (6.center);
		\draw[line width = 1.5pt] [in=0, out=-120, looseness=1.20] (5.center) to (7.center);
\end{scope}
}
      \end{tikzpicture}
\end{aligned}
}

\newcommand{\parcoassocl}[1]
{
\begin{aligned}
\begin{tikzpicture}[circuit ee IEC, set resistor graphic=var resistor IEC
      graphic,scale=.5]
\scalebox{1}{
\begin{scope}[rotate=180]
		\node [style=none] (0) at (1, -0) {};
		\node  [circle,draw,inner sep=1pt,fill]   (1) at (0.125, -0) {};
		\node [style=none] (2) at (-1, 0.5) {}; %
		\node [style=none] (3) at (-1, -0.5) {};
		\node [style=none] (4) at (1, -.5) {};
		\node  [circle,draw,inner sep=1pt,fill]   (5) at (0.125, -.5) {};
		\node [style=none] (6) at (-1, -0) {}; %
		\node [style=none] (7) at (-1, -1) {};
		\node [style=none] (8) at (-3, .5) {}; 
		\node [style=none] (9) at (-3, 1) {};
		\draw[line width = 1.5pt] (0.center) to (1.center);
		\draw[line width = 1.5pt] [in=0, out=120, looseness=1.20] (1.center) to (2.east);
		\draw[line width = 1.5pt] [in=0, out=-120, looseness=1.20] (1.center) to (3.east);
		\draw[line width = 1.5pt] (4.center) to (5.center);
		\draw[line width = 1.5pt] [in=0, out=120, looseness=1.20] (5.center) to (6.east);
		\draw[line width = 1.5pt] [in=0, out=-120, looseness=1.20] (5.center) to (7.east);
		\draw[line width = 1.5pt] [in=0, out=180, looseness=1.5] (2.center) to (9.center);
		\draw[line width = 1.5pt] [in=0, out=180, looseness=1.5] (6.center) to (8.center);
\begin{scope}[shift={(-2,-.5)}]
		\node [style=none] (0) at (1, -0) {};
		\node  [circle,draw,inner sep=1pt,fill]   (1) at (0.125, -0) {};
		\node [style=none] (2) at (-1, 0.5) {};
		\node [style=none] (3) at (-1, -0.5) {};
		\node [style=none] (4) at (1, -.5) {};
		\node  [circle,draw,inner sep=1pt,fill]   (5) at (0.125, -.5) {};
		\node [style=none] (6) at (-1, -0) {};
		\node [style=none] (7) at (-1, -1) {};
		\draw[line width = 1.5pt] (0.center) to (1.center);
		\draw[line width = 1.5pt] [in=0, out=120, looseness=1.20] (1.center) to (2.center);
		\draw[line width = 1.5pt] [in=0, out=-120, looseness=1.20] (1.center) to (3.center);
		\draw[line width = 1.5pt] (4.center) to (5.center);
		\draw[line width = 1.5pt] [in=0, out=120, looseness=1.20] (5.center) to (6.center);
		\draw[line width = 1.5pt] [in=0, out=-120, looseness=1.20] (5.center) to (7.center);
\end{scope}
\end{scope}
}
      \end{tikzpicture}
\end{aligned}
}

\newcommand{\parassocr}[1]
{
\begin{aligned}
\begin{tikzpicture}[circuit ee IEC, set resistor graphic=var resistor IEC
      graphic,scale=.5]
\scalebox{1}{
\begin{scope}[yscale=-1]
		\node [style=none] (0) at (1, -0) {};
		\node  [circle,draw,inner sep=1pt,fill]   (1) at (0.125, -0) {};
		\node [style=none] (2) at (-1, 0.5) {}; %
		\node [style=none] (3) at (-1, -0.5) {};
		\node [style=none] (4) at (1, -.5) {};
		\node  [circle,draw,inner sep=1pt,fill]   (5) at (0.125, -.5) {};
		\node [style=none] (6) at (-1, -0) {}; %
		\node [style=none] (7) at (-1, -1) {};
		\node [style=none] (8) at (-3, .5) {}; 
		\node [style=none] (9) at (-3, 1) {};
		\draw[line width = 1.5pt] (0.center) to (1.center);
		\draw[line width = 1.5pt] [in=0, out=120, looseness=1.20] (1.center) to (2.west);
		\draw[line width = 1.5pt] [in=0, out=-120, looseness=1.20] (1.center) to (3.west);
		\draw[line width = 1.5pt] (4.center) to (5.center);
		\draw[line width = 1.5pt] [in=0, out=120, looseness=1.20] (5.center) to (6.west);
		\draw[line width = 1.5pt] [in=0, out=-120, looseness=1.20] (5.center) to (7.west);
		\draw[line width = 1.5pt] [in=0, out=180, looseness=1.5] (2.center) to (9.center);
		\draw[line width = 1.5pt] [in=0, out=180, looseness=1.5] (6.center) to (8.center);
\end{scope}
\begin{scope}[shift={(-2,.5)}, yscale=-1]
		\node [style=none] (0) at (1, -0) {};
		\node  [circle,draw,inner sep=1pt,fill]   (1) at (0.125, -0) {};
		\node [style=none] (2) at (-1, 0.5) {};
		\node [style=none] (3) at (-1, -0.5) {};
		\node [style=none] (4) at (1, -.5) {};
		\node  [circle,draw,inner sep=1pt,fill]   (5) at (0.125, -.5) {};
		\node [style=none] (6) at (-1, -0) {};
		\node [style=none] (7) at (-1, -1) {};
		\draw[line width = 1.5pt] (0.center) to (1.center);
		\draw[line width = 1.5pt] [in=0, out=120, looseness=1.20] (1.center) to (2.center);
		\draw[line width = 1.5pt] [in=0, out=-120, looseness=1.20] (1.center) to (3.center);
		\draw[line width = 1.5pt] (4.center) to (5.center);
		\draw[line width = 1.5pt] [in=0, out=120, looseness=1.20] (5.center) to (6.center);
		\draw[line width = 1.5pt] [in=0, out=-120, looseness=1.20] (5.center) to (7.center);
\end{scope}
}
      \end{tikzpicture}
\end{aligned}
}

\newcommand{\parcoassocr}[1]
{
\begin{aligned}
\begin{tikzpicture}[circuit ee IEC, set resistor graphic=var resistor IEC
      graphic,scale=.5]
\scalebox{1}{
\begin{scope}[rotate=180]
\begin{scope}[yscale=-1]
		\node [style=none] (0) at (1, -0) {};
		\node  [circle,draw,inner sep=1pt,fill]   (1) at (0.125, -0) {};
		\node [style=none] (2) at (-1, 0.5) {}; %
		\node [style=none] (3) at (-1, -0.5) {};
		\node [style=none] (4) at (1, -.5) {};
		\node  [circle,draw,inner sep=1pt,fill]   (5) at (0.125, -.5) {};
		\node [style=none] (6) at (-1, -0) {}; %
		\node [style=none] (7) at (-1, -1) {};
		\node [style=none] (8) at (-3, .5) {}; 
		\node [style=none] (9) at (-3, 1) {};
		\draw[line width = 1.5pt] (0.center) to (1.center);
		\draw[line width = 1.5pt] [in=0, out=120, looseness=1.20] (1.center) to (2.east);
		\draw[line width = 1.5pt] [in=0, out=-120, looseness=1.20] (1.center) to (3.east);
		\draw[line width = 1.5pt] (4.center) to (5.center);
		\draw[line width = 1.5pt] [in=0, out=120, looseness=1.20] (5.center) to (6.east);
		\draw[line width = 1.5pt] [in=0, out=-120, looseness=1.20] (5.center) to (7.east);
		\draw[line width = 1.5pt] [in=0, out=180, looseness=1.5] (2.center) to (9.center);
		\draw[line width = 1.5pt] [in=0, out=180, looseness=1.5] (6.center) to (8.center);
\end{scope}
\begin{scope}[shift={(-2,.5)}, yscale=-1]
		\node [style=none] (0) at (1, -0) {};
		\node  [circle,draw,inner sep=1pt,fill]   (1) at (0.125, -0) {};
		\node [style=none] (2) at (-1, 0.5) {};
		\node [style=none] (3) at (-1, -0.5) {};
		\node [style=none] (4) at (1, -.5) {};
		\node  [circle,draw,inner sep=1pt,fill]   (5) at (0.125, -.5) {};
		\node [style=none] (6) at (-1, -0) {};
		\node [style=none] (7) at (-1, -1) {};
		\draw[line width = 1.5pt] (0.center) to (1.center);
		\draw[line width = 1.5pt] [in=0, out=120, looseness=1.20] (1.center) to (2.center);
		\draw[line width = 1.5pt] [in=0, out=-120, looseness=1.20] (1.center) to (3.center);
		\draw[line width = 1.5pt] (4.center) to (5.center);
		\draw[line width = 1.5pt] [in=0, out=120, looseness=1.20] (5.center) to (6.center);
		\draw[line width = 1.5pt] [in=0, out=-120, looseness=1.20] (5.center) to (7.center);
\end{scope}
\end{scope}
}
      \end{tikzpicture}
\end{aligned}
}

\newcommand{\monadassocl}[1]
{
\begin{aligned}
\begin{tikzpicture}[circuit ee IEC, set resistor graphic=var resistor IEC
      graphic,scale=.5]
\scalebox{1}{
		\node [style=none] (0) at (1, -0) {};
		\node [style=none] (1) at (0.125, -0) {};
		\node [style=none] (2) at (-1, 0.5) {};
		\node [style=none] (3) at (-1, -0.5) {};
		\node [style=none] (4) at (1, -0) {};
		\node [style=none] (5) at (-1, 1) {};
		\node [style=none] (6) at (1, .25) {};
		\node [style=none] (7) at (-1, -1) {};
		\node [style=none] (8) at (1, -.25) {};
		\node [style=none] (9) at (1, 1.25) {};
		\node [style=none] (10) at (3, .5) {};
		\node [style=none] (11) at (1, 1.75) {};
		\node [style=none] (12) at (3, 1) {};
		\node [style=none] (13) at (-1, 2) {};
		\node [style=none] (14) at (-1, 2.5) {};
	\draw[line width = 1.5pt] (2) to [in=0, out=0,looseness = 3] (3);
	\draw[line width = 1.5pt] (6) to [in=0, out=0,looseness = 3] (9);
	\draw[line width = 1.5pt] (5) to [in=180, out=0,looseness = 1] (6.east);
	\draw[line width = 1.5pt] (7) to [in=180, out=0, looseness = 1] (8.east);
	\draw[line width = 1.5pt] (8) to [in=180, out=0, looseness = 1] (10);
	\draw[line width = 1.5pt] (11) to [in=180, out=0, looseness = 1] (12);
	\draw[line width = 1.5pt] (13) to [in=180, out=0, looseness = 1] (9.east);
	\draw[line width = 1.5pt] (14) to [in=180, out=0, looseness = 1] (11.east);
}
      \end{tikzpicture}
\end{aligned}
}

\newcommand{\monadcoassocl}[1]
{
\begin{aligned}
\begin{tikzpicture}[circuit ee IEC, set resistor graphic=var resistor IEC
      graphic,scale=.5]
\scalebox{1}{
\begin{scope}[rotate=180]
		\node [style=none] (0) at (1, -0) {};
		\node [style=none] (1) at (0.125, -0) {};
		\node [style=none] (2) at (-1, 0.5) {};
		\node [style=none] (3) at (-1, -0.5) {};
		\node [style=none] (4) at (1, -0) {};
		\node [style=none] (5) at (-1, 1) {};
		\node [style=none] (6) at (1, .25) {};
		\node [style=none] (7) at (-1, -1) {};
		\node [style=none] (8) at (1, -.25) {};
		\node [style=none] (9) at (1, 1.25) {};
		\node [style=none] (10) at (3, .5) {};
		\node [style=none] (11) at (1, 1.75) {};
		\node [style=none] (12) at (3, 1) {};
		\node [style=none] (13) at (-1, 2) {};
		\node [style=none] (14) at (-1, 2.5) {};
	\draw[line width = 1.5pt] (2) to [in=0, out=0,looseness = 3] (3.west);
	\draw[line width = 1.5pt] (6) to [in=0, out=0,looseness = 3] (9.west);
	\draw[line width = 1.5pt] (5) to [in=180, out=0,looseness = 1] (6.west);
	\draw[line width = 1.5pt] (7) to [in=180, out=0, looseness = 1] (8.west);
	\draw[line width = 1.5pt] (8) to [in=180, out=0, looseness = 1] (10);
	\draw[line width = 1.5pt] (11) to [in=180, out=0, looseness = 1] (12);
	\draw[line width = 1.5pt] (13) to [in=180, out=0, looseness = 1] (9.west);
	\draw[line width = 1.5pt] (14) to [in=180, out=0, looseness = 1] (11.west);
\end{scope}
}
      \end{tikzpicture}
\end{aligned}
}

\newcommand{\monadassocm}[1]
{
\begin{aligned}
\begin{tikzpicture}[circuit ee IEC, set resistor graphic=var resistor IEC
      graphic,scale=.5]
\scalebox{1}{
		\node [style=none] (0) at (1, -0) {};
		\node [style=none] (1) at (0.125, -0) {};
		\node [style=none] (2) at (-1, 0.5) {};
		\node [style=none] (3) at (-1, -0.5) {};
		\node [style=none] (4) at (1, -0) {};
		\node [style=none] (5) at (-1, 1) {};
		\node [style=none] (6) at (1, .25) {};
		\node [style=none] (7) at (-1, -1) {};
		\node [style=none] (8) at (1, -.25) {};
		\node [style=none] (9) at (1, 1.25) {};
		\node [style=none] (10) at (3, .5) {};
		\node [style=none] (11) at (1, 1.75) {};
		\node [style=none] (12) at (3, 1) {};
		\node [style=none] (13) at (-1, 2) {};
		\node [style=none] (14) at (-1, 2.5) {};
	\draw[line width = 1.5pt] (2) to [in=0, out=0,looseness = 3] (3);
	\draw[line width = 1.5pt] (5) to [in=0, out=0,looseness = 3] (13);
	\draw[line width = 1.5pt] (7) to [in=180, out=0, looseness = 1] (10);
	\draw[line width = 1.5pt] (14) to [in=180, out=0, looseness = 1] (12);
}
      \end{tikzpicture}
\end{aligned}
}

\newcommand{\monadcoassocm}[1]
{
\begin{aligned}
\begin{tikzpicture}[circuit ee IEC, set resistor graphic=var resistor IEC
      graphic,scale=.5]
\scalebox{1}{
\begin{scope}[rotate=180]
		\node [style=none] (0) at (1, -0) {};
		\node [style=none] (1) at (0.125, -0) {};
		\node [style=none] (2) at (-1, 0.5) {};
		\node [style=none] (3) at (-1, -0.5) {};
		\node [style=none] (4) at (1, -0) {};
		\node [style=none] (5) at (-1, 1) {};
		\node [style=none] (6) at (1, .25) {};
		\node [style=none] (7) at (-1, -1) {};
		\node [style=none] (8) at (1, -.25) {};
		\node [style=none] (9) at (1, 1.25) {};
		\node [style=none] (10) at (3, .5) {};
		\node [style=none] (11) at (1, 1.75) {};
		\node [style=none] (12) at (3, 1) {};
		\node [style=none] (13) at (-1, 2) {};
		\node [style=none] (14) at (-1, 2.5) {};
	\draw[line width = 1.5pt] (2) to [in=0, out=0,looseness = 3] (3.west);
	\draw[line width = 1.5pt] (5) to [in=0, out=0,looseness = 3] (13.west);
	\draw[line width = 1.5pt] (7) to [in=180, out=0, looseness = 1] (10);
	\draw[line width = 1.5pt] (14) to [in=180, out=0, looseness = 1] (12);
\end{scope}
}
      \end{tikzpicture}
\end{aligned}
}

\newcommand{\monadassocr}[1]
{
\begin{aligned}
\begin{tikzpicture}[circuit ee IEC, set resistor graphic=var resistor IEC
      graphic,scale=.5]
\scalebox{1}{
\begin{scope}[yscale=-1]
		\node [style=none] (0) at (1, -0) {};
		\node [style=none] (1) at (0.125, -0) {};
		\node [style=none] (2) at (-1, 0.5) {};
		\node [style=none] (3) at (-1, -0.5) {};
		\node [style=none] (4) at (1, -0) {};
		\node [style=none] (5) at (-1, 1) {};
		\node [style=none] (6) at (1, .25) {};
		\node [style=none] (7) at (-1, -1) {};
		\node [style=none] (8) at (1, -.25) {};
		\node [style=none] (9) at (1, 1.25) {};
		\node [style=none] (10) at (3, .5) {};
		\node [style=none] (11) at (1, 1.75) {};
		\node [style=none] (12) at (3, 1) {};
		\node [style=none] (13) at (-1, 2) {};
		\node [style=none] (14) at (-1, 2.5) {};
	\draw[line width = 1.5pt] (2) to [in=0, out=0,looseness = 3] (3);
	\draw[line width = 1.5pt] (6) to [in=0, out=0,looseness = 3] (9);
	\draw[line width = 1.5pt] (5) to [in=180, out=0,looseness = 1] (6.east);
	\draw[line width = 1.5pt] (7) to [in=180, out=0, looseness = 1] (8.east);
	\draw[line width = 1.5pt] (8) to [in=180, out=0, looseness = 1] (10);
	\draw[line width = 1.5pt] (11) to [in=180, out=0, looseness = 1] (12);
	\draw[line width = 1.5pt] (13) to [in=180, out=0, looseness = 1] (9.east);
	\draw[line width = 1.5pt] (14) to [in=180, out=0, looseness = 1] (11.east);
\end{scope}
}
      \end{tikzpicture}
\end{aligned}
}

\newcommand{\monadcoassocr}[1]
{
\begin{aligned}
\begin{tikzpicture}[circuit ee IEC, set resistor graphic=var resistor IEC
      graphic,scale=.5]
\scalebox{1}{
\begin{scope}[yscale=-1,rotate = 180]
		\node [style=none] (0) at (1, -0) {};
		\node [style=none] (1) at (0.125, -0) {};
		\node [style=none] (2) at (-1, 0.5) {};
		\node [style=none] (3) at (-1, -0.5) {};
		\node [style=none] (4) at (1, -0) {};
		\node [style=none] (5) at (-1, 1) {};
		\node [style=none] (6) at (1, .25) {};
		\node [style=none] (7) at (-1, -1) {};
		\node [style=none] (8) at (1, -.25) {};
		\node [style=none] (9) at (1, 1.25) {};
		\node [style=none] (10) at (3, .5) {};
		\node [style=none] (11) at (1, 1.75) {};
		\node [style=none] (12) at (3, 1) {};
		\node [style=none] (13) at (-1, 2) {};
		\node [style=none] (14) at (-1, 2.5) {};
	\draw[line width = 1.5pt] (2) to [in=0, out=0,looseness = 3] (3.west);
	\draw[line width = 1.5pt] (6) to [in=0, out=0,looseness = 3] (9.west);
	\draw[line width = 1.5pt] (5) to [in=180, out=0,looseness = 1] (6.west);
	\draw[line width = 1.5pt] (7) to [in=180, out=0, looseness = 1] (8.west);
	\draw[line width = 1.5pt] (8) to [in=180, out=0, looseness = 1] (10);
	\draw[line width = 1.5pt] (11) to [in=180, out=0, looseness = 1] (12);
	\draw[line width = 1.5pt] (13) to [in=180, out=0, looseness = 1] (9.west);
	\draw[line width = 1.5pt] (14) to [in=180, out=0, looseness = 1] (11.west);
\end{scope}
}
      \end{tikzpicture}
\end{aligned}
}

\newcommand{\coassocr}[1]
{
  \begin{aligned}
\begin{tikzpicture}[circuit ee IEC, set resistor graphic=var resistor IEC
      graphic,scale=.5]
\scalebox{1}{
		\node  [circle,draw,inner sep=1pt,fill]   (0) at (1.125, 0) {};
		\node [style=none] (1) at (2.25, 0.5) {};
		\node [style=none] (2) at (2.25, -0.5) {};
		\node [style=none] (3) at (1.25, -1) {};
		\node [style=none] (4) at (-1, -0.5) {};
		\node [style=none] (5) at (1, 0) {};
		\node  [circle,draw,inner sep=1pt,fill]   (6) at (0, -0.5) {};
		\node [style=none] (7) at (2.25, -1) {};
		\draw[line width = 1.5pt] [in=180, out=60, looseness=1.20] (0.center) to (1.center);
		\draw[line width = 1.5pt] [in=180, out=-60, looseness=1.20] (0.center) to (2.center);
		\draw[line width = 1.5pt] (4.center) to (6);
		\draw[line width = 1.5pt] [in=180, out=60, looseness=1.20] (6) to (5.center);
		\draw[line width = 1.5pt] [in=180, out=-60, looseness=1.20] (6) to (3.center);
		\draw[line width = 1.5pt] (3.center) to (7.center);
}
      \end{tikzpicture}
  \end{aligned}
}

\newcommand{\unitl}[1]
{
  \begin{aligned}
\begin{tikzpicture}[circuit ee IEC, set resistor graphic=var resistor IEC
      graphic,scale=.5]
\scalebox{1}{
		\node [style=none] (0) at (1, -0) {};
		\node  [circle,draw,inner sep=1pt,fill]   (1) at (0.125, -0) {};
		\node  [circle,draw,inner sep=1pt,fill]   (2) at (-1, 0.5) {};
		\node [style=none] (3) at (-1, -0.5) {};
		\node [style=none] (4) at (-2, -0.5) {};
		\draw[line width = 1.5pt] (0.center) to (1.center);
		\draw[line width = 1.5pt] [in=0, out=120, looseness=1.20] (1.center) to (2.center);
		\draw[line width = 1.5pt] [in=0, out=-120, looseness=1.20] (1.center) to (3.center);
		\draw[line width = 1.5pt] (4.center) to (3.center);
}
\end{tikzpicture}
  \end{aligned}
}

\newcommand{\unitr}[1]
{
  \begin{aligned}
\begin{tikzpicture}[circuit ee IEC, set resistor graphic=var resistor IEC
      graphic,scale=.5]
\scalebox{1}{
		\node [style=none] (0) at (1, -0) {};
		\node  [circle,draw,inner sep=1pt,fill]   (1) at (0.125, -0) {};
		\node  [circle,draw,inner sep=1pt,fill]   (2) at (-1, -0.5) {};
		\node [style=none] (3) at (-1, 0.5) {};
		\node [style=none] (4) at (-2, 0.5) {};
		\draw[line width = 1.5pt] (0.center) to (1.center);
		\draw[line width = 1.5pt] [in=0, out=120, looseness=1.20] (1.center) to (3.center);
		\draw[line width = 1.5pt] [in=0, out=-120, looseness=1.20] (1.center) to (2.center);
		\draw[line width = 1.5pt] (4.center) to (3.center);
}
\end{tikzpicture}
  \end{aligned}
}

\newcommand{\parunitl}[1]
{
\begin{aligned}
\begin{tikzpicture}[circuit ee IEC, set resistor graphic=var resistor IEC
      graphic,scale=.5]
\scalebox{1}{
		\node [style=none] (0) at (1, -0) {};
		\node  [circle,draw,inner sep=1pt,fill]   (1) at (0.125, -0) {};
		\node [style=none] (2) at (-1, 0.5) {};
		\node [style=none] (3) at (-1, -0.5) {}; 
		\node [style=none] (4) at (1, -.5) {};
		\node  [circle,draw,inner sep=1pt,fill]   (5) at (0.125, -.5) {};
		\node [style=none] (6) at (-1, -0) {};
		\node [style=none] (7) at (-1, -1) {}; 
		\node  [circle,draw,inner sep=1pt,fill]   (8) at (-1.5, -1) {}; 
		\node  [circle,draw,inner sep=1pt,fill]   (9) at (-1.5, -0.5) {}; 
		\node [style=none] (10) at (-2, .5) {};
		\node [style=none] (11) at (-2, 0) {}; 
		\draw[line width = 1.5pt] (0.center) to (1.center);
		\draw[line width = 1.5pt] [in=0, out=120, looseness=1.20] (1.center) to (2.center);
		\draw[line width = 1.5pt] [in=0, out=-120, looseness=1.20] (1.center) to (3.center);
		\draw[line width = 1.5pt] (4.center) to (5.center);
		\draw[line width = 1.5pt] [in=0, out=120, looseness=1.20] (5.center) to (6.center);
		\draw[line width = 1.5pt] [in=0, out=-120, looseness=1.20] (5.center) to (7.center);
		\draw[line width = 1.5pt] (7.east) to (8);
		\draw[line width = 1.5pt] (3.east) to (9);
		\draw[line width = 1.5pt] (2.east) to (10);
		\draw[line width = 1.5pt] (6.east) to (11);
}
      \end{tikzpicture}
\end{aligned}
}

\newcommand{\parcounitl}[1]
{
\begin{aligned}
\begin{tikzpicture}[circuit ee IEC, set resistor graphic=var resistor IEC
      graphic,scale=.5]
\scalebox{1}{
\begin{scope}[rotate=180]
		\node [style=none] (0) at (1, -0) {};
		\node  [circle,draw,inner sep=1pt,fill]   (1) at (0.125, -0) {};
		\node [style=none] (2) at (-1, 0.5) {};
		\node [style=none] (3) at (-1, -0.5) {}; 
		\node [style=none] (4) at (1, -.5) {};
		\node  [circle,draw,inner sep=1pt,fill]   (5) at (0.125, -.5) {};
		\node [style=none] (6) at (-1, -0) {};
		\node [style=none] (7) at (-1, -1) {}; 
		\node  [circle,draw,inner sep=1pt,fill]   (8) at (-1.5, -1) {}; 
		\node  [circle,draw,inner sep=1pt,fill]   (9) at (-1.5, -0.5) {}; 
		\node [style=none] (10) at (-2, .5) {};
		\node [style=none] (11) at (-2, 0) {}; 
		\draw[line width = 1.5pt] (0.center) to (1.center);
		\draw[line width = 1.5pt] [in=0, out=120, looseness=1.20] (1.center) to (2.center);
		\draw[line width = 1.5pt] [in=0, out=-120, looseness=1.20] (1.center) to (3.center);
		\draw[line width = 1.5pt] (4.center) to (5.center);
		\draw[line width = 1.5pt] [in=0, out=120, looseness=1.20] (5.center) to (6.center);
		\draw[line width = 1.5pt] [in=0, out=-120, looseness=1.20] (5.center) to (7.center);
		\draw[line width = 1.5pt] (7.west) to (8);
		\draw[line width = 1.5pt] (3.west) to (9);
		\draw[line width = 1.5pt] (2.west) to (10);
		\draw[line width = 1.5pt] (6.west) to (11);
\end{scope}
}
      \end{tikzpicture}
\end{aligned}
}

\newcommand{\parunitr}[1]
{
\begin{aligned}
\begin{tikzpicture}[circuit ee IEC, set resistor graphic=var resistor IEC
      graphic,scale=.5]
\scalebox{1}{
\begin{scope}[yscale=-1]
		\node [style=none] (0) at (1, -0) {};
		\node  [circle,draw,inner sep=1pt,fill]   (1) at (0.125, -0) {};
		\node [style=none] (2) at (-1, 0.5) {};
		\node [style=none] (3) at (-1, -0.5) {}; 
		\node [style=none] (4) at (1, -.5) {};
		\node  [circle,draw,inner sep=1pt,fill]   (5) at (0.125, -.5) {};
		\node [style=none] (6) at (-1, -0) {};
		\node [style=none] (7) at (-1, -1) {}; 
		\node  [circle,draw,inner sep=1pt,fill]   (8) at (-1.5, -1) {}; 
		\node  [circle,draw,inner sep=1pt,fill]   (9) at (-1.5, -0.5) {}; 
		\node [style=none] (10) at (-2, .5) {};
		\node [style=none] (11) at (-2, 0) {}; 
		\draw[line width = 1.5pt] (0.center) to (1.center);
		\draw[line width = 1.5pt] [in=0, out=120, looseness=1.20] (1.center) to (2.center);
		\draw[line width = 1.5pt] [in=0, out=-120, looseness=1.20] (1.center) to (3.center);
		\draw[line width = 1.5pt] (4.center) to (5.center);
		\draw[line width = 1.5pt] [in=0, out=120, looseness=1.20] (5.center) to (6.center);
		\draw[line width = 1.5pt] [in=0, out=-120, looseness=1.20] (5.center) to (7.center);
		\draw[line width = 1.5pt] (7.east) to (8);
		\draw[line width = 1.5pt] (3.east) to (9);
		\draw[line width = 1.5pt] (2.east) to (10);
		\draw[line width = 1.5pt] (6.east) to (11);
\end{scope}
}
      \end{tikzpicture}
\end{aligned}
}

\newcommand{\parcounitr}[1]
{
\begin{aligned}
\begin{tikzpicture}[circuit ee IEC, set resistor graphic=var resistor IEC
      graphic,scale=.5]
\scalebox{1}{
\begin{scope}[rotate=180]
\begin{scope}[yscale=-1]
		\node [style=none] (0) at (1, -0) {};
		\node  [circle,draw,inner sep=1pt,fill]   (1) at (0.125, -0) {};
		\node [style=none] (2) at (-1, 0.5) {};
		\node [style=none] (3) at (-1, -0.5) {}; 
		\node [style=none] (4) at (1, -.5) {};
		\node  [circle,draw,inner sep=1pt,fill]   (5) at (0.125, -.5) {};
		\node [style=none] (6) at (-1, -0) {};
		\node [style=none] (7) at (-1, -1) {}; 
		\node  [circle,draw,inner sep=1pt,fill]   (8) at (-1.5, -1) {}; 
		\node  [circle,draw,inner sep=1pt,fill]   (9) at (-1.5, -0.5) {}; 
		\node [style=none] (10) at (-2, .5) {};
		\node [style=none] (11) at (-2, 0) {}; 
		\draw[line width = 1.5pt] (0.center) to (1.center);
		\draw[line width = 1.5pt] [in=0, out=120, looseness=1.20] (1.center) to (2.center);
		\draw[line width = 1.5pt] [in=0, out=-120, looseness=1.20] (1.center) to (3.center);
		\draw[line width = 1.5pt] (4.center) to (5.center);
		\draw[line width = 1.5pt] [in=0, out=120, looseness=1.20] (5.center) to (6.center);
		\draw[line width = 1.5pt] [in=0, out=-120, looseness=1.20] (5.center) to (7.center);
		\draw[line width = 1.5pt] (7.west) to (8);
		\draw[line width = 1.5pt] (3.west) to (9);
		\draw[line width = 1.5pt] (2.west) to (10);
		\draw[line width = 1.5pt] (6.west) to (11);
\end{scope}
\end{scope}
}
      \end{tikzpicture}
\end{aligned}
}

\newcommand{\monadunitl}[1]
{
\begin{aligned}
\begin{tikzpicture}[circuit ee IEC, set resistor graphic=var resistor IEC
      graphic,scale=.5]
\scalebox{1}{
		\node [style=none] (0) at (1, -0) {};
		\node [style=none] (1) at (0.125, -0) {};
		\node [style=none] (2) at (-1, 0.5) {};
		\node [style=none] (3) at (-1, -0.5) {};
		\node [style=none] (4) at (1, -0) {};
		\node [style=none] (5) at (-1, 1) {};
		\node [style=none] (6) at (1, .25) {};
		\node [style=none] (7) at (-1, -1) {};
		\node [style=none] (8) at (1, -.25) {};
		\node [style=none] (9) at (-2, 1) {};
		\node [style=none] (10) at (-2, .5) {};
	\draw[line width = 1.5pt] (2) to [in=0, out=0,looseness = 3] (3);
	\draw[line width = 1.5pt] (3.east) to [in=180, out=180,looseness = 4] (7.east);
	\draw[line width = 1.5pt] (5) to [in=180, out=0,looseness = 1] (6);
	\draw[line width = 1.5pt] (7) to [in=180, out=0, looseness = 1] (8);
	\draw[line width = 1.5pt] (5.east) to  (9);
	\draw[line width = 1.5pt] (2.east) to  (10);
}
      \end{tikzpicture}
\end{aligned}
}

\newcommand{\monadcounitl}[1]
{
\begin{aligned}
\begin{tikzpicture}[circuit ee IEC, set resistor graphic=var resistor IEC
      graphic,scale=.5]
\scalebox{1}{
\begin{scope}[rotate=180]
		\node [style=none] (0) at (1, -0) {};
		\node [style=none] (1) at (0.125, -0) {};
		\node [style=none] (2) at (-1, 0.5) {};
		\node [style=none] (3) at (-1, -0.5) {};
		\node [style=none] (4) at (1, -0) {};
		\node [style=none] (5) at (-1, 1) {};
		\node [style=none] (6) at (1, .25) {};
		\node [style=none] (7) at (-1, -1) {};
		\node [style=none] (8) at (1, -.25) {};
		\node [style=none] (9) at (-2, 1) {};
		\node [style=none] (10) at (-2, .5) {};
	\draw[line width = 1.5pt] (2) to [in=0, out=0,looseness = 3] (3);
	\draw[line width = 1.5pt] (3.west) to [in=180, out=180,looseness = 4] (7.west);
	\draw[line width = 1.5pt] (5) to [in=180, out=0,looseness = 1] (6);
	\draw[line width = 1.5pt] (7) to [in=180, out=0, looseness = 1] (8);
	\draw[line width = 1.5pt] (5.west) to  (9);
	\draw[line width = 1.5pt] (2.west) to  (10);
\end{scope}
}
      \end{tikzpicture}
\end{aligned}
}

\newcommand{\monadunitr}[1]
{
\begin{aligned}
\begin{tikzpicture}[circuit ee IEC, set resistor graphic=var resistor IEC
      graphic,scale=.5]
\scalebox{1}{
\begin{scope}[yscale=-1]
		\node [style=none] (0) at (1, -0) {};
		\node [style=none] (1) at (0.125, -0) {};
		\node [style=none] (2) at (-1, 0.5) {};
		\node [style=none] (3) at (-1, -0.5) {};
		\node [style=none] (4) at (1, -0) {};
		\node [style=none] (5) at (-1, 1) {};
		\node [style=none] (6) at (1, .25) {};
		\node [style=none] (7) at (-1, -1) {};
		\node [style=none] (8) at (1, -.25) {};
		\node [style=none] (9) at (-2, 1) {};
		\node [style=none] (10) at (-2, .5) {};
	\draw[line width = 1.5pt] (2) to [in=0, out=0,looseness = 3] (3);
	\draw[line width = 1.5pt] (3.east) to [in=180, out=180,looseness = 4] (7.east);
	\draw[line width = 1.5pt] (5) to [in=180, out=0,looseness = 1] (6);
	\draw[line width = 1.5pt] (7) to [in=180, out=0, looseness = 1] (8);
	\draw[line width = 1.5pt] (5.east) to  (9);
	\draw[line width = 1.5pt] (2.east) to  (10);
\end{scope}
}
      \end{tikzpicture}
\end{aligned}
}

\newcommand{\monadcounitr}[1]
{
\begin{aligned}
\begin{tikzpicture}[circuit ee IEC, set resistor graphic=var resistor IEC
      graphic,scale=.5]
\scalebox{1}{
\begin{scope}[yscale=-1, rotate=180]
		\node [style=none] (0) at (1, -0) {};
		\node [style=none] (1) at (0.125, -0) {};
		\node [style=none] (2) at (-1, 0.5) {};
		\node [style=none] (3) at (-1, -0.5) {};
		\node [style=none] (4) at (1, -0) {};
		\node [style=none] (5) at (-1, 1) {};
		\node [style=none] (6) at (1, .25) {};
		\node [style=none] (7) at (-1, -1) {};
		\node [style=none] (8) at (1, -.25) {};
		\node [style=none] (9) at (-2, 1) {};
		\node [style=none] (10) at (-2, .5) {};
	\draw[line width = 1.5pt] (2) to [in=0, out=0,looseness = 3] (3);
	\draw[line width = 1.5pt] (3.west) to [in=180, out=180,looseness = 4] (7.west);
	\draw[line width = 1.5pt] (5) to [in=180, out=0,looseness = 1] (6);
	\draw[line width = 1.5pt] (7) to [in=180, out=0, looseness = 1] (8);
	\draw[line width = 1.5pt] (5.west) to  (9);
	\draw[line width = 1.5pt] (2.west) to  (10);
\end{scope}
}
      \end{tikzpicture}
\end{aligned}
}

\newcommand{\counitl}[1]
{
  \begin{aligned}
\begin{tikzpicture}[circuit ee IEC, set resistor graphic=var resistor IEC
      graphic,scale=.5]
\scalebox{1}{
		\node [style=none] (0) at (-2, -0) {};
		\node  [circle,draw,inner sep=1pt,fill]   (1) at (-1.125, -0) {};
		\node  [circle,draw,inner sep=1pt,fill]   (2) at (0, 0.5) {};
		\node [style=none] (3) at (0, -0.5) {};
		\node [style=none] (4) at (1, -0.5) {};
		\draw[line width = 1.5pt] (0.center) to (1.center);
		\draw[line width = 1.5pt] [in=180, out=60, looseness=1.20] (1.center) to (2.center);
		\draw[line width = 1.5pt] [in=180, out=-60, looseness=1.20] (1.center) to (3.center);
		\draw[line width = 1.5pt] (4.center) to (3.center);
}
\end{tikzpicture}
  \end{aligned}
}

\newcommand{\counitr}[1]
{
  \begin{aligned}
\begin{tikzpicture}[circuit ee IEC, set resistor graphic=var resistor IEC
      graphic,scale=.5]
\scalebox{1}{
		\node [style=none] (0) at (-2, -0) {};
		\node  [circle,draw,inner sep=1pt,fill]   (1) at (-1.125, -0) {};
		\node  [circle,draw,inner sep=1pt,fill]   (2) at (0, -0.5) {};
		\node [style=none] (3) at (0, 0.5) {};
		\node [style=none] (4) at (1, 0.5) {};
		\draw[line width = 1.5pt] (0.center) to (1.center);
		\draw[line width = 1.5pt] [in=180, out=60, looseness=1.20] (1.center) to (3.center);
		\draw[line width = 1.5pt] [in=180, out=-60, looseness=1.20] (1.center) to (2.center);
		\draw[line width = 1.5pt] (4.center) to (3.center);
}
\end{tikzpicture}
  \end{aligned}
}

\newcommand{\commute}[1]
{
  \begin{aligned}
\begin{tikzpicture}[circuit ee IEC, set resistor graphic=var resistor IEC
      graphic,scale=.5]
\scalebox{1}{
		\node [style=none] (0) at (1.25, -0) {};
		\node  [circle,draw,inner sep=1pt,fill]   (1) at (0.375, -0) {};
		\node [style=none] (2) at (-0.5, -0.5) {};
		\node [style=none] (3) at (-2, 0.5) {};
		\node [style=none] (4) at (-0.5, 0.5) {};
		\node [style=none] (5) at (-2, -0.5) {};
		\draw[line width = 1.5pt] (0.center) to (1.center);
		\draw[line width = 1.5pt] [in=0, out=-120, looseness=1.20] (1.center) to (2.center);
		\draw[line width = 1.5pt] [in=180, out=0, looseness=1.00] (3.center) to (2.center);
		\draw[line width = 1.5pt] [in=0, out=120, looseness=1.20] (1.center) to (4.center);
		\draw[line width = 1.5pt] [in=0, out=180, looseness=1.00] (4.center) to (5.center);
}
\end{tikzpicture}
  \end{aligned}
}

\newcommand{\cocommute}[1]
{
  \begin{aligned}
\begin{tikzpicture}[circuit ee IEC, set resistor graphic=var resistor IEC
      graphic,scale=.5]
\scalebox{1}{
		\node [style=none] (0) at (-2, -0) {};
		\node  [circle,draw,inner sep=1pt,fill]   (1) at (-1.125, -0) {};
		\node [style=none] (2) at (-0.25, -0.5) {};
		\node [style=none] (3) at (1.25, 0.5) {};
		\node [style=none] (4) at (-0.25, 0.5) {};
		\node [style=none] (5) at (1.25, -0.5) {};
		\draw[line width = 1.5pt] (0.center) to (1.center);
		\draw[line width = 1.5pt] [in=180, out=-60, looseness=1.20] (1.center) to (2.center);
		\draw[line width = 1.5pt] [in=0, out=180, looseness=1.00] (3.center) to (2.center);
		\draw[line width = 1.5pt] [in=180, out=60, looseness=1.20] (1.center) to (4.center);
		\draw[line width = 1.5pt] [in=180, out=0, looseness=1.00] (4.center) to (5.center);
}
\end{tikzpicture}
  \end{aligned}
}

\newcommand{\symmetric}[1]
{
  \begin{aligned}
\begin{tikzpicture}[circuit ee IEC, set resistor graphic=var resistor IEC
      graphic,scale=.5]
\scalebox{1}{
		\node [style=none] (0) at (1.25, -0) {};
		\node  [circle,draw,inner sep=1pt,fill]   (1) at (0.375, -0) {};
		\node [style=none] (2) at (-0.5, -0.5) {};
		\node [style=none] (3) at (-2, 0.5) {};
		\node [style=none] (4) at (-0.5, 0.5) {};
		\node [style=none] (5) at (-2, -0.5) {};
		\node  [circle,draw,inner sep=1pt,fill]   (6) at (1.25, -0) {};
		\draw[line width = 1.5pt] (0.center) to (1.center);
		\draw[line width = 1.5pt] [in=0, out=-120, looseness=1.20] (1.center) to (2.center);
		\draw[line width = 1.5pt] [in=180, out=0, looseness=1.00] (3.center) to (2.center);
		\draw[line width = 1.5pt] [in=0, out=120, looseness=1.20] (1.center) to (4.center);
		\draw[line width = 1.5pt] [in=0, out=180, looseness=1.00] (4.center) to (5.center);
}
\end{tikzpicture}
  \end{aligned}
}

\newcommand{\monadsymml}[1]
{
\begin{aligned}
\begin{tikzpicture}[circuit ee IEC, set resistor graphic=var resistor IEC
      graphic,scale=.5]
\scalebox{1}{
		\node [style=none] (0) at (1, -0) {};
		\node [style=none] (1) at (0.125, -0) {};
		\node [style=none] (2) at (-1, 0.5) {};
		\node [style=none] (3) at (-1, -0.5) {};
		\node [style=none] (4) at (1, -0) {};
		\node [style=none] (5) at (-1, 1) {};
		\node [style=none] (6) at (1, .25) {};
		\node [style=none] (7) at (-1, -1) {};
		\node [style=none] (8) at (1, -.25) {};
		\node [style=none] (10) at (-3, 1) {};
		\node [style=none] (11) at (-3, .5) {};
		\node [style=none] (12) at (-3, -.5) {};
		\node [style=none] (13) at (-3, -1) {};
	\draw[line width = 1.5pt] (2) to [in=0, out=0,looseness = 3] (3);
	\draw[line width = 1.5pt] (6.west) to [in=0, out=0,looseness = 3] (8.west);
	\draw[line width = 1.5pt] (5) to [in=180, out=0,looseness = 1] (6);
	\draw[line width = 1.5pt] (7) to [in=180, out=0, looseness = 1] (8);
	\draw[line width = 1.5pt] (5.east) to  [in=0, out=180, looseness = .5]   (12);
	\draw[line width = 1.5pt] (2.east) to [in=0, out=180, looseness = .5]  (13);
	\draw[line width = 1.5pt] (7.east) to  [in=0, out=180, looseness = .5] (11);
	\draw[line width = 1.5pt] (3.east) to [in=0, out=180, looseness = .5]  (10);
}
      \end{tikzpicture}
\end{aligned}
}

\newcommand{\monadsymmm}[1]
{
\begin{aligned}
\begin{tikzpicture}[circuit ee IEC, set resistor graphic=var resistor IEC
      graphic,scale=.5]
\scalebox{1}{
		\node [style=none] (0) at (1, -0) {};
		\node [style=none] (1) at (0.125, -0) {};
		\node [style=none] (2) at (-1, 0.5) {};
		\node [style=none] (3) at (-1, -0.5) {};
		\node [style=none] (4) at (1, -0) {};
		\node [style=none] (5) at (-1, 1) {};
		\node [style=none] (6) at (1, .25) {};
		\node [style=none] (7) at (-1, -1) {};
		\node [style=none] (8) at (1, -.25) {};
		\node [style=none] (10) at (-3, 1) {};
		\node [style=none] (11) at (-3, .5) {};
		\node [style=none] (12) at (-3, -.5) {};
		\node [style=none] (13) at (-3, -1) {};
	\draw[line width = 1.5pt] (2) to [in=0, out=0,looseness = 3] (3);
	\draw[line width = 1.5pt] (6.west) to [in=0, out=0,looseness = 3] (8.west);
	\draw[line width = 1.5pt] (5) to [in=180, out=0,looseness = 1] (6);
	\draw[line width = 1.5pt] (7) to [in=180, out=0, looseness = 1] (8);
	\draw[line width = 1.5pt] (5.east) to  [in=0, out=180, looseness = .5]   (12);
	\draw[line width = 1.5pt] (2.east) to [in=0, out=180, looseness = 1]  (10);
	\draw[line width = 1.5pt] (7.east) to  [in=0, out=180, looseness = .5] (11);
	\draw[line width = 1.5pt] (3.east) to [in=0, out=180, looseness = 1]  (13);
}
      \end{tikzpicture}
\end{aligned}
}

\newcommand{\monadsymmr}[1]
{
\begin{aligned}
\begin{tikzpicture}[circuit ee IEC, set resistor graphic=var resistor IEC
      graphic,scale=.5]
\scalebox{1}{
		\node [style=none] (0) at (1, -0) {};
		\node [style=none] (1) at (0.125, -0) {};
		\node [style=none] (2) at (-1, 0.5) {};
		\node [style=none] (3) at (-1, -0.5) {};
		\node [style=none] (4) at (1, -0) {};
		\node [style=none] (5) at (-1, 1) {};
		\node [style=none] (6) at (1, .25) {};
		\node [style=none] (7) at (-1, -1) {};
		\node [style=none] (8) at (1, -.25) {};
		\node [style=none] (10) at (-3, 1) {};
		\node [style=none] (11) at (-3, .5) {};
		\node [style=none] (12) at (-3, -.5) {};
		\node [style=none] (13) at (-3, -1) {};
	\draw[line width = 1.5pt] (2) to [in=0, out=0,looseness = 3] (3);
	\draw[line width = 1.5pt] (6.west) to [in=0, out=0,looseness = 3] (8.west);
	\draw[line width = 1.5pt] (5) to [in=180, out=0,looseness = 1] (6);
	\draw[line width = 1.5pt] (7) to [in=180, out=0, looseness = 1] (8);
	\draw[line width = 1.5pt] (5.east) to  [in=0, out=180, looseness = 1]   (11);
	\draw[line width = 1.5pt] (2.east) to [in=0, out=180, looseness = 1]  (10);
	\draw[line width = 1.5pt] (7.east) to  [in=0, out=180, looseness = 1] (12);
	\draw[line width = 1.5pt] (3.east) to [in=0, out=180, looseness = 1]  (13);
}
      \end{tikzpicture}
\end{aligned}
}

\newcommand{\monadsymmrr}[1]
{
\begin{aligned}
\begin{tikzpicture}[circuit ee IEC, set resistor graphic=var resistor IEC
      graphic,scale=.5]
\scalebox{1}{
		\node [style=none] (0) at (1, -0) {};
		\node [style=none] (1) at (0.125, -0) {};
		\node [style=none] (2) at (-1, 0.5) {};
		\node [style=none] (3) at (-1, -0.5) {};
		\node [style=none] (4) at (1, -0) {};
		\node [style=none] (5) at (-1, 1) {};
		\node [style=none] (6) at (1, .25) {};
		\node [style=none] (7) at (-1, -1) {};
		\node [style=none] (8) at (1, -.25) {};
		\node [style=none] (10) at (-3, 1) {};
		\node [style=none] (11) at (-3, .5) {};
		\node [style=none] (12) at (-3, -.5) {};
		\node [style=none] (13) at (-3, -1) {};
	\draw[line width = 1.5pt] (10) to [in=0, out=0,looseness = 4] (13);
	\draw[line width = 1.5pt] (5) to [in=0, out=0,looseness = 2] (7);
	\draw[line width = 1.5pt] (5.east) to  [in=0, out=180, looseness = 1.5]   (11);
	\draw[line width = 1.5pt] (7.east) to  [in=0, out=180, looseness = 1.5] (12);
}
      \end{tikzpicture}
\end{aligned}
}

\newcommand{\monadsymmcups}[1]
{
\begin{aligned}
\begin{tikzpicture}[circuit ee IEC, set resistor graphic=var resistor IEC
      graphic,scale=.5]
\scalebox{1}{
		\node [style=none] (-1) at (-4.5, -0) {};
		\node [style=none] (0) at (1, -0) {};
		\node [style=none] (1) at (0.125, -0) {};
		\node [style=none] (2) at (-1, 0.5) {};
		\node [style=none] (3) at (-1, -0.5) {};
		\node [style=none] (4) at (1, -0) {};
		\node [style=none] (5) at (-1, 1) {};
		\node [style=none] (6) at (1, .25) {};
		\node [style=none] (7) at (-1, -1) {};
		\node [style=none] (8) at (1, -.25) {};
		\node [style=none] (10) at (-3, 1) {};
		\node [style=none] (11) at (-3, .5) {};
		\node [style=none] (12) at (-3, -.5) {};
		\node [style=none] (13) at (-3, -1) {};
	\draw[line width = 1.5pt] (10) to [in=0, out=0,looseness = 4] (13);
	\draw[line width = 1.5pt] (11) to [in=0, out=0,looseness = 5] (12);
}
      \end{tikzpicture}
\end{aligned}
}

\newcommand{\monadsymmend}[1]
{
\begin{aligned}
\begin{tikzpicture}[circuit ee IEC, set resistor graphic=var resistor IEC
      graphic,scale=.5]
\scalebox{1}{
\begin{scope}[shift={(-1,0)}]
		\node [style=none] (-1) at (2.5, -0) {};
		\node [style=none] (0) at (1, -0) {};
		\node [style=none] (1) at (0.125, -0) {};
		\node [style=none] (2) at (-1, 0.5) {};
		\node [style=none] (3) at (-1, -0.5) {};
		\node [style=none] (4) at (1, -0) {};
		\node [style=none] (5) at (-1, 1) {};
		\node [style=none] (6) at (1, .25) {};
		\node [style=none] (7) at (-1, -1) {};
		\node [style=none] (8) at (1, -.25) {};
		\node [style=none] (10) at (-3, 1) {};
		\node [style=none] (11) at (-3, .5) {};
	\draw[line width = 1.5pt] (2) to [in=0, out=0,looseness = 3] (3);
	\draw[line width = 1.5pt] (6.west) to [in=0, out=0,looseness = 3] (8.west);
	\draw[line width = 1.5pt] (5) to [in=180, out=0,looseness = 1] (6);
	\draw[line width = 1.5pt] (7) to [in=180, out=0, looseness = 1] (8);
\end{scope}
}
      \end{tikzpicture}
\end{aligned}
}

\newcommand{\weirdthing}[1]
{
\begin{aligned}
\begin{tikzpicture}[circuit ee IEC, set resistor graphic=var resistor IEC
      graphic,scale=.5]
\scalebox{1}{
\begin{scope}[shift={(-1,0)}]
		\node [style=none] (-1) at (2.5, -0) {};
		\node [style=none] (0) at (1, -0) {};
		\node [style=none] (1) at (0.125, -0) {};
		\node [style=none] (2) at (-1, 0.5) {};
		\node [style=none] (3) at (-1, -0.5) {};
		\node [style=none] (4) at (1, -0) {};
		\node [style=none] (5) at (-1, 1) {};
		\node [style=none] (6) at (1, .25) {};
		\node [style=none] (7) at (-1, -1) {};
		\node [style=none] (8) at (1, -.25) {};
		\node [style=none] (10) at (-3, 1) {};
		\node [style=none] (11) at (-3, .5) {};
	\draw[line width = 1.5pt] (2) to [in=0, out=0,looseness = 3] (3);
	\draw[line width = 1.5pt] (6.west) to [in=0, out=0,looseness = 3] (8.west);
	\draw[line width = 1.5pt] (5) to [in=180, out=0,looseness = 1] (6);
	\draw[line width = 1.5pt] (7) to [in=180, out=0, looseness = 1] (8);
\end{scope}
\begin{scope}[rotate=180]
\begin{scope}[shift={(-3,0)}]
		\node [style=none] (-1) at (2.5, -0) {};
		\node [style=none] (0) at (1, -0) {};
		\node [style=none] (1) at (0.125, -0) {};
		\node [style=none] (2) at (-1, 0.5) {};
		\node [style=none] (3) at (-1, -0.5) {};
		\node [style=none] (4) at (1, -0) {};
		\node [style=none] (5) at (-1, 1) {};
		\node [style=none] (6) at (1, .25) {};
		\node [style=none] (7) at (-1, -1) {};
		\node [style=none] (8) at (1, -.25) {};
		\node [style=none] (10) at (-3, 1) {};
		\node [style=none] (11) at (-3, .5) {};
	\draw[line width = 1.5pt] (2) to [in=0, out=0,looseness = 3] (3);
	\draw[line width = 1.5pt] (6.west) to [in=0, out=0,looseness = 3] (8.west);
	\draw[line width = 1.5pt] (5) to [in=180, out=0,looseness = 1] (6);
	\draw[line width = 1.5pt] (7) to [in=180, out=0, looseness = 1] (8);
\end{scope}
\end{scope}
}
      \end{tikzpicture}
\end{aligned}
}

\newcommand{\multunit}[1]
{
\begin{aligned}
\begin{tikzpicture}[circuit ee IEC, set resistor graphic=var resistor IEC
      graphic,scale=.5]
\scalebox{1}{
		\node [style=none] (0) at (1, -0) {};
		\node  [circle,draw,inner sep=1pt,fill]   (1) at (0.125, -0) {};
		\node [style=none] (2) at (-1, 0.5) {};
		\node [style=none] (3) at (-1, -0.5) {};
		\node  [circle,draw,inner sep=1pt,fill]   (4) at (1, -0) {};
		\draw[line width = 1.5pt] (0.center) to (1.center);
		\draw[line width = 1.5pt] [in=0, out=120, looseness=1.20] (1.center) to (2.center);
		\draw[line width = 1.5pt] [in=0, out=-120, looseness=1.20] (1.center) to (3.center);
}
      \end{tikzpicture}
\end{aligned}
}

\newcommand{\cuptwo}[1]
{
\begin{aligned}
\begin{tikzpicture}[circuit ee IEC, set resistor graphic=var resistor IEC
      graphic,scale=.5]
\scalebox{1}{
		\node [style=none] (0) at (1, -0) {};
		\node [style=none] (1) at (0.125, -0) {};
		\node [style=none] (2) at (-1, 0.5) {};
		\node [style=none] (3) at (-1, -0.5) {};
		\node [style=none] (4) at (1, -0) {};
	\draw[line width = 1.5pt] (2) to [in=0, out=0,looseness = 4] (3);
}
      \end{tikzpicture}
\end{aligned}
}

\newcommand{\monadmult}[1]
{
\begin{aligned}
\begin{tikzpicture}[circuit ee IEC, set resistor graphic=var resistor IEC
      graphic,scale=.5]
\scalebox{1}{
		\node [style=none] (0) at (1, -0) {};
		\node [style=none] (1) at (0.125, -0) {};
		\node [style=none] (2) at (-1, 0.5) {};
		\node [style=none] (3) at (-1, -0.5) {};
		\node [style=none] (4) at (1, -0) {};
		\node [style=none] (5) at (-1, 1) {};
		\node [style=none] (6) at (1, .25) {};
		\node [style=none] (7) at (-1, -1) {};
		\node [style=none] (8) at (1, -.25) {};
	\draw[line width = 1.5pt] (2) to [in=0, out=0,looseness = 3] (3);
	\draw[line width = 1.5pt] (5) to [in=180, out=0,looseness = 1] (6);
	\draw[line width = 1.5pt] (7) to [in=180, out=0, looseness = 1] (8);
}
      \end{tikzpicture}
\end{aligned}
}

\newcommand{\monadcomult}[1]
{
\begin{aligned}
\begin{tikzpicture}[circuit ee IEC, set resistor graphic=var resistor IEC
      graphic,scale=.5]
\scalebox{1}{
\begin{scope}[rotate=180]
		\node [style=none] (0) at (1, -0) {};
		\node [style=none] (1) at (0.125, -0) {};
		\node [style=none] (2) at (-1, 0.5) {};
		\node [style=none] (3) at (-1, -0.5) {};
		\node [style=none] (4) at (1, -0) {};
		\node [style=none] (5) at (-1, 1) {};
		\node [style=none] (6) at (1, .25) {};
		\node [style=none] (7) at (-1, -1) {};
		\node [style=none] (8) at (1, -.25) {};
	\draw[line width = 1.5pt] (2) to [in=0, out=0,looseness = 3] (3.west);
	\draw[line width = 1.5pt] (5) to [in=180, out=0,looseness = 1] (6);
	\draw[line width = 1.5pt] (7) to [in=180, out=0, looseness = 1] (8);
\end{scope}
}
      \end{tikzpicture}
\end{aligned}
}

\newcommand{\cosymmetric}[1]
{
  \begin{aligned}
\begin{tikzpicture}[circuit ee IEC, set resistor graphic=var resistor IEC
      graphic,scale=.5]
\scalebox{1}{
		\node [style=none] (0) at (-2, -0) {};
		\node  [circle,draw,inner sep=1pt,fill]   (1) at (-1.125, -0) {};
		\node [style=none] (2) at (-0.25, -0.5) {};
		\node [style=none] (3) at (1.25, 0.5) {};
		\node [style=none] (4) at (-0.25, 0.5) {};
		\node [style=none] (5) at (1.25, -0.5) {};
		\node  [circle,draw,inner sep=1pt,fill]   (6) at (-2, -0) {};
		\draw[line width = 1.5pt] (0.center) to (1.center);
		\draw[line width = 1.5pt] [in=180, out=-60, looseness=1.20] (1.center) to (2.center);
		\draw[line width = 1.5pt] [in=0, out=180, looseness=1.00] (3.center) to (2.center);
		\draw[line width = 1.5pt] [in=180, out=60, looseness=1.20] (1.center) to (4.center);
		\draw[line width = 1.5pt] [in=180, out=0, looseness=1.00] (4.center) to (5.center);
}
\end{tikzpicture}
  \end{aligned}
}

\newcommand{\comultcounit}[1]
{
\begin{aligned}
\begin{tikzpicture}
[circuit ee IEC, set resistor graphic=var resistor IEC
      graphic,scale=.5]
\scalebox{1}{
		\node [style=none] (0) at (-1, -0) {};
		\node  [circle,draw,inner sep=1pt,fill]   (1) at (-0.125, -0) {};
		\node [style=none] (2) at (1, 0.5) {};
		\node [style=none] (3) at (1, -0.5) {};
		\node  [circle,draw,inner sep=1pt,fill]   (4) at (-1, -0) {};
		\draw[line width = 1.5pt] (0.center) to (1.center);
		\draw[line width = 1.5pt] [in=180, out=60, looseness=1.20] (1.center) to (2.center);
		\draw[line width = 1.5pt] [in=180, out=-60, looseness=1.20] (1.center) to (3.center);
}
      \end{tikzpicture}
\end{aligned}
}

\newcommand{\captwo}[1]
{
\begin{aligned}
\begin{tikzpicture}
[circuit ee IEC, set resistor graphic=var resistor IEC
      graphic,scale=.5]
\scalebox{1}{
		\node [style=none] (0) at (-.75, -0) {};
		\node [style=none] (1) at (-0.125, -0) {};
		\node [style=none] (2) at (1, 0.5) {};
		\node [style=none] (3) at (1, -0.5) {};
		\node [style=none] (4) at (-1, -0) {};
		\node [style=none] (5) at (-0.5, -.5) {};
		\node [style=none] (6) at (-0.5, .5) {};
	\draw[line width = 1.5pt] (2) to [in=180, out=180,looseness = 4] (3);
}
      \end{tikzpicture}
\end{aligned}
}

\newcommand{\capn}[1]
{
\begin{aligned}
\begin{tikzpicture}
[circuit ee IEC, set resistor graphic=var resistor IEC
      graphic,scale=.5]
\scalebox{1}{
		\node [style=none] (0) at (-.75, -0) {};
		\node [style=none] (1) at (-0.125, -0) {};
		\node [style=none] (2) at (1, 0.5) {};
		\node [style=none] (3) at (1, -0.5) {};
		\node [style=none] (4) at (-1, -0) {};
		\node [style=none] (5) at (-0.5, -.5) {};
		\node [style=none] (6) at (-0.5, .5) {};
		\node [style=none] (7) at (1, 1) {};
		\node [style=none] (8) at (1, -1) {};
	\draw[line width = 1.5pt] (2) to [in=180, out=180,looseness = 4] (3);
	\draw[line width = 1.5pt] (7) to [in=180, out=180,looseness = 3] (8);
}
      \end{tikzpicture}
\end{aligned}
}

\newcommand{\zigzag}[1]
{
\begin{aligned}
\begin{tikzpicture}[circuit ee IEC, set resistor graphic=var resistor IEC
      graphic,scale=.5]
\scalebox{1}{
		\node [style=none] (0) at (-1, -0) {};
		\node  [circle,draw,inner sep=1pt,fill]   (1) at (-0.125, -0) {};
		\node [style=none] (2) at (1, 0.5) {};
		\node [style=none] (3) at (1, -0.5) {};
		\node [style=circ, draw,inner sep=1pt,fill] (4) at (-1, -0) {};
		\node [style=none] (5) at (1, -0.5) {};
		\node [style=none] (6) at (1, -1.5) {};
		\node  [circle,draw,inner sep=1pt,fill]   (7) at (2.125, -1) {};
		\node  [circle,draw,inner sep=1pt,fill]   (8) at (3, -1) {};
		\node [style=none] (9) at (3, -1) {};
		\node [style=none] (10) at (3, .5) {};
		\node [style=none] (11) at (-1, -1.5) {};
		\draw[line width=1pt] (0.center) to (1.center);
		\draw[line width=1pt] [in=180, out=60, looseness=1.20] (1.center) to (2.center);
		\draw[line width=1pt] [in=180, out=-60, looseness=1.20] (1.center) to (3.center);
		\draw[line width=1pt] [in=0, out=120, looseness=1.20] (7.center) to (5.center);
		\draw[line width=1pt] [in=0, out=-120, looseness=1.20] (7.center) to (6.center);
		\draw[line width=1pt] (9.center) to (7.center);
		\draw[line width=1pt] (10.center) to (2.center);
		\draw[line width=1pt] (11.center) to (6.center);
}
      \end{tikzpicture}
\end{aligned}
}

\newcommand{\zigzaglaw}[1]
{
\begin{aligned}
\begin{tikzpicture}
[circuit ee IEC, set resistor graphic=var resistor IEC
      graphic,scale=.5]
\scalebox{1}{
		\node [style=none] (2) at (1, 0.5) {};
		\node [style=none] (3) at (1, -0.5) {};
		\node [style=none] (4) at (1, -1.5) {};
		\node [style=none] (5) at (2.25, 0.5) {};
		\node [style=none] (6) at (-.25, -1.5) {};
	\draw[line width = 1.5pt] (2) to [in=180, out=180,looseness = 4] (3.east);
	\draw[line width = 1.5pt] (3) to [in=0, out=0,looseness = 4] (4);
	\draw[line width = 1.5pt] (2.west) to (5.west);
	\draw[line width = 1.5pt] (4.east) to  (6.east);
}
      \end{tikzpicture}
\end{aligned}
}

\newcommand{\zigzaglawother}[1]
{
\begin{aligned}
\begin{tikzpicture}
[circuit ee IEC, set resistor graphic=var resistor IEC
      graphic,scale=.5]
\scalebox{1}{
		\node [style=none] (2) at (1, 0.5) {};
		\node [style=none] (3) at (1, -0.5) {};
		\node [style=none] (4) at (1, -1.5) {};
		\node [style=none] (5) at (-.25, 0.5) {};
		\node [style=none] (6) at (2.25, -1.5) {};
	\draw[line width = 1.5pt] (2) to [in=0, out=0,looseness = 4] (3.west);
	\draw[line width = 1.5pt] (3) to [in=180, out=180,looseness = 4] (4);
	\draw[line width = 1.5pt] (2.east) to (5.east);
	\draw[line width = 1.5pt] (4.west) to  (6.west);
}
      \end{tikzpicture}
\end{aligned}
}

\newcommand{\frobs}[1]
{
  \begin{aligned}
\begin{tikzpicture}[circuit ee IEC, set resistor graphic=var resistor IEC
      graphic,scale=.5]
\scalebox{1}{
		\node [style=none] (0) at (-1.5, 0.5) {};
		\node  [circle,draw,inner sep=1pt,fill]   (1) at (-0.75, 0.5) {};
		\node [style=none] (2) at (0.25, -0) {};
		\node [style=none] (3) at (0.25, 1) {};
		\node  [circle,draw,inner sep=1pt,fill]   (4) at (1, -0.5) {};
		\node [style=none] (5) at (0, -0) {};
		\node [style=none] (6) at (1.75, -0.5) {};
		\node [style=none] (7) at (0, -1) {};
		\node [style=none] (8) at (1.75, 1) {};
		\node [style=none] (9) at (-1.5, -1) {};
		\draw[line width = 1.5pt] [in=180, out=-60, looseness=1.20] (1) to (2.center);
		\draw[line width = 1.5pt] [in=180, out=60, looseness=1.20] (1) to (3.center);
		\draw[line width = 1.5pt] (0.center) to (1);
		\draw[line width = 1.5pt] (6.center) to (4);
		\draw[line width = 1.5pt] [in=0, out=120, looseness=1.20] (4) to (5.center);
		\draw[line width = 1.5pt] [in=0, out=-120, looseness=1.20] (4) to (7.center);
		\draw[line width = 1.5pt] (3.center) to (8.center);
		\draw[line width = 1.5pt] (7.center) to (9.center);
}
\end{tikzpicture}
  \end{aligned}
}

\newcommand{\frobx}[1]
{
  \begin{aligned}
\begin{tikzpicture}[circuit ee IEC, set resistor graphic=var resistor IEC
      graphic,scale=.5]
\scalebox{1}{
		\node  [circle,draw,inner sep=1pt,fill]   (0) at (-0.5, -0) {};
		\node [style=none] (1) at (-1.5, -0.5) {};
		\node [style=none] (2) at (-1.5, 0.5) {};
		\node  [circle,draw,inner sep=1pt,fill]   (3) at (0.5, -0) {};
		\node [style=none] (4) at (1.5, -0.5) {};
		\node [style=none] (5) at (1.5, 0.5) {};
		\draw[line width = 1.5pt] [in=0, out=-120, looseness=1.20] (0.center) to (1.center);
		\draw[line width = 1.5pt] [in=0, out=120, looseness=1.20] (0.center) to (2.center);
		\draw[line width = 1.5pt] [in=180, out=-60, looseness=1.20] (3) to (4.center);
		\draw[line width = 1.5pt] [in=180, out=60, looseness=1.20] (3) to (5.center);
		\draw[line width = 1.5pt] (0) to (3);
}
\end{tikzpicture}
  \end{aligned}
}

\newcommand{\frobz}[1]
{
  \begin{aligned}
\begin{tikzpicture}[circuit ee IEC, set resistor graphic=var resistor IEC
      graphic,scale=.5]
\scalebox{1}{
		\node [style=none] (0) at (1.75, 0.5) {};
		\node  [circle,draw,inner sep=1pt,fill]   (1) at (1, 0.5) {};
		\node [style=none] (2) at (0, -0) {};
		\node [style=none] (3) at (0, 1) {};
		\node  [circle,draw,inner sep=1pt,fill]   (4) at (-0.75, -0.5) {};
		\node [style=none] (5) at (0.25, -0) {};
		\node [style=none] (6) at (-1.5, -0.5) {};
		\node [style=none] (7) at (0.25, -1) {};
		\node [style=none] (8) at (-1.5, 1) {};
		\node [style=none] (9) at (1.75, -1) {};
		\draw[line width = 1.5pt] [in=0, out=-120, looseness=1.20] (1) to (2.center);
		\draw[line width = 1.5pt] [in=0, out=120, looseness=1.20] (1) to (3.center);
		\draw[line width = 1.5pt] (0.center) to (1);
		\draw[line width = 1.5pt] (6.center) to (4);
		\draw[line width = 1.5pt] [in=180, out=60, looseness=1.20] (4) to (5.center);
		\draw[line width = 1.5pt] [in=180, out=-60, looseness=1.20] (4) to (7.center);
		\draw[line width = 1.5pt] (3.center) to (8.center);
		\draw[line width = 1.5pt] (7.center) to (9.center);
}
\end{tikzpicture}
  \end{aligned}
}

\newcommand{\parfrobl}[1]
{
\begin{aligned}
\begin{tikzpicture}[circuit ee IEC, set resistor graphic=var resistor IEC
      graphic,scale=.5]
\scalebox{1}{
		\node [style=none] (0) at (1, -0) {};
		\node  [circle,draw,inner sep=1pt,fill]   (1) at (0.125, -0) {};
		\node [style=none] (2) at (-1, 0.5) {}; %
		\node [style=none] (3) at (-1, -0.5) {};
		\node [style=none] (4) at (1, -.5) {};
		\node  [circle,draw,inner sep=1pt,fill]   (5) at (0.125, -.5) {};
		\node [style=none] (6) at (-1, -0) {}; %
		\node [style=none] (7) at (-1, -1) {};
		\node [style=none] (8) at (-3, 0) {}; 
		\node [style=none] (9) at (-3, .5) {};
		\draw[line width = 1.5pt] (0.center) to (1.center);
		\draw[line width = 1.5pt] [in=0, out=120, looseness=1.20] (1.center) to (2.west);
		\draw[line width = 1.5pt] [in=0, out=-120, looseness=1.20] (1.center) to (3.west);
		\draw[line width = 1.5pt] (4.center) to (5.center);
		\draw[line width = 1.5pt] [in=0, out=120, looseness=1.20] (5.center) to (6.west);
		\draw[line width = 1.5pt] [in=0, out=-120, looseness=1.20] (5.center) to (7.west);
		\draw[line width = 1.5pt] [in=0, out=180] (2.center) to (9.center);
		\draw[line width = 1.5pt] [in=0, out=180] (6.center) to (8.center);
\begin{scope}[shift={(-2,-1.5)}, rotate=180]
		\node [style=none] (0) at (1, -0) {};
		\node  [circle,draw,inner sep=1pt,fill]   (1) at (0.125, -0) {};
		\node [style=none] (2) at (-1, 0.5) {}; %
		\node [style=none] (3) at (-1, -0.5) {};
		\node [style=none] (4) at (1, -.5) {};
		\node  [circle,draw,inner sep=1pt,fill]   (5) at (0.125, -.5) {};
		\node [style=none] (6) at (-1, -0) {}; %
		\node [style=none] (7) at (-1, -1) {};
		\node [style=none] (8) at (-3, 0) {}; 
		\node [style=none] (9) at (-3, .5) {};
		\draw[line width = 1.5pt] (0.center) to (1.center);
		\draw[line width = 1.5pt] [in=0, out=120, looseness=1.20] (1.center) to (2.east);
		\draw[line width = 1.5pt] [in=0, out=-120, looseness=1.20] (1.center) to (3.east);
		\draw[line width = 1.5pt] (4.center) to (5.center);
		\draw[line width = 1.5pt] [in=0, out=120, looseness=1.20] (5.center) to (6.east);
		\draw[line width = 1.5pt] [in=0, out=-120, looseness=1.20] (5.center) to (7.east);
		\draw[line width = 1.5pt] [in=0, out=180] (2.center) to (9.center);
		\draw[line width = 1.5pt] [in=0, out=180] (6.center) to (8.center);
\end{scope}
}
      \end{tikzpicture}
\end{aligned}
}

\newcommand{\parfrobm}[1]
{
\begin{aligned}
\begin{tikzpicture}[circuit ee IEC, set resistor graphic=var resistor IEC
      graphic,scale=.5]
\scalebox{1}{
		\node [style=none] (0) at (1, -0) {};
		\node  [circle,draw,inner sep=1pt,fill]   (1) at (0.125, -0) {};
		\node [style=none] (2) at (-1, 0.5) {};
		\node [style=none] (3) at (-1, -0.5) {};
		\node [style=none] (4) at (1, -.5) {};
		\node  [circle,draw,inner sep=1pt,fill]   (5) at (0.125, -.5) {};
		\node [style=none] (6) at (-1, -0) {};
		\node [style=none] (7) at (-1, -1) {};
		\draw[line width = 1.5pt] (0.center) to (1.center);
		\draw[line width = 1.5pt] [in=0, out=120, looseness=1.20] (1.center) to (2.center);
		\draw[line width = 1.5pt] [in=0, out=-120, looseness=1.20] (1.center) to (3.center);
		\draw[line width = 1.5pt] (4.center) to (5.center);
		\draw[line width = 1.5pt] [in=0, out=120, looseness=1.20] (5.center) to (6.center);
		\draw[line width = 1.5pt] [in=0, out=-120, looseness=1.20] (5.center) to (7.center);
\begin{scope}[shift={(2,-.5)}, rotate =180]
		\node [style=none] (0) at (1, -0) {};
		\node  [circle,draw,inner sep=1pt,fill]   (1) at (0.125, -0) {};
		\node [style=none] (2) at (-1, 0.5) {};
		\node [style=none] (3) at (-1, -0.5) {};
		\node [style=none] (4) at (1, -.5) {};
		\node  [circle,draw,inner sep=1pt,fill]   (5) at (0.125, -.5) {};
		\node [style=none] (6) at (-1, -0) {};
		\node [style=none] (7) at (-1, -1) {};
		\draw[line width = 1.5pt] (0.center) to (1.center);
		\draw[line width = 1.5pt] [in=0, out=120, looseness=1.20] (1.center) to (2.center);
		\draw[line width = 1.5pt] [in=0, out=-120, looseness=1.20] (1.center) to (3.center);
		\draw[line width = 1.5pt] (4.center) to (5.center);
		\draw[line width = 1.5pt] [in=0, out=120, looseness=1.20] (5.center) to (6.center);
		\draw[line width = 1.5pt] [in=0, out=-120, looseness=1.20] (5.center) to (7.center);
\end{scope}
}
      \end{tikzpicture}
\end{aligned}
}

\newcommand{\parfrobr}[1]
{
\begin{aligned}
\begin{tikzpicture}[circuit ee IEC, set resistor graphic=var resistor IEC
      graphic,scale=.5]
\scalebox{1}{
\begin{scope}[yscale=-1]
		\node [style=none] (0) at (1, -0) {};
		\node  [circle,draw,inner sep=1pt,fill]   (1) at (0.125, -0) {};
		\node [style=none] (2) at (-1, 0.5) {}; %
		\node [style=none] (3) at (-1, -0.5) {};
		\node [style=none] (4) at (1, -.5) {};
		\node  [circle,draw,inner sep=1pt,fill]   (5) at (0.125, -.5) {};
		\node [style=none] (6) at (-1, -0) {}; %
		\node [style=none] (7) at (-1, -1) {};
		\node [style=none] (8) at (-3, 0) {}; 
		\node [style=none] (9) at (-3, .5) {};
		\draw[line width = 1.5pt] (0.center) to (1.center);
		\draw[line width = 1.5pt] [in=0, out=120, looseness=1.20] (1.center) to (2.west);
		\draw[line width = 1.5pt] [in=0, out=-120, looseness=1.20] (1.center) to (3.west);
		\draw[line width = 1.5pt] (4.center) to (5.center);
		\draw[line width = 1.5pt] [in=0, out=120, looseness=1.20] (5.center) to (6.west);
		\draw[line width = 1.5pt] [in=0, out=-120, looseness=1.20] (5.center) to (7.west);
		\draw[line width = 1.5pt] [in=0, out=180] (2.center) to (9.center);
		\draw[line width = 1.5pt] [in=0, out=180] (6.center) to (8.center);
\begin{scope}[shift={(-2,-1.5)}, rotate=180]
		\node [style=none] (0) at (1, -0) {};
		\node  [circle,draw,inner sep=1pt,fill]   (1) at (0.125, -0) {};
		\node [style=none] (2) at (-1, 0.5) {}; %
		\node [style=none] (3) at (-1, -0.5) {};
		\node [style=none] (4) at (1, -.5) {};
		\node  [circle,draw,inner sep=1pt,fill]   (5) at (0.125, -.5) {};
		\node [style=none] (6) at (-1, -0) {}; %
		\node [style=none] (7) at (-1, -1) {};
		\node [style=none] (8) at (-3, 0) {}; 
		\node [style=none] (9) at (-3, .5) {};
		\draw[line width = 1.5pt] (0.center) to (1.center);
		\draw[line width = 1.5pt] [in=0, out=120, looseness=1.20] (1.center) to (2.east);
		\draw[line width = 1.5pt] [in=0, out=-120, looseness=1.20] (1.center) to (3.east);
		\draw[line width = 1.5pt] (4.center) to (5.center);
		\draw[line width = 1.5pt] [in=0, out=120, looseness=1.20] (5.center) to (6.east);
		\draw[line width = 1.5pt] [in=0, out=-120, looseness=1.20] (5.center) to (7.east);
		\draw[line width = 1.5pt] [in=0, out=180] (2.center) to (9.center);
		\draw[line width = 1.5pt] [in=0, out=180] (6.center) to (8.center);
\end{scope}
\end{scope}
}
      \end{tikzpicture}
\end{aligned}
}

\newcommand{\monadfrobl}[1]
{
\begin{aligned}
\begin{tikzpicture}[circuit ee IEC, set resistor graphic=var resistor IEC
      graphic,scale=.5]
\scalebox{1}{
		\node [style=none] (0) at (1, -0) {};
		\node [style=none] (1) at (0.125, -0) {};
		\node [style=none] (2) at (-1, 0.5) {}; 
		\node [style=none] (3) at (-1, -0.5) {}; 
		\node [style=none] (4) at (1, -0) {};
		\node [style=none] (5) at (-1, 1) {};
		\node [style=none] (6) at (1, .25) {};
		\node [style=none] (7) at (-1, -1) {}; 
		\node [style=none] (8) at (1, -.25) {};
		\node [style=none] (9) at (-1, -2) {}; 
		\node [style=none] (10) at (-3, -1.25) {}; 
		\node [style=none] (11) at (-1, -2.5) {}; 
		\node [style=none] (12) at (-3, -1.75) {}; 
		\node [style=none] (13) at (1, -2) {}; 
		\node [style=none] (14) at (1, -2.5) {}; 
		\node [style=none] (15) at (-3, 1) {}; 
		\node [style=none] (16) at (-3, .5) {}; 
	\draw[line width = 1.5pt] (2) to [in=0, out=0,looseness = 3] (3);
	\draw[line width = 1.5pt] (7.east) to [in=180, out=180,looseness = 3] (9.east);
	\draw[line width = 1.5pt] (5) to [in=180, out=0,looseness = 1] (6);
	\draw[line width = 1.5pt] (7) to [in=180, out=0, looseness = 1] (8);
	\draw[line width = 1.5pt] (3.east) to [in=0, out=180,looseness = 1] (10);
	\draw[line width = 1.5pt] (11.east) to [in=0, out=180,looseness = 1] (12);
	\draw[line width = 1.5pt] (9.east) to (13);
	\draw[line width = 1.5pt] (11.east) to (14);
	\draw[line width = 1.5pt] (5.east) to (15);
	\draw[line width = 1.5pt] (2.east) to (16);
}
      \end{tikzpicture}
\end{aligned}
}

\newcommand{\monadfrobm}[1]
{
\begin{aligned}
\begin{tikzpicture}[circuit ee IEC, set resistor graphic=var resistor IEC
      graphic,scale=.5]
\scalebox{1}{
		\node [style=none] (0) at (1, -0) {};
		\node [style=none] (1) at (0.125, -0) {};
		\node [style=none] (2) at (-1, 0.5) {}; 
		\node [style=none] (3) at (-1, -0.5) {}; 
		\node [style=none] (4) at (1, -0) {};
		\node [style=none] (5) at (-1, 1) {};
		\node [style=none] (6) at (1, .25) {};
		\node [style=none] (7) at (-1, -1) {}; 
		\node [style=none] (8) at (1, -1) {};
		\node [style=none] (9) at (-1, -2) {}; 
		\node [style=none] (10) at (-3, -.5) {}; 
		\node [style=none] (11) at (-1, -2.5) {}; 
		\node [style=none] (12) at (-3, -1.75) {}; 
		\node [style=none] (13) at (1, -2) {}; 
		\node [style=none] (14) at (1, -2.5) {}; 
		\node [style=none] (15) at (-3, 1) {}; 
		\node [style=none] (16) at (-3, .5) {}; 
	\draw[line width = 1.5pt] (2) to [in=0, out=0,looseness = 3] (3);
	\draw[line width = 1.5pt] (7.east) to [in=180, out=180,looseness = 3] (9.east);
	\draw[line width = 1.5pt] (5) to [in=180, out=0,looseness = 1] (6);
	\draw[line width = 1.5pt] (7) to [in=180, out=0, looseness = 1] (8);
	\draw[line width = 1.5pt] (3.east) to [in=0, out=180,looseness = 1] (10);
	\draw[line width = 1.5pt] (11.east) to [in=0, out=180,looseness = 1] (12);
	\draw[line width = 1.5pt] (9.east) to (13);
	\draw[line width = 1.5pt] (11.east) to (14);
	\draw[line width = 1.5pt] (5.east) to (15);
	\draw[line width = 1.5pt] (2.east) to (16);
}
      \end{tikzpicture}
\end{aligned}
}

\newcommand{\monadfrob}[1]
{
\begin{aligned}
\begin{tikzpicture}[circuit ee IEC, set resistor graphic=var resistor IEC
      graphic,scale=.5]
\scalebox{1}{
		\node [style=none] (0) at (1, -0) {};
		\node [style=none] (1) at (0.125, -0) {};
		\node [style=none] (2) at (-1, 0.5) {};
		\node [style=none] (3) at (-1, -0.5) {};
		\node [style=none] (4) at (1, -0) {};
		\node [style=none] (5) at (-1, 1) {};
		\node [style=none] (6) at (1, .25) {};
		\node [style=none] (7) at (-1, -1) {};
		\node [style=none] (8) at (1, -.25) {};
		\node [style=none] (9) at (3, 1) {};
		\node [style=none] (10) at (3, -1) {};
		\node [style=none] (11) at (3, .5) {};
		\node [style=none] (12) at (3, -.5) {};
	\draw[line width = 1.5pt] (2) to [in=0, out=0,looseness = 3] (3);
	\draw[line width = 1.5pt] (11) to [in=180, out=180,looseness = 3] (12);
	\draw[line width = 1.5pt] (5) to [in=180, out=0,looseness = 1] (6.east);
	\draw[line width = 1.5pt] (7) to [in=180, out=0, looseness = 1] (8.east);
	\draw[line width = 1.5pt] (6.east) to [in=180, out=0,looseness = 1] (9);
	\draw[line width = 1.5pt] (8.east) to [in=180, out=0, looseness = 1] (10);
}
      \end{tikzpicture}
\end{aligned}
}

\newcommand{\bi}[1]
{
  \begin{aligned}
\begin{tikzpicture}[circuit ee IEC, set resistor graphic=var resistor IEC
      graphic,scale=.5]
\scalebox{1}{
		\node [style=none] (0) at (-2, 0.75) {};
		\node [style=none] (1) at (-2, -0.5) {};
		\node  [circle,draw,inner sep=1pt,fill]   (2) at (-1, 0.75) {};
		\node  [circle,draw,inner sep=1pt,fill]   (3) at (-1, -0.5) {};
		\node [style=none] (4) at (0, 1.25) {};
		\node [style=none] (5) at (-0.25, 0.25) {};
		\node [style=none] (6) at (-0.25, -0) {};
		\node [style=none] (7) at (0, -1) {};
		\node [style=none] (8) at (0, 1.25) {};
		\node [style=none] (9) at (0.25, 0.25) {};
		\node [style=none] (10) at (0.25, -0) {};
		\node [style=none] (11) at (0, -1) {};
		\node  [circle,draw,inner sep=1pt,fill]   (12) at (1, 0.75) {};
		\node  [circle,draw,inner sep=1pt,fill]   (13) at (1, -0.5) {};
		\node [style=none] (14) at (2, 0.75) {};
		\node [style=none] (15) at (2, -0.5) {};
		\draw[line width = 1.5pt] (0.center) to (2);
		\draw[line width = 1.5pt] (1.center) to (3);
		\draw[line width = 1.5pt] (12) to (14.center);
		\draw[line width = 1.5pt] (13) to (15.center);
		\draw[line width = 1.5pt] [in=180, out=60, looseness=1.20] (2) to (4.center);
		\draw[line width = 1.5pt] [in=180, out=-60, looseness=1.20] (2) to (5.center);
		\draw[line width = 1.5pt] [in=180, out=60, looseness=1.20] (3) to (6.center);
		\draw[line width = 1.5pt] [in=180, out=-60, looseness=1.20] (3) to (7.center);
		\draw[line width = 1.5pt] [in=120, out=0, looseness=1.20] (8.center) to (12);
		\draw[line width = 1.5pt] [in=-120, out=0, looseness=1.20] (9.center) to (12);
		\draw[line width = 1.5pt] [in=120, out=0, looseness=1.20] (10.center) to (13);
		\draw[line width = 1.5pt] [in=-120, out=0, looseness=1.20] (11.center) to (13);
		\draw[line width = 1.5pt] (4.center) to (8.center);
		\draw[line width = 1.5pt] [in=150, out=-30, looseness=1.00] (5.center) to (10.center);
		\draw[line width = 1.5pt] [in=-150, out=30, looseness=1.00] (6.center) to (9.center);
		\draw[line width = 1.5pt] (7.center) to (11.center);
}
\end{tikzpicture}
  \end{aligned}
}

\newcommand{\bimultl}[1]
{
\begin{aligned}
\begin{tikzpicture}[circuit ee IEC, set resistor graphic=var resistor IEC
      graphic,scale=.5]
\scalebox{1}{
		\node  [circle,draw,inner sep=1pt,fill]   (0) at (1, -0) {};
		\node  [circle,draw,inner sep=1pt,fill]   (1) at (0.125, -0) {};
		\node [style=none] (2) at (-1, 0.5) {};
		\node [style=none] (3) at (-1, -0.5) {};
		\draw[line width = 1.5pt] (0.center) to (1.center);
		\draw[line width = 1.5pt] [in=0, out=120, looseness=1.20] (1.center) to (2.center);
		\draw[line width = 1.5pt] [in=0, out=-120, looseness=1.20] (1.center) to (3.center);
}
      \end{tikzpicture}
\end{aligned}
}

\newcommand{\bimultr}[1]
{
\begin{aligned}
\begin{tikzpicture}[circuit ee IEC, set resistor graphic=var resistor IEC
      graphic,scale=.5]
\scalebox{1}{
		\node [style=none] (e) at (1, -0) {};
		\node  [circle,draw,inner sep=1pt,fill]   (0) at (0.125, 0.5) {};
		\node  [circle,draw,inner sep=1pt,fill]   (1) at (0.125, -0.5) {};
		\node [style=none] (2) at (-1, 0.5) {};
		\node [style=none] (3) at (-1, -0.5) {};
		\draw[line width = 1.5pt] (2.center) to (0.center);
		\draw[line width = 1.5pt] (3.center) to (1.center);
}
      \end{tikzpicture}
\end{aligned}
}

\newcommand{\bicomultl}[1]
{
\begin{aligned}
\begin{tikzpicture}[circuit ee IEC, set resistor graphic=var resistor IEC
      graphic,scale=.5]
\scalebox{1}{
		\node  [circle,draw,inner sep=1pt,fill]   (0) at (-1, -0) {};
		\node  [circle,draw,inner sep=1pt,fill]   (1) at (-0.125, -0) {};
		\node [style=none] (2) at (1, 0.5) {};
		\node [style=none] (3) at (1, -0.5) {};
		\draw[line width = 1.5pt] (0.center) to (1.center);
		\draw[line width = 1.5pt] [in=180, out=60, looseness=1.20] (1.center) to (2.center);
		\draw[line width = 1.5pt] [in=180, out=-60, looseness=1.20] (1.center) to (3.center);
}
      \end{tikzpicture}
\end{aligned}
}

\newcommand{\bicomultr}[1]
{
\begin{aligned}
\begin{tikzpicture}[circuit ee IEC, set resistor graphic=var resistor IEC
      graphic,scale=.5]
\scalebox{1}{
		\node [style=none] (e) at (-1, -0) {};
		\node  [circle,draw,inner sep=1pt,fill]   (0) at (-0.125, 0.5) {};
		\node  [circle,draw,inner sep=1pt,fill]   (1) at (-0.125, -0.5) {};
		\node [style=none] (2) at (1, 0.5) {};
		\node [style=none] (3) at (1, -0.5) {};
		\draw[line width = 1.5pt] (0.center) to (2.center);
		\draw[line width = 1.5pt] (1.center) to (3.center);
}
      \end{tikzpicture}
\end{aligned}
}

\newcommand{\weakbil}[1]
{
\begin{aligned}
\begin{tikzpicture}[circuit ee IEC, set resistor graphic=var resistor IEC
      graphic,scale=.5]
\scalebox{1}{
		\node [style=none] (0) at (1, -0) {};
		\node  [circle,draw,inner sep=1pt,fill]   (1) at (0.125, -0) {};
		\node [style=none] (2) at (-1, 0.5) {};
		\node [style=none] (3) at (-1, -0.5) {};
		\node  [circle,draw,inner sep=1pt,fill]   (4) at (1, -0) {};
		\node [style=none] (5) at (-1, -0) {};
		\draw[line width = 1.5pt] (0.center) to (1.center);
		\draw[line width = 1.5pt] [in=0, out=120, looseness=1.20] (1.center) to (2.center);
		\draw[line width = 1.5pt] [in=0, out=-120, looseness=1.20] (1.center) to (3.center);
		\draw[line width = 1.5pt] [in=0, out=180] (1.center) to (5.center);
}
      \end{tikzpicture}
\end{aligned}
}

\newcommand{\weakbim}[1]
{
\begin{aligned}
\begin{tikzpicture}[circuit ee IEC, set resistor graphic=var resistor IEC
      graphic,scale=.5]
\scalebox{1}{
		\node [style=none] (0) at (1, -0) {};
		\node  [circle,draw,inner sep=1pt,fill]   (1) at (0.125, -0) {};
		\node [style=none] (2) at (-1, 0.5) {};
		\node [style=none] (3) at (-1, -0.5) {};
		\node  [circle,draw,inner sep=1pt,fill]   (4) at (1, -0) {};
		\node [style=none] (5) at (-3, -0.5) {};
		\draw[line width = 1.5pt] (0.center) to (1.center);
		\draw[line width = 1.5pt] (3.center) to (5.center);
		\draw[line width = 1.5pt] [in=0, out=120, looseness=1.20] (1.center) to (2.center);
		\draw[line width = 1.5pt] [in=0, out=-120, looseness=1.20] (1.center) to (3.center);
\begin{scope}[shift={(0,2)}]
		\node [style=none] (0) at (1, -0) {};
		\node  [circle,draw,inner sep=1pt,fill]   (1) at (0.125, -0) {};
		\node [style=none] (2) at (-1, 0.5) {};
		\node [style=none] (3) at (-1, -0.5) {};
		\node  [circle,draw,inner sep=1pt,fill]   (4) at (1, -0) {};
	     \node [style=none] (5) at (-3, 0.5) {};
		\draw[line width = 1.5pt] (0.center) to (1.center);
		\draw[line width = 1.5pt] (2.center) to (5.center);
		\draw[line width = 1.5pt] [in=0, out=120, looseness=1.20] (1.center) to (2.center);
		\draw[line width = 1.5pt] [in=0, out=-120, looseness=1.20] (1.center) to (3.center);
\end{scope}
\begin{scope}[shift={(-2,1)}, rotate =180 ]
		\node [style=none] (0) at (1, -0) {};
		\node  [circle,draw,inner sep=1pt,fill]   (1) at (0.125, -0) {};
		\node [style=none] (2) at (-1, 0.5) {};
		\node [style=none] (3) at (-1, -0.5) {};
		\draw[line width = 1.5pt] (0.center) to (1.center);
		\draw[line width = 1.5pt] [in=0, out=120, looseness=1.20] (1.center) to (2.center);
		\draw[line width = 1.5pt] [in=0, out=-120, looseness=1.20] (1.center) to (3.center);
\end{scope}
}
      \end{tikzpicture}
\end{aligned}
}

\newcommand{\weakbir}[1]
{
\begin{aligned}
\begin{tikzpicture}[circuit ee IEC, set resistor graphic=var resistor IEC
      graphic,scale=.5]
\scalebox{1}{
		\node [style=none] (0) at (1, -0) {};
		\node  [circle,draw,inner sep=1pt,fill]   (1) at (0.125, -0) {};
		\node [style=none] (2) at (-1, 0.5) {};
		\node [style=none] (3) at (-1, -0.5) {};
		\node  [circle,draw,inner sep=1pt,fill]   (4) at (1, -0) {};
		\node [style=none] (5) at (-4, -0.5) {};
		\draw[line width = 1.5pt] (0.center) to (1.center);
		\draw[line width = 1.5pt] (3.center) to (5.center);
		\draw[line width = 1.5pt] [in=0, out=120, looseness=1.20] (1.center) to (2.center);
		\draw[line width = 1.5pt] [in=0, out=-120, looseness=1.20] (1.center) to (3.center);
\begin{scope}[shift={(0,2)}]
		\node [style=none] (0) at (1, -0) {};
		\node  [circle,draw,inner sep=1pt,fill]   (1) at (0.125, -0) {};
		\node [style=none] (2) at (-1, 0.5) {};
		\node [style=none] (3) at (-1, -0.5) {};
		\node  [circle,draw,inner sep=1pt,fill]   (4) at (1, -0) {};
	     \node [style=none] (5) at (-4, 0.5) {};
		\draw[line width = 1.5pt] (0.center) to (1.center);
		\draw[line width = 1.5pt] (2.center) to (5.center);
		\draw[line width = 1.5pt] [in=0, out=120, looseness=1.20] (1.center) to (2.center);
		\draw[line width = 1.5pt] [in=0, out=-120, looseness=1.20] (1.center) to (3.center);
\end{scope}
\begin{scope}[shift={(-3,1)}, rotate =180 ]
		\node [style=none] (0) at (1.25, -0) {};
		\node  [circle,draw,inner sep=1pt,fill]   (1) at (0.375, -0) {};
		\node [style=none] (2) at (-0.5, -0.5) {};
		\node [style=none] (3) at (-2, 0.5) {};
		\node [style=none] (4) at (-0.5, 0.5) {};
		\node [style=none] (5) at (-2, -0.5) {};
		\draw[line width = 1.5pt] (0.center) to (1.center);
		\draw[line width = 1.5pt] [in=0, out=-120, looseness=1.20] (1.center) to (2.center);
		\draw[line width = 1.5pt] [in=180, out=0, looseness=1.00] (3.center) to (2.center);
		\draw[line width = 1.5pt] [in=0, out=120, looseness=1.20] (1.center) to (4.center);
		\draw[line width = 1.5pt] [in=0, out=180, looseness=1.00] (4.center) to (5.center);
\end{scope}
}
      \end{tikzpicture}
\end{aligned}
}

\newcommand{\weakcobil}[1]
{
\begin{aligned}
\begin{tikzpicture}[circuit ee IEC, set resistor graphic=var resistor IEC
      graphic,scale=.5]
\scalebox{1}{
\begin{scope}[rotate =180]
		\node [style=none] (0) at (1, -0) {};
		\node  [circle,draw,inner sep=1pt,fill]   (1) at (0.125, -0) {};
		\node [style=none] (2) at (-1, 0.5) {};
		\node [style=none] (3) at (-1, -0.5) {};
		\node  [circle,draw,inner sep=1pt,fill]   (4) at (1, -0) {};
		\node [style=none] (5) at (-1, -0) {};
		\draw[line width = 1.5pt] (0.center) to (1.center);
		\draw[line width = 1.5pt] [in=0, out=120, looseness=1.20] (1.center) to (2.center);
		\draw[line width = 1.5pt] [in=0, out=-120, looseness=1.20] (1.center) to (3.center);
		\draw[line width = 1.5pt] [in=0, out=180] (1.center) to (5.center);
\end{scope}
}
      \end{tikzpicture}
\end{aligned}
}

\newcommand{\weakcobim}[1]
{
\begin{aligned}
\begin{tikzpicture}[circuit ee IEC, set resistor graphic=var resistor IEC
      graphic,scale=.5]
\scalebox{1}{
\begin{scope}[rotate=180]
		\node [style=none] (0) at (1, -0) {};
		\node  [circle,draw,inner sep=1pt,fill]   (1) at (0.125, -0) {};
		\node [style=none] (2) at (-1, 0.5) {};
		\node [style=none] (3) at (-1, -0.5) {};
		\node  [circle,draw,inner sep=1pt,fill]   (4) at (1, -0) {};
		\node [style=none] (5) at (-3, -0.5) {};
		\draw[line width = 1.5pt] (0.center) to (1.center);
		\draw[line width = 1.5pt] (3.center) to (5.center);
		\draw[line width = 1.5pt] [in=0, out=120, looseness=1.20] (1.center) to (2.center);
		\draw[line width = 1.5pt] [in=0, out=-120, looseness=1.20] (1.center) to (3.center);
\begin{scope}[shift={(0,2)}]
		\node [style=none] (0) at (1, -0) {};
		\node  [circle,draw,inner sep=1pt,fill]   (1) at (0.125, -0) {};
		\node [style=none] (2) at (-1, 0.5) {};
		\node [style=none] (3) at (-1, -0.5) {};
		\node  [circle,draw,inner sep=1pt,fill]   (4) at (1, -0) {};
	     \node [style=none] (5) at (-3, 0.5) {};
		\draw[line width = 1.5pt] (0.center) to (1.center);
		\draw[line width = 1.5pt] (2.center) to (5.center);
		\draw[line width = 1.5pt] [in=0, out=120, looseness=1.20] (1.center) to (2.center);
		\draw[line width = 1.5pt] [in=0, out=-120, looseness=1.20] (1.center) to (3.center);
\end{scope}
\begin{scope}[shift={(-2,1)}, rotate =180 ]
		\node [style=none] (0) at (1, -0) {};
		\node  [circle,draw,inner sep=1pt,fill]   (1) at (0.125, -0) {};
		\node [style=none] (2) at (-1, 0.5) {};
		\node [style=none] (3) at (-1, -0.5) {};
		\draw[line width = 1.5pt] (0.center) to (1.center);
		\draw[line width = 1.5pt] [in=0, out=120, looseness=1.20] (1.center) to (2.center);
		\draw[line width = 1.5pt] [in=0, out=-120, looseness=1.20] (1.center) to (3.center);
\end{scope}
\end{scope}
}
      \end{tikzpicture}
\end{aligned}
}

\newcommand{\weakcobir}[1]
{
\begin{aligned}
\begin{tikzpicture}[circuit ee IEC, set resistor graphic=var resistor IEC
      graphic,scale=.5]
\scalebox{1}{
\begin{scope}[rotate=180]
		\node [style=none] (0) at (1, -0) {};
		\node  [circle,draw,inner sep=1pt,fill]   (1) at (0.125, -0) {};
		\node [style=none] (2) at (-1, 0.5) {};
		\node [style=none] (3) at (-1, -0.5) {};
		\node  [circle,draw,inner sep=1pt,fill]   (4) at (1, -0) {};
		\node [style=none] (5) at (-4, -0.5) {};
		\draw[line width = 1.5pt] (0.center) to (1.center);
		\draw[line width = 1.5pt] (3.center) to (5.center);
		\draw[line width = 1.5pt] [in=0, out=120, looseness=1.20] (1.center) to (2.center);
		\draw[line width = 1.5pt] [in=0, out=-120, looseness=1.20] (1.center) to (3.center);
\begin{scope}[shift={(0,2)}]
		\node [style=none] (0) at (1, -0) {};
		\node  [circle,draw,inner sep=1pt,fill]   (1) at (0.125, -0) {};
		\node [style=none] (2) at (-1, 0.5) {};
		\node [style=none] (3) at (-1, -0.5) {};
		\node  [circle,draw,inner sep=1pt,fill]   (4) at (1, -0) {};
	     \node [style=none] (5) at (-4, 0.5) {};
		\draw[line width = 1.5pt] (0.center) to (1.center);
		\draw[line width = 1.5pt] (2.center) to (5.center);
		\draw[line width = 1.5pt] [in=0, out=120, looseness=1.20] (1.center) to (2.center);
		\draw[line width = 1.5pt] [in=0, out=-120, looseness=1.20] (1.center) to (3.center);
\end{scope}
\begin{scope}[shift={(-3,1)}, rotate =180 ]
		\node [style=none] (0) at (1.25, -0) {};
		\node  [circle,draw,inner sep=1pt,fill]   (1) at (0.375, -0) {};
		\node [style=none] (2) at (-0.5, -0.5) {};
		\node [style=none] (3) at (-2, 0.5) {};
		\node [style=none] (4) at (-0.5, 0.5) {};
		\node [style=none] (5) at (-2, -0.5) {};
		\draw[line width = 1.5pt] (0.center) to (1.center);
		\draw[line width = 1.5pt] [in=0, out=-120, looseness=1.20] (1.center) to (2.center);
		\draw[line width = 1.5pt] [in=180, out=0, looseness=1.00] (3.center) to (2.center);
		\draw[line width = 1.5pt] [in=0, out=120, looseness=1.20] (1.center) to (4.center);
		\draw[line width = 1.5pt] [in=0, out=180, looseness=1.00] (4.center) to (5.center);
\end{scope}
\end{scope}
}
      \end{tikzpicture}
\end{aligned}
}

\newcommand{\weakbimonoidtwol}[1]
{
\begin{aligned}
\begin{tikzpicture}[circuit ee IEC, set resistor graphic=var resistor IEC
      graphic,scale=.5]
\scalebox{1}{
		\node [style=none] (0) at (1, -0) {};
		\node  [circle,draw,inner sep=1pt,fill]   (1) at (0.125, -0) {};
		\node [style=none] (2) at (-1, 0.5) {}; %
		\node [style=none] (3) at (-1, -0.5) {};
		\node [style=none] (4) at (1, -.5) {};
		\node  [circle,draw,inner sep=1pt,fill]   (5) at (0.125, -.5) {};
		\node [style=none] (6) at (-1, -0) {}; %
		\node [style=none] (7) at (-1, -1) {};
		\node [style=none] (8) at (-3, .5) {}; 
		\node [style=none] (9) at (-3, 1) {};
		\draw[line width = 1.5pt] (0.center) to (1.center);
		\draw[line width = 1.5pt] [in=0, out=120, looseness=1.20] (1.center) to (2.west);
		\draw[line width = 1.5pt] [in=0, out=-120, looseness=1.20] (1.center) to (3.west);
		\draw[line width = 1.5pt] (4.center) to (5.center);
		\draw[line width = 1.5pt] [in=0, out=120, looseness=1.20] (5.center) to (6.west);
		\draw[line width = 1.5pt] [in=0, out=-120, looseness=1.20] (5.center) to (7.west);
		\draw[line width = 1.5pt] [in=0, out=180, looseness=1.5] (2.center) to (9.center);
		\draw[line width = 1.5pt] [in=0, out=180, looseness=1.5] (6.center) to (8.center);
		\draw[line width = 1.5pt] (0.west) to [in=0, out=0,looseness = 3] (4.west);
\begin{scope}[shift={(-2,-.5)}]
		\node [style=none] (0) at (1, -0) {};
		\node  [circle,draw,inner sep=1pt,fill]   (1) at (0.125, -0) {};
		\node [style=none] (2) at (-1, 0.5) {};
		\node [style=none] (3) at (-1, -0.5) {};
		\node [style=none] (4) at (1, -.5) {};
		\node  [circle,draw,inner sep=1pt,fill]   (5) at (0.125, -.5) {};
		\node [style=none] (6) at (-1, -0) {};
		\node [style=none] (7) at (-1, -1) {};
		\draw[line width = 1.5pt] (0.center) to (1.center);
		\draw[line width = 1.5pt] [in=0, out=120, looseness=1.20] (1.center) to (2.center);
		\draw[line width = 1.5pt] [in=0, out=-120, looseness=1.20] (1.center) to (3.center);
		\draw[line width = 1.5pt] (4.center) to (5.center);
		\draw[line width = 1.5pt] [in=0, out=120, looseness=1.20] (5.center) to (6.center);
		\draw[line width = 1.5pt] [in=0, out=-120, looseness=1.20] (5.center) to (7.center);
\end{scope}
}
      \end{tikzpicture}
\end{aligned}
}

\newcommand{\weakbimonoidtwom}[1]
{
\begin{aligned}
\begin{tikzpicture}[circuit ee IEC, set resistor graphic=var resistor IEC
      graphic,scale=.5]
\scalebox{1}{
		\node [style=none] (0) at (1, -0) {};
		\node  [circle,draw,inner sep=1pt,fill]   (1) at (0.125, -0) {};
		\node [style=none] (2) at (-1, 0.5) {}; %
		\node [style=none] (3) at (-1, -0.5) {};
		\node [style=none] (4) at (1, -.5) {};
		\node  [circle,draw,inner sep=1pt,fill]   (5) at (0.125, -.5) {};
		\node [style=none] (6) at (-1, -0) {}; %
		\node [style=none] (7) at (-1, -1) {};
		\node [style=none] (8) at (-3, .5) {}; 
		\node [style=none] (9) at (-3, 1) {};
		\node [style=none] (10) at (-3, -1) {}; 
		\node [style=none] (11) at (-3, -1.5) {};
		\node [style=none] (12) at (-5, .5) {}; 
		\node [style=none] (13) at (-5, 1) {};
		\draw[line width = 1.5pt] (0.center) to (1.center);
		\draw[line width = 1.5pt] [in=0, out=120, looseness=1.20] (1.center) to (2.west);
		\draw[line width = 1.5pt] [in=0, out=-120, looseness=1.20] (1.center) to (3.west);
		\draw[line width = 1.5pt] (4.center) to (5.center);
		\draw[line width = 1.5pt] [in=0, out=120, looseness=1.20] (5.center) to (6.west);
		\draw[line width = 1.5pt] [in=0, out=-120, looseness=1.20] (5.center) to (7.west);
		\draw[line width = 1.5pt] [in=0, out=180, looseness=1.5] (2.center) to (9.center);
		\draw[line width = 1.5pt] [in=0, out=180, looseness=1.5] (6.center) to (8.center);
		\draw[line width = 1.5pt] [in=0, out=180, looseness=1.5] (7.center) to (11.west);
		\draw[line width = 1.5pt] [in=0, out=180, looseness=1.5] (3.center) to (10.west);
		\draw[line width = 1.5pt] [in=0, out=180, looseness=1.5] (8.center) to (12.west);
		\draw[line width = 1.5pt] [in=0, out=180, looseness=1.5] (9.center) to (13.west);
		\draw[line width = 1.5pt] (0.west) to [in=0, out=0,looseness = 3] (4.west);
\begin{scope}[shift={(0,-4)}]
\begin{scope}[yscale=-1]
		\node [style=none] (0) at (1, -0) {};
		\node  [circle,draw,inner sep=1pt,fill]   (1) at (0.125, -0) {};
		\node [style=none] (2) at (-1, 0.5) {}; %
		\node [style=none] (3) at (-1, -0.5) {};
		\node [style=none] (4) at (1, -.5) {};
		\node  [circle,draw,inner sep=1pt,fill]   (5) at (0.125, -.5) {};
		\node [style=none] (6) at (-1, -0) {}; %
		\node [style=none] (7) at (-1, -1) {};
		\node [style=none] (8) at (-3, .5) {}; 
		\node [style=none] (9) at (-3, 1) {};
		\node [style=none] (10) at (-3, -1) {}; 
		\node [style=none] (11) at (-3, -1.5) {};
		\node [style=none] (12) at (-5, .5) {}; 
		\node [style=none] (13) at (-5, 1) {};
		\draw[line width = 1.5pt] (0.center) to (1.center);
		\draw[line width = 1.5pt] [in=0, out=120, looseness=1.20] (1.center) to (2.west);
		\draw[line width = 1.5pt] [in=0, out=-120, looseness=1.20] (1.center) to (3.west);
		\draw[line width = 1.5pt] (4.center) to (5.center);
		\draw[line width = 1.5pt] [in=0, out=120, looseness=1.20] (5.center) to (6.west);
		\draw[line width = 1.5pt] [in=0, out=-120, looseness=1.20] (5.center) to (7.west);
		\draw[line width = 1.5pt] [in=0, out=180, looseness=1.5] (2.center) to (9.center);
		\draw[line width = 1.5pt] [in=0, out=180, looseness=1.5] (6.center) to (8.center);
		\draw[line width = 1.5pt] [in=0, out=180, looseness=1.5] (7.center) to (11.west);
		\draw[line width = 1.5pt] [in=0, out=180, looseness=1.5] (3.center) to (10.west);
		\draw[line width = 1.5pt] [in=0, out=180, looseness=1.5] (8.center) to (12.west);
		\draw[line width = 1.5pt] [in=0, out=180, looseness=1.5] (9.center) to (13.west);
		\draw[line width = 1.5pt] (0.west) to [in=0, out=0,looseness = 3] (4.west);
\end{scope}
\end{scope}
\begin{scope}[rotate=180,shift={(4,2)}]
		\node [style=none] (0) at (1, -0) {};
		\node [style=none] (1) at (0.125, -0) {};
		\node [style=none] (2) at (-1, 0.5) {};
		\node [style=none] (3) at (-1, -0.5) {};
		\node [style=none] (4) at (1, -0) {};
		\node [style=none] (5) at (-1, 1) {};
		\node [style=none] (6) at (1, .25) {};
		\node [style=none] (7) at (-1, -1) {};
		\node [style=none] (8) at (1, -.25) {};
	\draw[line width = 1.5pt] (2) to [in=0, out=0,looseness = 3] (3.west);
	\draw[line width = 1.5pt] (5) to [in=180, out=0,looseness = 1] (6);
	\draw[line width = 1.5pt] (7) to [in=180, out=0, looseness = 1] (8);
\end{scope}
}
      \end{tikzpicture}
\end{aligned}
}

\newcommand{\weakbimonoidtwor}[1]
{
\begin{aligned}
\begin{tikzpicture}[circuit ee IEC, set resistor graphic=var resistor IEC
      graphic,scale=.5]
\scalebox{1}{
		\node [style=none] (0) at (1, -0) {};
		\node  [circle,draw,inner sep=1pt,fill]   (1) at (0.125, -0) {};
		\node [style=none] (2) at (-1, 0.5) {}; %
		\node [style=none] (3) at (-1, -0.5) {};
		\node [style=none] (4) at (1, -.5) {};
		\node  [circle,draw,inner sep=1pt,fill]   (5) at (0.125, -.5) {};
		\node [style=none] (6) at (-1, -0) {}; %
		\node [style=none] (7) at (-1, -1) {};
		\node [style=none] (8) at (-3, .5) {}; 
		\node [style=none] (9) at (-3, 1) {};
		\node [style=none] (10) at (-3, -2.5) {}; 
		\node [style=none] (11) at (-3, -3) {};
		\node [style=none] (12) at (-5, .5) {}; 
		\node [style=none] (13) at (-5, 1) {};
		\draw[line width = 1.5pt] (0.center) to (1.center);
		\draw[line width = 1.5pt] [in=0, out=120, looseness=1.20] (1.center) to (2.west);
		\draw[line width = 1.5pt] [in=0, out=-120, looseness=1.20] (1.center) to (3.west);
		\draw[line width = 1.5pt] (4.center) to (5.center);
		\draw[line width = 1.5pt] [in=0, out=120, looseness=1.20] (5.center) to (6.west);
		\draw[line width = 1.5pt] [in=0, out=-120, looseness=1.20] (5.center) to (7.west);
		\draw[line width = 1.5pt] [in=0, out=180, looseness=1.5] (2.center) to (9.center);
		\draw[line width = 1.5pt] [in=0, out=180, looseness=1.5] (6.center) to (8.center);
		\draw[line width = 1.5pt] [in=0, out=180, looseness=1] (7.center) to (11.west);
		\draw[line width = 1.5pt] [in=0, out=180, looseness=1] (3.center) to (10.west);
		\draw[line width = 1.5pt] [in=0, out=180, looseness=1.5] (8.center) to (12.west);
		\draw[line width = 1.5pt] [in=0, out=180, looseness=1.5] (9.center) to (13.west);
		\draw[line width = 1.5pt] (0.west) to [in=0, out=0,looseness = 3] (4.west);
\begin{scope}[shift={(0,-4)}]
\begin{scope}[yscale=-1]
		\node [style=none] (0) at (1, -0) {};
		\node  [circle,draw,inner sep=1pt,fill]   (1) at (0.125, -0) {};
		\node [style=none] (2) at (-1, 0.5) {}; %
		\node [style=none] (3) at (-1, -0.5) {};
		\node [style=none] (4) at (1, -.5) {};
		\node  [circle,draw,inner sep=1pt,fill]   (5) at (0.125, -.5) {};
		\node [style=none] (6) at (-1, -0) {}; %
		\node [style=none] (7) at (-1, -1) {};
		\node [style=none] (8) at (-3, .5) {}; 
		\node [style=none] (9) at (-3, 1) {};
		\node [style=none] (10) at (-3, -2.5) {}; 
		\node [style=none] (11) at (-3, -3) {};
		\node [style=none] (12) at (-5, .5) {}; 
		\node [style=none] (13) at (-5, 1) {};
		\draw[line width = 1.5pt] (0.center) to (1.center);
		\draw[line width = 1.5pt] [in=0, out=120, looseness=1.20] (1.center) to (2.west);
		\draw[line width = 1.5pt] [in=0, out=-120, looseness=1.20] (1.center) to (3.west);
		\draw[line width = 1.5pt] (4.center) to (5.center);
		\draw[line width = 1.5pt] [in=0, out=120, looseness=1.20] (5.center) to (6.west);
		\draw[line width = 1.5pt] [in=0, out=-120, looseness=1.20] (5.center) to (7.west);
		\draw[line width = 1.5pt] [in=0, out=180, looseness=1.5] (2.center) to (9.center);
		\draw[line width = 1.5pt] [in=0, out=180, looseness=1.5] (6.center) to (8.center);
		\draw[line width = 1.5pt] [in=0, out=180, looseness=1] (7.center) to (11.west);
		\draw[line width = 1.5pt] [in=0, out=180, looseness=1] (3.center) to (10.west);
		\draw[line width = 1.5pt] [in=0, out=180, looseness=1.5] (8.center) to (12.west);
		\draw[line width = 1.5pt] [in=0, out=180, looseness=1.5] (9.center) to (13.west);
		\draw[line width = 1.5pt] (0.west) to [in=0, out=0,looseness = 3] (4.west);
\end{scope}
\end{scope}
\begin{scope}[rotate=180,shift={(4,2)}]
		\node [style=none] (0) at (1, -0) {};
		\node [style=none] (1) at (0.125, -0) {};
		\node [style=none] (2) at (-1, 0.5) {};
		\node [style=none] (3) at (-1, -0.5) {};
		\node [style=none] (4) at (1, -0) {};
		\node [style=none] (5) at (-1, 1) {};
		\node [style=none] (6) at (1, .25) {};
		\node [style=none] (7) at (-1, -1) {};
		\node [style=none] (8) at (1, -.25) {};
	\draw[line width = 1.5pt] (2) to [in=0, out=0,looseness = 3] (3.west);
	\draw[line width = 1.5pt] (5) to [in=180, out=0,looseness = 1] (6);
	\draw[line width = 1.5pt] (7) to [in=180, out=0, looseness = 1] (8);
\end{scope}
}
      \end{tikzpicture}
\end{aligned}
}

\newcommand{\weakbimonoidcotwol}[1]
{
\begin{aligned}
\begin{tikzpicture}[circuit ee IEC, set resistor graphic=var resistor IEC
      graphic,scale=.5]
\scalebox{1}{
\begin{scope}[rotate=180]
		\node [style=none] (0) at (1, -0) {};
		\node [style=none] (1) at (0.125, -0) {};
		\node [style=none] (2) at (-1, 0.5) {};
		\node [style=none] (3) at (-1, -0.5) {};
		\node [style=none] (4) at (1, -0) {};
		\node [style=none] (5) at (-1, 1) {};
		\node [style=none] (6) at (1, .25) {};
		\node [style=none] (7) at (-1, -1) {};
		\node [style=none] (8) at (1, -.25) {};
		\node [style=none] (9) at (1, 1.25) {};
		\node [style=none] (10) at (3, .5) {};
		\node [style=none] (11) at (1, 1.75) {};
		\node [style=none] (12) at (3, 1) {};
		\node [style=none] (13) at (-1, 2) {};
		\node [style=none] (14) at (-1, 2.5) {};
	\draw[line width = 1.5pt] (2) to [in=0, out=0,looseness = 3] (3.west);
	\draw[line width = 1.5pt] (5) to [in=0, out=0,looseness = 3] (13.west);
	\draw[line width = 1.5pt] (7) to [in=180, out=0, looseness = 1] (10);
	\draw[line width = 1.5pt] (14) to [in=180, out=0, looseness = 1] (12);
\end{scope}
\begin{scope}[shift={(-4,-.5)}]
		\node [style=none] (0) at (1, -0) {};
		\node [style=none] (1) at (-1, -0) {};
		\node  [circle,draw,inner sep=1pt,fill]   (2) at (0, -0) {};
		\node [style=none] (3) at (1, -.5) {};
		\node [style=none] (4) at (-1, -.5) {};
		\node  [circle,draw,inner sep=1pt,fill]   (5) at (0, -.5) {};
		\draw[line width = 1.5pt] (0.center) to (2);
		\draw[line width = 1.5pt] (3.center) to (5);
\end{scope}
}
      \end{tikzpicture}
\end{aligned}
}

\newcommand{\weakbimonoidcotwom}[1]
{
\begin{aligned}
\begin{tikzpicture}[circuit ee IEC, set resistor graphic=var resistor IEC
      graphic,scale=.5]
\scalebox{1}{
\begin{scope}[yscale=-1,rotate = 180]
		\node [style=none] (0) at (1, -0) {};
		\node [style=none] (1) at (0.125, -0) {};
		\node [style=none] (2) at (-1, 0.5) {};
		\node [style=none] (3) at (-1, -0.5) {};
		\node [style=none] (4) at (1, -0) {};
		\node [style=none] (5) at (-1, 1) {};
		\node [style=none] (6) at (1, .25) {};
		\node [style=none] (7) at (-1, -1) {};
		\node [style=none] (8) at (1, -.25) {};
		\node [style=none] (9) at (1, 1.25) {};
		\node [style=none] (10) at (3, .5) {};
		\node [style=none] (11) at (1, 1.75) {};
		\node [style=none] (12) at (3, 1) {};
		\node [style=none] (13) at (-2, 1.25) {};
		\node [style=none] (14) at (-2, 1.75) {};
		\node [style=none] (15) at (-0, .25) {};
		\node [style=none] (16) at (-0, -.25) {};
	\draw[line width = 1.5pt] (6) to [in=0, out=0,looseness = 3] (9.west);
	\draw[line width = 1.5pt] (8) to [in=180, out=0, looseness = 1] (10);
	\draw[line width = 1.5pt] (11) to [in=180, out=0, looseness = 1] (12);
	\draw[line width = 1.5pt] (13) to [in=180, out=0, looseness = 1] (9.west);
	\draw[line width = 1.5pt] (14) to [in=180, out=0, looseness = 1] (11.west);
	\draw[line width = 1.5pt] (6.west) to [in=0, out=180,looseness = 3] (15.east);
	\draw[line width = 1.5pt] (8.west) to [in=0, out=180,looseness = 3] (16.east);
\end{scope}
\begin{scope}[rotate = 180, shift={(0,1)}]
		\node [style=none] (0) at (1, -0) {};
		\node [style=none] (1) at (0.125, -0) {};
		\node [style=none] (2) at (-1, 0.5) {};
		\node [style=none] (3) at (-1, -0.5) {};
		\node [style=none] (4) at (1, -0) {};
		\node [style=none] (5) at (-1, 1) {};
		\node [style=none] (6) at (1, .25) {};
		\node [style=none] (7) at (-1, -1) {};
		\node [style=none] (8) at (1, -.25) {};
		\node [style=none] (9) at (1, 1.25) {};
		\node [style=none] (10) at (3, .5) {};
		\node [style=none] (11) at (1, 1.75) {};
		\node [style=none] (12) at (3, 1) {};
		\node [style=none] (13) at (-2, 1.25) {};
		\node [style=none] (14) at (-2, 1.75) {};
		\node [style=none] (15) at (-0, .25) {};
		\node [style=none] (16) at (-0, -.25) {};
	\draw[line width = 1.5pt] (6) to [in=0, out=0,looseness = 3] (9.west);
	\draw[line width = 1.5pt] (8) to [in=180, out=0, looseness = 1] (10);
	\draw[line width = 1.5pt] (11) to [in=180, out=0, looseness = 1] (12);
	\draw[line width = 1.5pt] (13) to [in=180, out=0, looseness = 1] (9.west);
	\draw[line width = 1.5pt] (14) to [in=180, out=0, looseness = 1] (11.west);
	\draw[line width = 1.5pt] (6.west) to [in=0, out=180,looseness = 3] (15.east);
	\draw[line width = 1.5pt] (8.west) to [in=0, out=180,looseness = 3] (16.east);
\end{scope}
\begin{scope}[shift={(1,-.25)}]
	\node [style=none] (0) at (1, -0) {};
		\node  [circle,draw,inner sep=1pt,fill]   (1) at (0.125, -0) {};
		\node [style=none] (2) at (-1, 0.5) {};
		\node [style=none] (3) at (-1, -0.5) {};
		\node [style=none] (4) at (1, -.5) {};
		\node  [circle,draw,inner sep=1pt,fill]   (5) at (0.125, -.5) {};
		\node [style=none] (6) at (-1, -0) {};
		\node [style=none] (7) at (-1, -1) {};
		\draw[line width = 1.5pt] (0.center) to (1.center);
		\draw[line width = 1.5pt] [in=0, out=120, looseness=1.20] (1.center) to (2.center);
		\draw[line width = 1.5pt] [in=0, out=-120, looseness=1.20] (1.center) to (3.center);
		\draw[line width = 1.5pt] (4.center) to (5.center);
		\draw[line width = 1.5pt] [in=0, out=120, looseness=1.20] (5.center) to (6.center);
		\draw[line width = 1.5pt] [in=0, out=-120, looseness=1.20] (5.center) to (7.center);
\end{scope}
\begin{scope}[shift={(-4,1)}]
		\node [style=none] (0) at (1, -0) {};
		\node  [circle,draw,inner sep=1pt,fill]   (2) at (0, -0) {};
		\node [style=none] (3) at (1, -.5) {};
		\node  [circle,draw,inner sep=1pt,fill]   (5) at (0, -.5) {};
		\draw[line width = 1.5pt] (0.center) to (2);
		\draw[line width = 1.5pt] (3.center) to (5);
\end{scope}
\begin{scope}[shift={(-4,-1.5)}]
		\node [style=none] (0) at (1, -0) {};
		\node  [circle,draw,inner sep=1pt,fill]   (2) at (0, -0) {};
		\node [style=none] (3) at (1, -.5) {};
		\node  [circle,draw,inner sep=1pt,fill]   (5) at (0, -.5) {};
		\draw[line width = 1.5pt] (0.center) to (2);
		\draw[line width = 1.5pt] (3.center) to (5);
\end{scope}
}
      \end{tikzpicture}
\end{aligned}
}

\newcommand{\weakbimonoidcotwor}[1]
{
\begin{aligned}
\begin{tikzpicture}[circuit ee IEC, set resistor graphic=var resistor IEC
      graphic,scale=.5]
\scalebox{1}{
\begin{scope}[yscale=-1,rotate = 180]
		\node [style=none] (0) at (1, -0) {};
		\node [style=none] (1) at (0.125, -0) {};
		\node [style=none] (2) at (-1, 0.5) {};
		\node [style=none] (3) at (-1, -0.5) {};
		\node [style=none] (4) at (1, -0) {};
		\node [style=none] (5) at (-1, 1) {};
		\node [style=none] (6) at (1, .25) {};
		\node [style=none] (7) at (-1, -1) {};
		\node [style=none] (8) at (1, -.25) {};
		\node [style=none] (9) at (1, 1.25) {};
		\node [style=none] (10) at (3, .5) {};
		\node [style=none] (11) at (1, 1.75) {};
		\node [style=none] (12) at (3, 1) {};
		\node [style=none] (13) at (-2, 1.25) {};
		\node [style=none] (14) at (-2, 1.75) {};
		\node [style=none] (15) at (-0, -.75) {};
		\node [style=none] (16) at (-0, -1.25) {};
	\draw[line width = 1.5pt] (6) to [in=0, out=0,looseness = 3] (9.west);
	\draw[line width = 1.5pt] (8) to [in=180, out=0, looseness = 1] (10);
	\draw[line width = 1.5pt] (11) to [in=180, out=0, looseness = 1] (12);
	\draw[line width = 1.5pt] (13) to [in=180, out=0, looseness = 1] (9.west);
	\draw[line width = 1.5pt] (14) to [in=180, out=0, looseness = 1] (11.west);
	\draw[line width = 1.5pt] (6.west) to [in=0, out=180,looseness = 1] (15.east);
	\draw[line width = 1.5pt] (8.west) to [in=0, out=180,looseness = 1] (16.east);
\end{scope}
\begin{scope}[rotate = 180, shift={(0,1)}]
		\node [style=none] (0) at (1, -0) {};
		\node [style=none] (1) at (0.125, -0) {};
		\node [style=none] (2) at (-1, 0.5) {};
		\node [style=none] (3) at (-1, -0.5) {};
		\node [style=none] (4) at (1, -0) {};
		\node [style=none] (5) at (-1, 1) {};
		\node [style=none] (6) at (1, .25) {};
		\node [style=none] (7) at (-1, -1) {};
		\node [style=none] (8) at (1, -.25) {};
		\node [style=none] (9) at (1, 1.25) {};
		\node [style=none] (10) at (3, .5) {};
		\node [style=none] (11) at (1, 1.75) {};
		\node [style=none] (12) at (3, 1) {};
		\node [style=none] (13) at (-2, 1.25) {};
		\node [style=none] (14) at (-2, 1.75) {};
		\node [style=none] (15) at (-0, -.75) {};
		\node [style=none] (16) at (-0, -1.25) {};
	\draw[line width = 1.5pt] (6) to [in=0, out=0,looseness = 3] (9.west);
	\draw[line width = 1.5pt] (8) to [in=180, out=0, looseness = 1] (10);
	\draw[line width = 1.5pt] (11) to [in=180, out=0, looseness = 1] (12);
	\draw[line width = 1.5pt] (13) to [in=180, out=0, looseness = 1] (9.west);
	\draw[line width = 1.5pt] (14) to [in=180, out=0, looseness = 1] (11.west);
	\draw[line width = 1.5pt] (6.west) to [in=0, out=180,looseness = 1] (15.east);
	\draw[line width = 1.5pt] (8.west) to [in=0, out=180,looseness = 1] (16.east);
\end{scope}
\begin{scope}[shift={(1,-.25)}]
	\node [style=none] (0) at (1, -0) {};
		\node  [circle,draw,inner sep=1pt,fill]   (1) at (0.125, -0) {};
		\node [style=none] (2) at (-1, 0.5) {};
		\node [style=none] (3) at (-1, -0.5) {};
		\node [style=none] (4) at (1, -.5) {};
		\node  [circle,draw,inner sep=1pt,fill]   (5) at (0.125, -.5) {};
		\node [style=none] (6) at (-1, -0) {};
		\node [style=none] (7) at (-1, -1) {};
		\draw[line width = 1.5pt] (0.center) to (1.center);
		\draw[line width = 1.5pt] [in=0, out=120, looseness=1.20] (1.center) to (2.center);
		\draw[line width = 1.5pt] [in=0, out=-120, looseness=1.20] (1.center) to (3.center);
		\draw[line width = 1.5pt] (4.center) to (5.center);
		\draw[line width = 1.5pt] [in=0, out=120, looseness=1.20] (5.center) to (6.center);
		\draw[line width = 1.5pt] [in=0, out=-120, looseness=1.20] (5.center) to (7.center);
\end{scope}
\begin{scope}[shift={(-4,1)}]
		\node [style=none] (0) at (1, -0) {};
		\node  [circle,draw,inner sep=1pt,fill]   (2) at (0, -0) {};
		\node [style=none] (3) at (1, -.5) {};
		\node  [circle,draw,inner sep=1pt,fill]   (5) at (0, -.5) {};
		\draw[line width = 1.5pt] (0.center) to (2);
		\draw[line width = 1.5pt] (3.center) to (5);
\end{scope}
\begin{scope}[shift={(-4,-1.5)}]
		\node [style=none] (0) at (1, -0) {};
		\node  [circle,draw,inner sep=1pt,fill]   (2) at (0, -0) {};
		\node [style=none] (3) at (1, -.5) {};
		\node  [circle,draw,inner sep=1pt,fill]   (5) at (0, -.5) {};
		\draw[line width = 1.5pt] (0.center) to (2);
		\draw[line width = 1.5pt] (3.center) to (5);
\end{scope}
}
      \end{tikzpicture}
\end{aligned}
}

\newcommand{\bigbimonoidone}[1]
{
\begin{aligned}
\begin{tikzpicture}[circuit ee IEC, set resistor graphic=var resistor IEC
      graphic,scale=.4]
\scalebox{1}{
		\node [style=none] (0) at (1, -0) {};
		\node  [circle,draw,inner sep=1pt,fill]   (1) at (0.125, -0) {};
		\node [style=none] (2) at (-1, 0.5) {};
		\node [style=none] (3) at (-1, -0.5) {};
		\node [style=none] (4) at (1, -.5) {};
		\node  [circle,draw,inner sep=1pt,fill]   (5) at (0.125, -.5) {};
		\node [style=none] (6) at (-1, -0) {};
		\node [style=none] (7) at (-1, -1) {};
		\node [style=none] (8) at (-7, -1) {};
		\node [style=none] (9) at (-7, -.5) {};
		\node [style=none] (10) at (-7, 3) {};
		\node [style=none] (11) at (-7, 3.5) {};
		\draw[line width = 1.5pt] (0.center) to (1.center);
		\draw[line width = 1.5pt] [in=0, out=120, looseness=1.20] (1.center) to (2.center);
		\draw[line width = 1.5pt] [in=0, out=-120, looseness=1.20] (1.center) to (3.center);
		\draw[line width = 1.5pt] (4.center) to (5.center);
		\draw[line width = 1.5pt] [in=0, out=120, looseness=1.20] (5.center) to (6.center);
		\draw[line width = 1.5pt] [in=0, out=-120, looseness=1.20] (5.center) to (7.center);
		\draw[line width = 1.5pt] (7.center) to (8.west);
		\draw[line width = 1.5pt] (3.center) to (9.west);
		\draw[line width = 1.5pt] [in=0, out=180, looseness=.75] (6.center) to (10.west);
		\draw[line width = 1.5pt] [in=0, out=180, looseness=.75] (2.center) to (11.west);
\begin{scope}[shift={(0,4.5)}]
		\node [style=none] (0) at (1, -0) {};
		\node  [circle,draw,inner sep=1pt,fill]   (1) at (0.125, -0) {};
		\node [style=none] (2) at (-1, 0.5) {};
		\node [style=none] (3) at (-1, -0.5) {};
		\node [style=none] (4) at (1, -.5) {};
		\node  [circle,draw,inner sep=1pt,fill]   (5) at (0.125, -.5) {};
		\node [style=none] (6) at (-1, -0) {};
		\node [style=none] (7) at (-1, -1) {};
		\node [style=none] (8) at (-7, 0) {};
		\node [style=none] (9) at (-7, .5) {};
		\node [style=none] (10) at (-7, -4) {};
		\node [style=none] (11) at (-7, -3.5) {};
		\draw[line width = 1.5pt] (0.center) to (1.center);
		\draw[line width = 1.5pt] [in=0, out=120, looseness=1.20] (1.center) to (2.center);
		\draw[line width = 1.5pt] [in=0, out=-120, looseness=1.20] (1.center) to (3.center);
		\draw[line width = 1.5pt] (4.center) to (5.center);
		\draw[line width = 1.5pt] [in=0, out=120, looseness=1.20] (5.center) to (6.center);
		\draw[line width = 1.5pt] [in=0, out=-120, looseness=1.20] (5.center) to (7.center);
		\draw[line width = 1.5pt] (6.center) to (8.west);
		\draw[line width = 1.5pt] (2.center) to (9.west);
		\draw[line width = 1.5pt] [in=0, out=180, looseness=.75] (7.center) to (10.west);
		\draw[line width = 1.5pt] [in=0, out=180, looseness=.75] (3.center) to (11.west);
\end{scope}
\begin{scope}[rotate=180, shift={(8,0)}]
		\node [style=none] (0) at (1, -0) {};
		\node [style=none] (1) at (0.125, -0) {};
		\node [style=none] (2) at (-1, 0.5) {};
		\node [style=none] (3) at (-1, -0.5) {};
		\node [style=none] (4) at (1, -0) {};
		\node [style=none] (5) at (-1, 1) {};
		\node [style=none] (6) at (1, .25) {};
		\node [style=none] (7) at (-1, -1) {};
		\node [style=none] (8) at (1, -.25) {};
	\draw[line width = 1.5pt] (2) to [in=0, out=0,looseness = 3] (3);
	\draw[line width = 1.5pt] (5) to [in=180, out=0,looseness = 1] (6);
	\draw[line width = 1.5pt] (7) to [in=180, out=0, looseness = 1] (8);
\end{scope}
\begin{scope}[rotate=180, shift={(8,-4)}]
		\node [style=none] (0) at (1, -0) {};
		\node [style=none] (1) at (0.125, -0) {};
		\node [style=none] (2) at (-1, 0.5) {};
		\node [style=none] (3) at (-1, -0.5) {};
		\node [style=none] (4) at (1, -0) {};
		\node [style=none] (5) at (-1, 1) {};
		\node [style=none] (6) at (1, .25) {};
		\node [style=none] (7) at (-1, -1) {};
		\node [style=none] (8) at (1, -.25) {};
	\draw[line width = 1.5pt] (2) to [in=0, out=0,looseness = 3] (3);
	\draw[line width = 1.5pt] (5) to [in=180, out=0,looseness = 1] (6);
	\draw[line width = 1.5pt] (7) to [in=180, out=0, looseness = 1] (8);
\end{scope}
}
      \end{tikzpicture}
\end{aligned}
}

\newcommand{\bigbimonoidtwo}[1]
{
\begin{aligned}
\begin{tikzpicture}[circuit ee IEC, set resistor graphic=var resistor IEC
      graphic,scale=.4]
\scalebox{1}{
		\node [style=none] (0) at (1, -0) {};
		\node  [circle,draw,inner sep=1pt,fill]   (1) at (0.125, -0) {};
		\node [style=none] (2) at (-1, 0.5) {};
		\node [style=none] (3) at (-1, -0.5) {};
		\node [style=none] (4) at (1, -.5) {};
		\node  [circle,draw,inner sep=1pt,fill]   (5) at (0.125, -.5) {};
		\node [style=none] (6) at (-1, -0) {};
		\node [style=none] (7) at (-1, -1) {};
		\node [style=none] (8) at (-7, -1) {};
		\node [style=none] (9) at (-3.5, -.5) {};
		\node [style=none] (10) at (-7, 3) {};
		\node [style=none] (11) at (-3.5, 3.5) {};
		\draw[line width = 1.5pt] (0.center) to (1.center);
		\draw[line width = 1.5pt] [in=0, out=120, looseness=1.20] (1.center) to (2.center);
		\draw[line width = 1.5pt] [in=0, out=-120, looseness=1.20] (1.center) to (3.center);
		\draw[line width = 1.5pt] (4.center) to (5.center);
		\draw[line width = 1.5pt] [in=0, out=120, looseness=1.20] (5.center) to (6.center);
		\draw[line width = 1.5pt] [in=0, out=-120, looseness=1.20] (5.center) to (7.center);
		\draw[line width = 1.5pt] (7.center) to (8.west);
		\draw[line width = 1.5pt] (3.center) to (9.west);
		\draw[line width = 1.5pt] [in=0, out=180, looseness=.75] (6.center) to (10.west);
		\draw[line width = 1.5pt] [in=0, out=180, looseness=.75] (2.center) to (11.west);
\begin{scope}[shift={(0,4.5)}]
		\node [style=none] (0) at (1, -0) {};
		\node  [circle,draw,inner sep=1pt,fill]   (1) at (0.125, -0) {};
		\node [style=none] (2) at (-1, 0.5) {};
		\node [style=none] (3) at (-1, -0.5) {};
		\node [style=none] (4) at (1, -.5) {};
		\node  [circle,draw,inner sep=1pt,fill]   (5) at (0.125, -.5) {};
		\node [style=none] (6) at (-1, -0) {};
		\node [style=none] (7) at (-1, -1) {};
		\node [style=none] (8) at (-3.5, 0) {};
		\node [style=none] (9) at (-7, .5) {};
		\node [style=none] (10) at (-3.5, -4) {};
		\node [style=none] (11) at (-7, -3.5) {};
		\draw[line width = 1.5pt] (0.center) to (1.center);
		\draw[line width = 1.5pt] [in=0, out=120, looseness=1.20] (1.center) to (2.center);
		\draw[line width = 1.5pt] [in=0, out=-120, looseness=1.20] (1.center) to (3.center);
		\draw[line width = 1.5pt] (4.center) to (5.center);
		\draw[line width = 1.5pt] [in=0, out=120, looseness=1.20] (5.center) to (6.center);
		\draw[line width = 1.5pt] [in=0, out=-120, looseness=1.20] (5.center) to (7.center);
		\draw[line width = 1.5pt] (6.center) to (8.west);
		\draw[line width = 1.5pt] (2.center) to (9.west);
		\draw[line width = 1.5pt] [in=0, out=180, looseness=.75] (7.center) to (10.center);
		\draw[line width = 1.5pt] [in=0, out=180, looseness=.75] (3.center) to (11.west);
\end{scope}
\begin{scope}[rotate=180, shift={(8,0)}]
		\node [style=none] (0) at (1, -0) {};
		\node [style=none] (1) at (0.125, -0) {};
		\node [style=none] (2) at (-4.5, 0.5) {};
		\node [style=none] (3) at (-4.5, -0.5) {};
		\node [style=none] (4) at (1, -0) {};
		\node [style=none] (5) at (-1, 1) {};
		\node [style=none] (6) at (1, .25) {};
		\node [style=none] (7) at (-1, -1) {};
		\node [style=none] (8) at (1, -.25) {};
	\draw[line width = 1.5pt] (2) to [in=0, out=0,looseness = 3] (3);
	\draw[line width = 1.5pt] (5) to [in=180, out=0,looseness = 1] (6);
	\draw[line width = 1.5pt] (7) to [in=180, out=0, looseness = 1] (8);
\end{scope}
\begin{scope}[rotate=180, shift={(8,-4)}]
		\node [style=none] (0) at (1, -0) {};
		\node [style=none] (1) at (0.125, -0) {};
		\node [style=none] (2) at (-4.5, 0.5) {};
		\node [style=none] (3) at (-4.5, -0.5) {};
		\node [style=none] (4) at (1, -0) {};
		\node [style=none] (5) at (-1, 1) {};
		\node [style=none] (6) at (1, .25) {};
		\node [style=none] (7) at (-1, -1) {};
		\node [style=none] (8) at (1, -.25) {};
	\draw[line width = 1.5pt] (2) to [in=0, out=0,looseness = 3] (3);
	\draw[line width = 1.5pt] (5) to [in=180, out=0,looseness = 1] (6);
	\draw[line width = 1.5pt] (7) to [in=180, out=0, looseness = 1] (8);
\end{scope}
}
      \end{tikzpicture}
\end{aligned}
}

\newcommand{\bigbimonoidthree}[1]
{
\begin{aligned}
\begin{tikzpicture}[circuit ee IEC, set resistor graphic=var resistor IEC
      graphic,scale=.4]
\scalebox{1}{
		\node [style=none] (0) at (1, 1) {};
		\node  [circle,draw,inner sep=1pt,fill]   (1) at (0.125, 1) {};
		\node [style=none] (2) at (-1, 1.5) {};
		\node [style=none] (3) at (-1, .5) {};
		\node [style=none] (4) at (1, -.5) {};
		\node  [circle,draw,inner sep=1pt,fill]   (5) at (0.125, -.5) {};
		\node [style=none] (6) at (-1, -0) {};
		\node [style=none] (7) at (-1, -1) {};
		\node [style=none] (8) at (-7, -1) {};
		\node [style=none] (9) at (-1.5, .5) {};
		\node [style=none] (10) at (-7, 3) {};
		\node [style=none] (11) at (-1.5, 2.5) {};
		\draw[line width = 1.5pt] (0.center) to (1.center);
		\draw[line width = 1.5pt] [in=0, out=120, looseness=1.20] (1.center) to (2.center);
		\draw[line width = 1.5pt] [in=0, out=-120, looseness=1.20] (1.center) to (3.center);
		\draw[line width = 1.5pt] (4.center) to (5.center);
		\draw[line width = 1.5pt] [in=0, out=120, looseness=1.20] (5.center) to (6.center);
		\draw[line width = 1.5pt] [in=0, out=-120, looseness=1.20] (5.center) to (7.center);
		\draw[line width = 1.5pt] (7.center) to (8.west);
		\draw[line width = 1.5pt] (3.center) to (9.west);
		\draw[line width = 1.5pt] [in=0, out=180, looseness=1] (6.center) to (10.west);
		\draw[line width = 1.5pt] [in=0, out=180, looseness=.8] (2.center) to (11.center);
\begin{scope}[shift={(0,4.5)}]
		\node [style=none] (0) at (1, -0) {};
		\node  [circle,draw,inner sep=1pt,fill]   (1) at (0.125, -0) {};
		\node [style=none] (2) at (-1, 0.5) {};
		\node [style=none] (3) at (-1, -0.5) {};
		\node [style=none] (4) at (1, -1.5) {};
		\node  [circle,draw,inner sep=1pt,fill]   (5) at (0.125, -1.5) {};
		\node [style=none] (6) at (-1, -1) {};
		\node [style=none] (7) at (-1, -2) {};
		\node [style=none] (8) at (-1.5, -1) {};
		\node [style=none] (9) at (-7, .5) {};
		\node [style=none] (10) at (-1.5, -3) {};
		\node [style=none] (11) at (-7, -3.5) {};
		\draw[line width = 1.5pt] (0.center) to (1.center);
		\draw[line width = 1.5pt] [in=0, out=120, looseness=1.20] (1.center) to (2.center);
		\draw[line width = 1.5pt] [in=0, out=-120, looseness=1.20] (1.center) to (3.center);
		\draw[line width = 1.5pt] (4.center) to (5.center);
		\draw[line width = 1.5pt] [in=0, out=120, looseness=1.20] (5.center) to (6.center);
		\draw[line width = 1.5pt] [in=0, out=-120, looseness=1.20] (5.center) to (7.center);
		\draw[line width = 1.5pt] (6.center) to (8.west);
		\draw[line width = 1.5pt] (2.center) to (9.west);
		\draw[line width = 1.5pt] [in=180, out=0, looseness=.8] (10.center) to (7.center);
		\draw[line width = 1.5pt] [in=0, out=180, looseness=1] (3.center) to (11.west);
\end{scope}
\begin{scope}[rotate=180, shift={(8,0)}]
		\node [style=none] (0) at (1, -0) {};
		\node [style=none] (1) at (0.125, -0) {};
		\node [style=none] (2) at (-6.5, -.5) {};
		\node [style=none] (3) at (-6.5, -1.5) {};
		\node [style=none] (4) at (1, -0) {};
		\node [style=none] (5) at (-1, 1) {};
		\node [style=none] (6) at (1, .25) {};
		\node [style=none] (7) at (-1, -1) {};
		\node [style=none] (8) at (1, -.25) {};
	\draw[line width = 1.5pt] (2.east) to [in=0, out=0,looseness = 3] (3.east);
	\draw[line width = 1.5pt] (5) to [in=180, out=0,looseness = 1] (6);
	\draw[line width = 1.5pt] (7) to [in=180, out=0, looseness = 1] (8);
\end{scope}
\begin{scope}[rotate=180, shift={(8,-4)}]
		\node [style=none] (0) at (1, -0) {};
		\node [style=none] (1) at (0.125, -0) {};
		\node [style=none] (2) at (-6.5, 1.5) {};
		\node [style=none] (3) at (-6.5, .5) {};
		\node [style=none] (4) at (1, -0) {};
		\node [style=none] (5) at (-1, 1) {};
		\node [style=none] (6) at (1, .25) {};
		\node [style=none] (7) at (-1, -1) {};
		\node [style=none] (8) at (1, -.25) {};
	\draw[line width = 1.5pt] (2.east) to [in=0, out=0,looseness = 3] (3.east);
	\draw[line width = 1.5pt] (5) to [in=180, out=0,looseness = 1] (6);
	\draw[line width = 1.5pt] (7) to [in=180, out=0, looseness = 1] (8);
\end{scope}
}
      \end{tikzpicture}
\end{aligned}
}

\newcommand{\bigbimonoidfour}[1]
{
\begin{aligned}
\begin{tikzpicture}[circuit ee IEC, set resistor graphic=var resistor IEC
      graphic,scale=.4]
\scalebox{1}{
		\node [style=none] (0) at (1, 1) {};
		\node  [circle,draw,inner sep=1pt,fill]   (1) at (0.125, 1) {};
		\node [style=none] (2) at (-1, 1.5) {};
		\node [style=none] (3) at (-1, .5) {};
		\node [style=none] (4) at (1, -.5) {};
		\node  [circle,draw,inner sep=1pt,fill]   (5) at (0.125, -.5) {};
		\node [style=none] (6) at (-1, -0) {};
		\node [style=none] (7) at (-1, -1) {};
		\node [style=none] (8) at (-7, -1) {};
		\node [style=none] (9) at (-1.5, .5) {};
		\node [style=none] (10) at (-7, 3) {};
		\node [style=none] (11) at (-1.5, 2.5) {};
		\node [style=none] (12) at (-1, 3.5) {};
		\node [style=none] (13) at (-1, 2.5) {};
		\draw[line width = 1.5pt] (0.center) to (1.center);
		\draw[line width = 1.5pt] [in=0, out=120, looseness=1.20] (1.center) to (2.center);
		\draw[line width = 1.5pt] [in=0, out=-120, looseness=1.20] (1.center) to (3.east);
		\draw[line width = 1.5pt] (4.center) to (5.center);
		\draw[line width = 1.5pt] [in=0, out=120, looseness=1.20] (5.center) to (6.center);
		\draw[line width = 1.5pt] [in=0, out=-120, looseness=1.20] (5.center) to (7.center);
		\draw[line width = 1.5pt] (7.center) to (8.west);
		\draw[line width = 1.5pt] [in=0, out=180, looseness=.75] (6.center) to (10.west);
		\draw[line width = 1.5pt] (12.east) to [in=180, out=180,looseness = 2.5] (3.east);
		\draw[line width = 1.5pt] (13.east) to [in=180, out=180,looseness = 3] (2.east);
\begin{scope}[shift={(0,4.5)}]
		\node [style=none] (0) at (1, -0) {};
		\node  [circle,draw,inner sep=1pt,fill]   (1) at (0.125, -0) {};
		\node [style=none] (2) at (-1, 0.5) {};
		\node [style=none] (3) at (-1, -0.5) {};
		\node [style=none] (4) at (1, -1.5) {};
		\node  [circle,draw,inner sep=1pt,fill]   (5) at (0.125, -1.5) {};
		\node [style=none] (6) at (-1, -1) {};
		\node [style=none] (7) at (-1, -2) {};
		\node [style=none] (8) at (-1.5, -1) {};
		\node [style=none] (9) at (-7, .5) {};
		\node [style=none] (10) at (-1.5, -3) {};
		\node [style=none] (11) at (-7, -3.5) {};
		\draw[line width = 1.5pt] (0.center) to (1.center);
		\draw[line width = 1.5pt] [in=0, out=120, looseness=1.20] (1.center) to (2.center);
		\draw[line width = 1.5pt] [in=0, out=-120, looseness=1.20] (1.center) to (3.center);
		\draw[line width = 1.5pt] (4.center) to (5.center);
		\draw[line width = 1.5pt] [in=0, out=120, looseness=1.20] (5.center) to (6.center);
		\draw[line width = 1.5pt] [in=0, out=-120, looseness=1.20] (5.center) to (7.center);
		\draw[line width = 1.5pt] (2.center) to (9.west);
		\draw[line width = 1.5pt] [in=0, out=180, looseness=.75] (3.center) to (11.west);
\end{scope}
\begin{scope}[rotate=180, shift={(8,0)}]
		\node [style=none] (0) at (1, -0) {};
		\node [style=none] (1) at (0.125, -0) {};
		\node [style=none] (2) at (-6.5, -.5) {};
		\node [style=none] (3) at (-6.5, -1.5) {};
		\node [style=none] (4) at (1, -0) {};
		\node [style=none] (5) at (-1, 1) {};
		\node [style=none] (6) at (1, .25) {};
		\node [style=none] (7) at (-1, -1) {};
		\node [style=none] (8) at (1, -.25) {};
	\draw[line width = 1.5pt] (5) to [in=180, out=0,looseness = 1] (6);
	\draw[line width = 1.5pt] (7) to [in=180, out=0, looseness = 1] (8);
\end{scope}
\begin{scope}[rotate=180, shift={(8,-4)}]
		\node [style=none] (0) at (1, -0) {};
		\node [style=none] (1) at (0.125, -0) {};
		\node [style=none] (2) at (-6.5, 1.5) {};
		\node [style=none] (3) at (-6.5, .5) {};
		\node [style=none] (4) at (1, -0) {};
		\node [style=none] (5) at (-1, 1) {};
		\node [style=none] (6) at (1, .25) {};
		\node [style=none] (7) at (-1, -1) {};
		\node [style=none] (8) at (1, -.25) {};
	\draw[line width = 1.5pt] (5) to [in=180, out=0,looseness = 1] (6);
	\draw[line width = 1.5pt] (7) to [in=180, out=0, looseness = 1] (8);
\end{scope}
}
      \end{tikzpicture}
\end{aligned}
}

\newcommand{\bigbimonoidfive}[1]
{
\begin{aligned}
\begin{tikzpicture}[circuit ee IEC, set resistor graphic=var resistor IEC
      graphic,scale=.4]
\scalebox{1}{
		\node [style=none] (0) at (1, 1) {};
		\node [style=none] (1) at (1, 3) {};
		\node [style=none] (2) at (-1, 1.5) {};
		\node [style=none] (3) at (-1, .5) {};
		\node [style=none] (4) at (1, -.5) {};
		\node  [circle,draw,inner sep=1pt,fill]   (5) at (0.125, -.5) {};
		\node [style=none] (6) at (-1, -0) {};
		\node [style=none] (7) at (-1, -1) {};
		\node [style=none] (8) at (-7, -1) {};
		\node [style=none] (9) at (-1.5, .5) {};
		\node [style=none] (10) at (-7, 3) {};
		\node [style=none] (11) at (-1.5, 2.5) {};
		\node [style=none] (12) at (-1, 3.5) {};
		\node [style=none] (13) at (-1, 2.5) {};
		\draw[line width = 1.5pt] (4.center) to (5.center);
		\draw[line width = 1.5pt] [in=0, out=120, looseness=1.20] (5.center) to (6.center);
		\draw[line width = 1.5pt] [in=0, out=-120, looseness=1.20] (5.center) to (7.center);
		\draw[line width = 1.5pt] (7.center) to (8.west);
		\draw[line width = 1.5pt] [in=0, out=180, looseness=.75] (6.center) to (10.west);
		\draw[line width = 1.5pt] (0) to [in=180, out=180,looseness = 4] (1);
\begin{scope}[shift={(0,4.5)}]
		\node [style=none] (0) at (1, -0) {};
		\node  [circle,draw,inner sep=1pt,fill]   (1) at (0.125, -0) {};
		\node [style=none] (2) at (-1, 0.5) {};
		\node [style=none] (3) at (-1, -0.5) {};
		\node [style=none] (8) at (-1.5, -1) {};
		\node [style=none] (9) at (-7, .5) {};
		\node [style=none] (10) at (-1.5, -3) {};
		\node [style=none] (11) at (-7, -3.5) {};
		\draw[line width = 1.5pt] (0.center) to (1.center);
		\draw[line width = 1.5pt] [in=0, out=120, looseness=1.20] (1.center) to (2.center);
		\draw[line width = 1.5pt] [in=0, out=-120, looseness=1.20] (1.center) to (3.center);
		\draw[line width = 1.5pt] (2.center) to (9.west);
		\draw[line width = 1.5pt] [in=0, out=180, looseness=.75] (3.center) to (11.west);
\end{scope}
\begin{scope}[rotate=180, shift={(8,0)}]
		\node [style=none] (0) at (1, -0) {};
		\node [style=none] (1) at (0.125, -0) {};
		\node [style=none] (2) at (-6.5, -.5) {};
		\node [style=none] (3) at (-6.5, -1.5) {};
		\node [style=none] (4) at (1, -0) {};
		\node [style=none] (5) at (-1, 1) {};
		\node [style=none] (6) at (1, .25) {};
		\node [style=none] (7) at (-1, -1) {};
		\node [style=none] (8) at (1, -.25) {};
	\draw[line width = 1.5pt] (5) to [in=180, out=0,looseness = 1] (6);
	\draw[line width = 1.5pt] (7) to [in=180, out=0, looseness = 1] (8);
\end{scope}
\begin{scope}[rotate=180, shift={(8,-4)}]
		\node [style=none] (0) at (1, -0) {};
		\node [style=none] (1) at (0.125, -0) {};
		\node [style=none] (2) at (-6.5, 1.5) {};
		\node [style=none] (3) at (-6.5, .5) {};
		\node [style=none] (4) at (1, -0) {};
		\node [style=none] (5) at (-1, 1) {};
		\node [style=none] (6) at (1, .25) {};
		\node [style=none] (7) at (-1, -1) {};
		\node [style=none] (8) at (1, -.25) {};
	\draw[line width = 1.5pt] (5) to [in=180, out=0,looseness = 1] (6);
	\draw[line width = 1.5pt] (7) to [in=180, out=0, looseness = 1] (8);
\end{scope}
}
      \end{tikzpicture}
\end{aligned}
}

\newcommand{\bigbimonoidsix}[1]
{
\begin{aligned}
\begin{tikzpicture}[circuit ee IEC, set resistor graphic=var resistor IEC
      graphic,scale=.4]
\scalebox{1}{
		\node [style=none] (0) at (1, 1) {};
		\node [style=none] (1) at (1, 3) {};
		\node [style=none] (2) at (-1, 1.5) {};
		\node [style=none] (3) at (-1, .5) {};
		\node [style=none] (4) at (1, 0) {};
		\node  [circle,draw,inner sep=1pt,fill]   (5) at (0.125-4, 0) {};
		\node [style=none] (6) at (-5, .5) {};
		\node [style=none] (7) at (-5, -.5) {};
		\node [style=none] (8) at (-7, -.5) {};
		\node [style=none] (9) at (-1.5, .5) {};
		\node [style=none] (10) at (-7, 2.5) {};
		\node [style=none] (11) at (-1.5, 2.5) {};
		\node [style=none] (12) at (-1, 3.5) {};
		\node [style=none] (13) at (-1, 2.5) {};
		\draw[line width = 1.5pt] (4.center) to (5.center);
		\draw[line width = 1.5pt] [in=0, out=120, looseness=1.20] (5.center) to (6.center);
		\draw[line width = 1.5pt] [in=0, out=-120, looseness=1.20] (5.center) to (7.center);
		\draw[line width = 1.5pt] (7.center) to (8.west);
		\draw[line width = 1.5pt] [in=0, out=180, looseness=.75] (6.center) to (10.west);
		\draw[line width = 1.5pt] (0) to [in=180, out=180,looseness = 6] (1);
\begin{scope}[shift={(0,4.5)}]
		\node [style=none] (0) at (1, -.5) {};
		\node  [circle,draw,inner sep=1pt,fill]   (1) at (0.125-4, -.5) {};
		\node [style=none] (2) at (-5, 0) {};
		\node [style=none] (3) at (-5, -1) {};
		\node [style=none] (8) at (-1.5, -1) {};
		\node [style=none] (9) at (-7, 0) {};
		\node [style=none] (10) at (-1.5, -3) {};
		\node [style=none] (11) at (-7, -3) {};
		\draw[line width = 1.5pt] (0.center) to (1.center);
		\draw[line width = 1.5pt] [in=0, out=120, looseness=1.20] (1.center) to (2.center);
		\draw[line width = 1.5pt] [in=0, out=-120, looseness=1.20] (1.center) to (3.center);
		\draw[line width = 1.5pt] (2.center) to (9.west);
		\draw[line width = 1.5pt] [in=0, out=180, looseness=.75] (3.center) to (11.west);
\end{scope}
\begin{scope}[rotate=180, shift={(8,-.5)}]
		\node [style=none] (0) at (1, -0) {};
		\node [style=none] (1) at (0.125, -0) {};
		\node [style=none] (2) at (-6.5, -.5) {};
		\node [style=none] (3) at (-6.5, -1.5) {};
		\node [style=none] (4) at (1, -0) {};
		\node [style=none] (5) at (-1, 1) {};
		\node [style=none] (6) at (1, .25) {};
		\node [style=none] (7) at (-1, -1) {};
		\node [style=none] (8) at (1, -.25) {};
	\draw[line width = 1.5pt] (5) to [in=180, out=0,looseness = 1] (6);
	\draw[line width = 1.5pt] (7) to [in=180, out=0, looseness = 1] (8);
\end{scope}
\begin{scope}[rotate=180, shift={(8,-3.5)}]
		\node [style=none] (0) at (1, -0) {};
		\node [style=none] (1) at (0.125, -0) {};
		\node [style=none] (2) at (-6.5, 1.5) {};
		\node [style=none] (3) at (-6.5, .5) {};
		\node [style=none] (4) at (1, -0) {};
		\node [style=none] (5) at (-1, 1) {};
		\node [style=none] (6) at (1, .25) {};
		\node [style=none] (7) at (-1, -1) {};
		\node [style=none] (8) at (1, -.25) {};
	\draw[line width = 1.5pt] (5) to [in=180, out=0,looseness = 1] (6);
	\draw[line width = 1.5pt] (7) to [in=180, out=0, looseness = 1] (8);
\end{scope}
}
      \end{tikzpicture}
\end{aligned}
}

\newcommand{\bigbimonoidseven}[1]
{
\begin{aligned}
\begin{tikzpicture}[circuit ee IEC, set resistor graphic=var resistor IEC
      graphic,scale=.4]
\scalebox{1}{
\begin{scope}[yscale=2]
		\node [style=none] (0) at (1, -0) {};
		\node  [circle,draw,inner sep=1pt,fill]   (1) at (0.125, -0) {};
		\node [style=none] (2) at (-1, 0.5) {};
		\node [style=none] (3) at (-1, -0.5) {};
		\node [style=none] (4) at (1, -.5) {};
		\node  [circle,draw,inner sep=1pt,fill]   (5) at (0.125, -.5) {};
		\node [style=none] (6) at (-1, -0) {};
		\node [style=none] (7) at (-1, -1) {};
		\node [style=none] (8) at (7, -0) {};
		\node [style=none] (9) at (7, -.5) {};
		\draw[line width = 1.5pt] (0.center) to (1.center);
		\draw[line width = 1.5pt] [in=0, out=120, looseness=1.20] (1.center) to (2.center);
		\draw[line width = 1.5pt] [in=0, out=-120, looseness=1.20] (1.center) to (3.center);
		\draw[line width = 1.5pt] (4.center) to (5.center);
		\draw[line width = 1.5pt] [in=0, out=120, looseness=1.20] (5.center) to (6.center);
		\draw[line width = 1.5pt] [in=0, out=-120, looseness=1.20] (5.center) to (7.center);
		\draw[line width = 1.5pt] (0.center) to (8.center);
		\draw[line width = 1.5pt] (4.center) to (9.center);
\begin{scope}[rotate=180, shift ={(-8,.25)}]
		\node [style=none] (0) at (1, -0) {};
		\node [style=none] (1) at (0.125, -0) {};
		\node [style=none] (2) at (-1, 0.5) {};
		\node [style=none] (3) at (-1, -0.5) {};
		\node [style=none] (4) at (1, -0) {};
		\node [style=none] (5) at (-1, 1) {};
		\node [style=none] (6) at (1, .25) {};
		\node [style=none] (7) at (-1, -1) {};
		\node [style=none] (8) at (1, -.25) {};
	\draw[line width = 1.5pt] (2) to [in=0, out=0,looseness = 3] (3);
	\draw[line width = 1.5pt] (5) to [in=180, out=0,looseness = 1] (6);
	\draw[line width = 1.5pt] (7) to [in=180, out=0, looseness = 1] (8);
\end{scope}
\end{scope}
}
      \end{tikzpicture}
\end{aligned}
}

\newcommand{\bigbimonoidstepthree}[1]
{
\begin{aligned}
\begin{tikzpicture}[circuit ee IEC, set resistor graphic=var resistor IEC
      graphic,scale=.5]
\scalebox{1}{
		\node [style=none] (0) at (1, 1) {};
		\node  [circle,draw,inner sep=1pt,fill]   (1) at (0.125, 1) {};
		\node [style=none] (2) at (-1, 1.5) {};
		\node [style=none] (3) at (-1, .5) {};
		\node [style=none] (9) at (-1.5, .5) {};
		\node [style=none] (11) at (-1.5, 2.5) {};
		\draw[line width = 1.5pt] (0.center) to (1.center);
		\draw[line width = 1.5pt] [in=0, out=120, looseness=1.20] (1.center) to (2.center);
		\draw[line width = 1.5pt] [in=0, out=-120, looseness=1.20] (1.center) to (3.center);
		\draw[line width = 1.5pt] (3.center) to (9.west);
		\draw[line width = 1.5pt] [in=0, out=180, looseness=.8] (2.center) to (11.center);
\begin{scope}[shift={(0,4.5)}]
		\node [style=none] (4) at (1, -1.5) {};
		\node  [circle,draw,inner sep=1pt,fill]   (5) at (0.125, -1.5) {};
		\node [style=none] (6) at (-1, -1) {};
		\node [style=none] (7) at (-1, -2) {};
		\node [style=none] (8) at (-1.5, -1) {};
		\node [style=none] (10) at (-1.5, -3) {};
		\draw[line width = 1.5pt] (4.center) to (5.center);
		\draw[line width = 1.5pt] [in=0, out=120, looseness=1.20] (5.center) to (6.center);
		\draw[line width = 1.5pt] [in=0, out=-120, looseness=1.20] (5.center) to (7.center);
		\draw[line width = 1.5pt] (6.center) to (8.west);
		\draw[line width = 1.5pt] [in=0, out=180, looseness=.8] (7.center) to (10.center);
\end{scope}
\begin{scope}[rotate=180, shift={(8,0)}]
		\node [style=none] (2) at (-6.5, -.5) {};
		\node [style=none] (3) at (-6.5, -1.5) {};
	\draw[line width = 1.5pt] (2.east) to [in=0, out=0,looseness = 3] (3.east);
\end{scope}
\begin{scope}[rotate=180, shift={(8,-4)}]
		\node [style=none] (2) at (-6.5, 1.5) {};
		\node [style=none] (3) at (-6.5, .5) {};
	\draw[line width = 1.5pt] (2.east) to [in=0, out=0,looseness = 3] (3.east);
\end{scope}
}
      \end{tikzpicture}
\end{aligned}
}

\newcommand{\bigbimonoidstepfour}[1]
{
\begin{aligned}
\begin{tikzpicture}[circuit ee IEC, set resistor graphic=var resistor IEC
      graphic,scale=.5]
\scalebox{1}{
		\node [style=none] (0) at (1, 1) {};
		\node  [circle,draw,inner sep=1pt,fill]   (1) at (0.125, 1) {};
		\node [style=none] (2) at (-1, 1.5) {};
		\node [style=none] (3) at (-1, .5) {};
		\node [style=none] (9) at (-1.5, .5) {};
		\node [style=none] (11) at (-1.5, 2.5) {};
		\node [style=none] (12) at (-1, 3.5) {};
		\node [style=none] (13) at (-1, 2.5) {};
		\draw[line width = 1.5pt] (0.center) to (1.center);
		\draw[line width = 1.5pt] (12.east) to [in=180, out=180,looseness = 2.5] (3.east);
		\draw[line width = 1.5pt] (13.east) to [in=180, out=180,looseness = 3] (2.east);
		\draw[line width = 1.5pt] [in=0, out=120, looseness=1.20] (1.center) to (2.center);
		\draw[line width = 1.5pt] [in=0, out=-120, looseness=1.20] (1.center) to (3.center);
\begin{scope}[shift={(0,4.5)}]
		\node [style=none] (4) at (1, -1.5) {};
		\node  [circle,draw,inner sep=1pt,fill]   (5) at (0.125, -1.5) {};
		\node [style=none] (6) at (-1, -1) {};
		\node [style=none] (7) at (-1, -2) {};
		\node [style=none] (8) at (-1.5, -1) {};
		\node [style=none] (10) at (-1.5, -3) {};
		\draw[line width = 1.5pt] (4.center) to (5.center);
		\draw[line width = 1.5pt] [in=0, out=120, looseness=1.20] (5.center) to (6.center);
		\draw[line width = 1.5pt] [in=0, out=-120, looseness=1.20] (5.center) to (7.center);
\end{scope}
}
      \end{tikzpicture}
\end{aligned}
}

\newcommand{\bigbimonoidstepfive}[1]
{
\begin{aligned}
\begin{tikzpicture}[circuit ee IEC, set resistor graphic=var resistor IEC
      graphic,scale=.5]
\scalebox{1}{
		\node [style=none] (0) at (-1, -0) {};
		\node  [circle,draw,inner sep=1pt,fill]   (1) at (-0.125, -0) {};
		\node [style=none] (2) at (1, 0.5) {};
		\node [style=none] (3) at (1, -0.5) {};
		\draw[line width = 1.5pt] (0.center) to (1.center);
		\draw[line width = 1.5pt] [in=180, out=60, looseness=1.20] (1.center) to (2.east);
		\draw[line width = 1.5pt] [in=180, out=-60, looseness=1.20] (1.center) to (3.east);
\begin{scope}[shift={(2,1)}]
		\node [style=none] (0) at (1, -0) {};
		\node [style=none] (1) at (0.125, -0) {};
		\node [style=none] (2) at (-1, 0.5) {};
		\node [style=none] (3) at (-1, -0.5) {};
		\node [style=none] (4) at (1, -0) {};
		\node [style=none] (5) at (-3, 0.5) {};
	\draw[line width = 1.5pt] (2) to [in=0, out=0,looseness = 4] (3);
	\draw[line width = 1.5pt] (2.east) to (5.center);
\end{scope}
}
      \end{tikzpicture}
\end{aligned}
}

\newcommand{\bimonoidextral}[1]
{
  \begin{aligned}
\begin{tikzpicture}[circuit ee IEC, set resistor graphic=var resistor IEC
      graphic,scale=.5]
\scalebox{1}{
		\node [style=none] (0) at (1, -0) {};
		\node [style=none] (1) at (-1, -0) {};
		\node  [circle,draw,inner sep=1pt,fill]   (2) at (0, -0) {};
		\node [style=none] (3) at (1, -.5) {};
		\node [style=none] (4) at (-1, -.5) {};
		\node  [circle,draw,inner sep=1pt,fill]   (5) at (0, -.5) {};
		\draw[line width = 1.5pt] (0.center) to (2);
		\draw[line width = 1.5pt] (3.center) to (5);
		\draw[line width = 1.5pt] (0.west) to [in=0, out=0,looseness = 3] (3.west);
}
      \end{tikzpicture}
  \end{aligned}
}

\newcommand{\bimonoidextram}[1]
{
  \begin{aligned}
\begin{tikzpicture}[circuit ee IEC, set resistor graphic=var resistor IEC
      graphic,scale=.5]
\scalebox{1}{
		\node [style=none] (0) at (1, -0) {};
		\node [style=none] (1) at (-1, -0) {};
		\node  [circle,draw,inner sep=1pt,fill]   (2) at (0, -0) {};
		\node [style=none] (3) at (1, -.5) {};
		\node [style=none] (4) at (-1, -.5) {};
		\node  [circle,draw,inner sep=1pt,fill]   (5) at (0, -.5) {};
		\node  [circle,draw,inner sep=1pt,fill]   (6) at (3, -.25) {};
		\node [style=none] (7) at (1.4, -.25) {};
		\draw[line width = 1.5pt] (0.center) to (2);
		\draw[line width = 1.5pt] (3.center) to (5);
		\draw[line width = 1.5pt] (6.center) to (7.west);
		\draw[line width = 1.5pt] (0.west) to [in=0, out=0,looseness = 3] (3.west);
}
      \end{tikzpicture}
  \end{aligned}
}

\newcommand{\bimonoidextrar}[1]
{
  \begin{aligned}
\begin{tikzpicture}[circuit ee IEC, set resistor graphic=var resistor IEC
      graphic,scale=.5]
\scalebox{1}{
		\node [style=none] (0) at (1, -0) {};
		\node [style=none] (1) at (-1, -0) {};
		\node  [circle,draw,inner sep=1pt,fill]   (2) at (0, -0) {};
		\node [style=none] (3) at (1, -.5) {};
		\node [style=none] (4) at (-1, -.5) {};
		\node  [circle,draw,inner sep=1pt,fill]   (5) at (0, -.5) {};
		\draw[line width = 1.5pt] (0.center) to (2);
		\draw[line width = 1.5pt] (3.center) to (5);
}
      \end{tikzpicture}
  \end{aligned}
}

\newcommand{\parcoml}[1]
{
\begin{aligned}
\begin{tikzpicture}[circuit ee IEC, set resistor graphic=var resistor IEC
      graphic,scale=.5]
\scalebox{1}{
		\node [style=none] (0) at (1, -0) {};
		\node  [circle,draw,inner sep=1pt,fill]   (1) at (0.125, -0) {};
		\node [style=none] (2) at (-1, 0.5) {};
		\node [style=none] (3) at (-1, -0.5) {};
		\node [style=none] (4) at (1, -.5) {};
		\node  [circle,draw,inner sep=1pt,fill]   (5) at (0.125, -.5) {};
		\node [style=none] (6) at (-1, -0) {};
		\node [style=none] (7) at (-1, -1) {};
		\node [style=none] (8) at (-3, .5) {};
		\node [style=none] (9) at (-3, -.5) {};
		\node [style=none] (10) at (-3, -0) {};
		\node [style=none] (11) at (-3, -1) {};
		\draw[line width = 1.5pt] (0.center) to (1.center);
		\draw[line width = 1.5pt] [in=0, out=120, looseness=1.20] (1.center) to (2.center);
		\draw[line width = 1.5pt] [in=0, out=-120, looseness=1.20] (1.center) to (3.center);
		\draw[line width = 1.5pt] (4.center) to (5.center);
		\draw[line width = 1.5pt] [in=0, out=120, looseness=1.20] (5.center) to (6.center);
		\draw[line width = 1.5pt] [in=0, out=-120, looseness=1.20] (5.center) to (7.center);
		\draw[line width = 1.5pt] [in=180, out=0, looseness=1] (9.center) to (2.center);
		\draw[line width = 1.5pt] [in=180, out=0, looseness=1] (8.center) to (3.center);
		\draw[line width = 1.5pt] [in=180, out=0, looseness=1] (11.center) to (6.center);
		\draw[line width = 1.5pt] [in=180, out=0, looseness=1] (10.center) to (7.center);
}
      \end{tikzpicture}
\end{aligned}
}

\newcommand{\parcomm}[1]
{
\begin{aligned}
\begin{tikzpicture}[circuit ee IEC, set resistor graphic=var resistor IEC
      graphic,scale=.5]
\scalebox{1}{
		\node [style=none] (0) at (1, -0) {};
		\node  [circle,draw,inner sep=1pt,fill]   (1) at (0.125, -0) {};
		\node [style=none] (2) at (-1, 0.5) {};
		\node [style=none] (3) at (-1, -0.5) {};
		\node [style=none] (4) at (1, -.5) {};
		\node  [circle,draw,inner sep=1pt,fill]   (5) at (0.125, -.5) {};
		\node [style=none] (6) at (-1, -0) {};
		\node [style=none] (7) at (-1, -1) {};
		\node [style=none] (8) at (-3, .5) {};
		\node [style=none] (9) at (-3, -.5) {};
		\node [style=none] (10) at (-3, -0) {};
		\node [style=none] (11) at (-3, -1) {};
		\draw[line width = 1.5pt] (0.center) to (1.center);
		\draw[line width = 1.5pt] [in=0, out=120, looseness=1.20] (1.center) to (2.center);
		\draw[line width = 1.5pt] [in=0, out=-120, looseness=1.20] (1.center) to (3.center);
		\draw[line width = 1.5pt] (4.center) to (5.center);
		\draw[line width = 1.5pt] [in=0, out=120, looseness=1.20] (5.center) to (6.center);
		\draw[line width = 1.5pt] [in=0, out=-120, looseness=1.20] (5.center) to (7.center);
		\draw[line width = 1.5pt] [in=180, out=0, looseness=1] (8.center) to (2.center);
		\draw[line width = 1.5pt] [in=180, out=0, looseness=1] (9.center) to (3.center);
		\draw[line width = 1.5pt] [in=180, out=0, looseness=1] (11.center) to (6.center);
		\draw[line width = 1.5pt] [in=180, out=0, looseness=1] (10.center) to (7.center);
}
      \end{tikzpicture}
\end{aligned}
}

\newcommand{\parcomr}[1]
{
\begin{aligned}
\begin{tikzpicture}[circuit ee IEC, set resistor graphic=var resistor IEC
      graphic,scale=.5]
\scalebox{1}{
		\node [style=none] (0) at (1, -0) {};
		\node  [circle,draw,inner sep=1pt,fill]   (1) at (0.125, -0) {};
		\node [style=none] (2) at (-1, 0.5) {};
		\node [style=none] (3) at (-1, -0.5) {};
		\node [style=none] (4) at (1, -.5) {};
		\node  [circle,draw,inner sep=1pt,fill]   (5) at (0.125, -.5) {};
		\node [style=none] (6) at (-1, -0) {};
		\node [style=none] (7) at (-1, -1) {};
		\node [style=none] (8) at (-3, .5) {};
		\node [style=none] (9) at (-3, -.5) {};
		\node [style=none] (10) at (-3, -0) {};
		\node [style=none] (11) at (-3, -1) {};
		\draw[line width = 1.5pt] (0.center) to (1.center);
		\draw[line width = 1.5pt] [in=0, out=120, looseness=1.20] (1.center) to (2.center);
		\draw[line width = 1.5pt] [in=0, out=-120, looseness=1.20] (1.center) to (3.center);
		\draw[line width = 1.5pt] (4.center) to (5.center);
		\draw[line width = 1.5pt] [in=0, out=120, looseness=1.20] (5.center) to (6.center);
		\draw[line width = 1.5pt] [in=0, out=-120, looseness=1.20] (5.center) to (7.center);
		\draw[line width = 1.5pt] [in=180, out=0, looseness=1] (8.center) to (2.center);
		\draw[line width = 1.5pt] [in=180, out=0, looseness=1] (9.center) to (3.center);
		\draw[line width = 1.5pt] [in=180, out=0, looseness=1] (10.center) to (6.center);
		\draw[line width = 1.5pt] [in=180, out=0, looseness=1] (11.center) to (7.center);
}
      \end{tikzpicture}
\end{aligned}
}

\newcommand{\spec}[1]
{
  \begin{aligned}
\begin{tikzpicture}[circuit ee IEC, set resistor graphic=var resistor IEC
      graphic,scale=.5]
\scalebox{1}{
		\node [style=none] (0) at (1.75, -0) {};
		\node  [circle,draw,inner sep=1pt,fill]   (1) at (0.75, -0) {};
		\node [style=none] (2) at (0, -0.5) {};
		\node [style=none] (3) at (0, 0.5) {};
		\node  [circle,draw,inner sep=1pt,fill]   (4) at (-0.75, -0) {};
		\node [style=none] (5) at (0, -0.5) {};
		\node [style=none] (6) at (-1.75, -0) {};
		\node [style=none] (7) at (0, 0.5) {};
		\draw[line width = 1.5pt] (0.center) to (1.center);
		\draw[line width = 1.5pt] [in=0, out=-120, looseness=1.20] (1.center) to (2.center);
		\draw[line width = 1.5pt] [in=0, out=120, looseness=1.20] (1.center) to (3.center);
		\draw[line width = 1.5pt] (6.center) to (4);
		\draw[line width = 1.5pt] [in=180, out=-60, looseness=1.20] (4) to (5.center);
		\draw[line width = 1.5pt] [in=180, out=60, looseness=1.20] (4) to (7.center);
}
\end{tikzpicture}
  \end{aligned}
}

\newcommand{\parspec}[1]
{
\begin{aligned}
\begin{tikzpicture}[circuit ee IEC, set resistor graphic=var resistor IEC
      graphic,scale=.5]
\scalebox{1}{
		\node [style=none] (0) at (1, -0) {};
		\node  [circle,draw,inner sep=1pt,fill]   (1) at (0.125, -0) {};
		\node [style=none] (2) at (-1, 0.5) {};
		\node [style=none] (3) at (-1, -0.5) {};
		\node [style=none] (4) at (1, -.5) {};
		\node  [circle,draw,inner sep=1pt,fill]   (5) at (0.125, -.5) {};
		\node [style=none] (6) at (-1, -0) {};
		\node [style=none] (7) at (-1, -1) {};
		\draw[line width = 1.5pt] (0.center) to (1.center);
		\draw[line width = 1.5pt] [in=0, out=120, looseness=1.20] (1.center) to (2.west);
		\draw[line width = 1.5pt] [in=0, out=-120, looseness=1.20] (1.center) to (3.west);
		\draw[line width = 1.5pt] (4.center) to (5.center);
		\draw[line width = 1.5pt] [in=0, out=120, looseness=1.20] (5.center) to (6.west);
		\draw[line width = 1.5pt] [in=0, out=-120, looseness=1.20] (5.center) to (7.west);
\begin{scope}[shift={(-2,-.5)}, rotate = 180]
		\node [style=none] (0) at (1, -0) {};
		\node  [circle,draw,inner sep=1pt,fill]   (1) at (0.125, -0) {};
		\node [style=none] (2) at (-1, 0.5) {};
		\node [style=none] (3) at (-1, -0.5) {};
		\node [style=none] (4) at (1, -.5) {};
		\node  [circle,draw,inner sep=1pt,fill]   (5) at (0.125, -.5) {};
		\node [style=none] (6) at (-1, -0) {};
		\node [style=none] (7) at (-1, -1) {};
		\draw[line width = 1.5pt] (0.center) to (1.center);
		\draw[line width = 1.5pt] [in=0, out=120, looseness=1.20] (1.center) to (2.center);
		\draw[line width = 1.5pt] [in=0, out=-120, looseness=1.20] (1.center) to (3.center);
		\draw[line width = 1.5pt] (4.center) to (5.center);
		\draw[line width = 1.5pt] [in=0, out=120, looseness=1.20] (5.center) to (6.center);
		\draw[line width = 1.5pt] [in=0, out=-120, looseness=1.20] (5.center) to (7.center);
\end{scope}
}
      \end{tikzpicture}
\end{aligned}
}

\newcommand{\parextra}[1]
{
  \begin{aligned}
\begin{tikzpicture}[circuit ee IEC, set resistor graphic=var resistor IEC
      graphic,scale=.5]
\scalebox{1}{
		\node [style=none] (0) at (1, -0) {};
		\node [style=none] (1) at (-1, -0) {};
		\node  [circle,draw,inner sep=1pt,fill]   (2) at (0, -0) {};
		\node [style=none] (3) at (1, -.5) {};
		\node [style=none] (4) at (-1, -.5) {};
		\node  [circle,draw,inner sep=1pt,fill]   (5) at (0, -.5) {};
		\draw[line width = 1.5pt] (0.center) to (2.center);
		\draw[line width = 1.5pt] (3.center) to (5.center);
\begin{scope}[shift={(1.5,-.5)}, rotate = 180]
		\node [style=none] (0) at (1, -0) {};
		\node [style=none] (1) at (-1, -0) {};
		\node  [circle,draw,inner sep=1pt,fill]   (2) at (0, -0) {};
		\node [style=none] (3) at (1, -.5) {};
		\node [style=none] (4) at (-1, -.5) {};
		\node  [circle,draw,inner sep=1pt,fill]   (5) at (0, -.5) {};
		\draw[line width = 1.5pt] (0.center) to (2);
		\draw[line width = 1.5pt] (3.center) to (5);
\end{scope}
}
      \end{tikzpicture}
  \end{aligned}
}

\newcommand{\monadextra}[1]
{
\begin{aligned}
\begin{tikzpicture}[circuit ee IEC, set resistor graphic=var resistor IEC
      graphic,scale=.5]
\scalebox{1}{
		\node [style=none] (0) at (1, -0) {};
		\node [style=none] (1) at (0.125, -0) {};
		\node [style=none] (2) at (-1, 0.5) {};
		\node [style=none] (3) at (-1, -0.5) {};
		\node [style=none] (4) at (1, -0) {};
	\draw[line width = 1.5pt] (2.east) to [in=0, out=0,looseness = 3] (3.west);
	\draw[line width = 1.5pt] (2.east) to [in=180, out=180,looseness = 3] (3.west);
}
      \end{tikzpicture}
\end{aligned}
}

\newcommand{\monadextram}[1]
{
\begin{aligned}
\begin{tikzpicture}[circuit ee IEC, set resistor graphic=var resistor IEC
      graphic,scale=.5]
\scalebox{1}{
		\node [style=none] (0) at (1, -0) {};
		\node  [circle,draw,inner sep=1pt,fill]   (1) at (0.125, -0) {};
		\node [style=none] (2) at (-1, 0.5) {};
		\node [style=none] (3) at (-1, -0.5) {};
		\node  [circle,draw,inner sep=1pt,fill]   (4) at (1, -0) {};
		\draw[line width = 1.5pt] (0.center) to (1.center);
		\draw[line width = 1.5pt] [in=0, out=120, looseness=1.20] (1.center) to (2.center);
		\draw[line width = 1.5pt] [in=0, out=-120, looseness=1.20] (1.center) to (3.center);
\begin{scope}[shift={(-2,0)}, rotate = 180]
		\node [style=none] (0) at (1, -0) {};
		\node  [circle,draw,inner sep=1pt,fill]   (1) at (0.125, -0) {};
		\node [style=none] (2) at (-1, 0.5) {};
		\node [style=none] (3) at (-1, -0.5) {};
		\node  [circle,draw,inner sep=1pt,fill]   (4) at (1, -0) {};
		\draw[line width = 1.5pt] (0.center) to (1.center);
		\draw[line width = 1.5pt] [in=0, out=120, looseness=1.20] (1.center) to (2.center);
		\draw[line width = 1.5pt] [in=0, out=-120, looseness=1.20] (1.center) to (3.center);
\end{scope}
}
      \end{tikzpicture}
\end{aligned}
}

\newcommand{\monadspec}[1]
{
\begin{aligned}
\begin{tikzpicture}[circuit ee IEC, set resistor graphic=var resistor IEC
      graphic,scale=.5]
\scalebox{1}{
		\node [style=none] (0) at (1, -0) {};
		\node [style=none] (1) at (0.125, -0) {};
		\node [style=none] (2) at (-1, 0.5) {};
		\node [style=none] (3) at (-1, -0.5) {};
		\node [style=none] (4) at (1, -0) {};
		\node [style=none] (5) at (-1, 1) {};
		\node [style=none] (6) at (1, .25) {};
		\node [style=none] (7) at (-1, -1) {};
		\node [style=none] (8) at (1, -.25) {};
	\draw[line width = 1.5pt] (2.west) to [in=0, out=0,looseness = 3] (3.west);
	\draw[line width = 1.5pt] (5.west) to [in=180, out=0,looseness = 1] (6);
	\draw[line width = 1.5pt] (7.west) to [in=180, out=0, looseness = 1] (8);
\begin{scope}[shift={(-2,0)}, rotate = 180]
		\node [style=none] (0) at (1, -0) {};
		\node [style=none] (1) at (0.125, -0) {};
		\node [style=none] (2) at (-1, 0.5) {};
		\node [style=none] (3) at (-1, -0.5) {};
		\node [style=none] (4) at (1, -0) {};
		\node [style=none] (5) at (-1, 1) {};
		\node [style=none] (6) at (1, .25) {};
		\node [style=none] (7) at (-1, -1) {};
		\node [style=none] (8) at (1, -.25) {};
	\draw[line width = 1.5pt] (2) to [in=0, out=0,looseness = 3] (3);
	\draw[line width = 1.5pt] (5) to [in=180, out=0,looseness = 1] (6);
	\draw[line width = 1.5pt] (7) to [in=180, out=0, looseness = 1] (8);
\end{scope}
}
      \end{tikzpicture}
\end{aligned}
}

\newcommand{\extral}[1]
{
  \begin{aligned}
\begin{tikzpicture}[circuit ee IEC, set resistor graphic=var resistor IEC
      graphic,scale=.5]
\scalebox{1}{
		\node [style=none] (0) at (1.75, -0) {};
		\node  [circle,draw,inner sep=1pt,fill]   (1) at (0.75, -0) {};
		\node  [circle,draw,inner sep=1pt,fill]   (4) at (-0.75, -0) {};
		\node [style=none] (6) at (-1.75, -0) {};
	  \draw[line width = 1.5pt] (1.center) to (4.center);
}
\end{tikzpicture}
  \end{aligned}
}

\newcommand{\extrar}[1]
{
  \begin{aligned}
\begin{tikzpicture}[circuit ee IEC, set resistor graphic=var resistor IEC
      graphic,scale=.5]
\scalebox{1}{
		\node [style=none] (0) at (1.75, -0) {};
		\node [style=none] (6) at (-1.75, -0) {};
}
\end{tikzpicture}
  \end{aligned}
}

\newcommand{\strangelawl}[1]
{
\begin{aligned}
\begin{tikzpicture}[circuit ee IEC, set resistor graphic=var resistor IEC
      graphic,scale=.5]
\scalebox{1}{
		\node [style=none] (0) at (1, -0) {};
		\node [style=none] (1) at (0.125, -0) {};
		\node [style=none] (2) at (-1, 0.5) {};
		\node [style=none] (3) at (-1, -0.5) {};
		\node [style=none] (4) at (1, -0) {};
		\node [style=none] (5) at (-1, 1) {};
		\node [style=none] (6) at (1, .25) {};
		\node [style=none] (7) at (-1, -1) {};
		\node [style=none] (8) at (1, -.25) {};
		\node [style=none] (9) at (-2, .75) {};
		\node [style=none] (10) at (-2, -.75) {};
		\node [style=none] (11) at (-2, .25) {};
		\node [style=none] (12) at (-2, -.25) {};
	\draw[line width = 1.5pt] (2) to [in=0, out=0,looseness = 3] (3);
	\draw[line width = 1.5pt] (5) to [in=180, out=0,looseness = 1] (6);
	\draw[line width = 1.5pt] (7) to [in=180, out=0, looseness = 1] (8);
	\draw[line width = 1.5pt] (5.east) to [in=0, out=180,looseness = 1] (9.west);
	\draw[line width = 1.5pt] (7.east) to [in=0, out=180, looseness = 1] (10.west);
	\draw[line width = 1.5pt] (2.east) to [in=0, out=180,looseness = 1] (11.west);
	\draw[line width = 1.5pt] (3.east) to [in=0, out=180, looseness = 1] (12.west);
\begin{scope}[shift={(-3,-.25)}]
\begin{scope}[rotate=180]
		\node [style=none] (0) at (1, -0) {};
		\node  [circle,draw,inner sep=1pt,fill]   (1) at (0.125, -0) {};
		\node [style=none] (2) at (-1, 0.5) {};
		\node [style=none] (3) at (-1, -0.5) {};
		\node [style=none] (4) at (1, -.5) {};
		\node  [circle,draw,inner sep=1pt,fill]   (5) at (0.125, -.5) {};
		\node [style=none] (6) at (-1, -0) {};
		\node [style=none] (7) at (-1, -1) {};
		\draw[line width = 1.5pt] (0.center) to (1.center);
		\draw[line width = 1.5pt] [in=0, out=120, looseness=1.20] (1.center) to (2.center);
		\draw[line width = 1.5pt] [in=0, out=-120, looseness=1.20] (1.center) to (3.center);
		\draw[line width = 1.5pt] (4.center) to (5.center);
		\draw[line width = 1.5pt] [in=0, out=120, looseness=1.20] (5.center) to (6.center);
		\draw[line width = 1.5pt] [in=0, out=-120, looseness=1.20] (5.center) to (7.center);
\end{scope}
\end{scope}
}
      \end{tikzpicture}
\end{aligned}
}

\newcommand{\strangelawr}[1]
{
\begin{aligned}
\begin{tikzpicture}[circuit ee IEC, set resistor graphic=var resistor IEC
      graphic,scale=.5]
\scalebox{1}{
\begin{scope}[rotate=180]
		\node [style=none] (0) at (1, -0) {};
		\node [style=none] (1) at (0.125, -0) {};
		\node [style=none] (2) at (-1, 0.5) {};
		\node [style=none] (3) at (-1, -0.5) {};
		\node [style=none] (4) at (1, -0) {};
		\node [style=none] (5) at (-1, 1) {};
		\node [style=none] (6) at (1, .25) {};
		\node [style=none] (7) at (-1, -1) {};
		\node [style=none] (8) at (1, -.25) {};
		\node [style=none] (9) at (-2, .75) {};
		\node [style=none] (10) at (-2, -.75) {};
		\node [style=none] (11) at (-2, .25) {};
		\node [style=none] (12) at (-2, -.25) {};
	\draw[line width = 1.5pt] (2) to [in=0, out=0,looseness = 3] (3);
	\draw[line width = 1.5pt] (5) to [in=180, out=0,looseness = 1] (6);
	\draw[line width = 1.5pt] (7) to [in=180, out=0, looseness = 1] (8);
	\draw[line width = 1.5pt] (5.west) to [in=0, out=180,looseness = 1] (9.east);
	\draw[line width = 1.5pt] (7.west) to [in=0, out=180, looseness = 1] (10.east);
	\draw[line width = 1.5pt] (2.west) to [in=0, out=180,looseness = 1] (11.east);
	\draw[line width = 1.5pt] (3.east) to [in=0, out=180, looseness = 1] (12.east);
\begin{scope}[shift={(-3,-.25)}]
\begin{scope}[rotate=180]
		\node [style=none] (0) at (1, -0) {};
		\node  [circle,draw,inner sep=1pt,fill]   (1) at (0.125, -0) {};
		\node [style=none] (2) at (-1, 0.5) {};
		\node [style=none] (3) at (-1, -0.5) {};
		\node [style=none] (4) at (1, -.5) {};
		\node  [circle,draw,inner sep=1pt,fill]   (5) at (0.125, -.5) {};
		\node [style=none] (6) at (-1, -0) {};
		\node [style=none] (7) at (-1, -1) {};
		\draw[line width = 1.5pt] (0.center) to (1.center);
		\draw[line width = 1.5pt] [in=0, out=120, looseness=1.20] (1.center) to (2.center);
		\draw[line width = 1.5pt] [in=0, out=-120, looseness=1.20] (1.center) to (3.center);
		\draw[line width = 1.5pt] (4.center) to (5.center);
		\draw[line width = 1.5pt] [in=0, out=120, looseness=1.20] (5.center) to (6.center);
		\draw[line width = 1.5pt] [in=0, out=-120, looseness=1.20] (5.center) to (7.center);
\end{scope}
\end{scope}
\end{scope}
}
      \end{tikzpicture}
\end{aligned}
}

\begin{document}

\begin{center}
{\bf A Compositional Framework for Bond Graphs\\}
\vspace{0.3cm}
{\em Brandon Coya\\}
\vspace{0.3cm}
{\small Department of Mathematics\\
University of California\\
Riverside CA, USA 92521 \\ }
\vspace{0.3cm}
{\small email: bcoya001@ucr.edu\\}
\vspace{0.3cm}
{\small \today}
\vspace{0.3cm}
\end{center}

\begin{abstract}
Electrical circuits made only of perfectly conductive wires can be seen as partitions between finite sets: that is, isomorphism classes of jointly epic cospans. These are also known as ``corelations" and are the morphisms in the category $\Fin\Corel$. The two-element set has two different Frobenius monoid structures in $\Fin\Corel$. These two Frobenius monoids are related to ``series" and ``parallel" junctions, which are used to connect pairs of wires and electrical components together.   We show that these Frobenius monoids interact to form a ``weak bimonoid" as defined by Pastro and Street. We conjecture a presentation for the subcategory of $\Fin\Corel$ generated by the morphisms associated to these two Frobenius monoids, which we call $\Fin\Corel^{\circ}$. We are interested in ``bond graphs," which are studied in electrical engineering and are built from series and parallel junctions. Although the morphisms of $\Fin\Corel^{\circ}$ resemble bond graphs, there is not a perfect correspondence. This motivates the search for a category whose morphisms more precisely model bond graphs. We approach this by considering a subcategory, $\Lag\Rel_k^{\circ}$, of the category of Lagrangian relations, $\Lag\Rel_k$. This is because both bond graphs and circuits determine Lagrangian relations between symplectic vector spaces. The categories $\Fin\Corel^{\circ}$ and $\Lag\Rel_k^{\circ}$ have a correspondence between their sets of generating morphisms.  Thus we define the category $\BondGraph$ by using generators and imposing equations that are found in both $\Fin\Corel^{\circ}$ and $\Lag\Rel_k^{\circ}$. We study the functorial semantics of $\BondGraph$ by giving two different functors from it to the category $\Lag\Rel_k$ and a natural transformation between them. Given a bond graph, the first functor picks out a Lagrangian relation in terms of ``effort" and ``flow," while the second picks one out in terms of ``potential" and ``current."  The natural transformation arises from the way that effort and flow relate to potential and current.
\end{abstract}


\section{Introduction}
 In the 1980s Joyal and Street \cite{JS1} showed that string diagrams can be used to reason about morphisms in any symmetric monoidal category. Prior to this, scientists had been using diagrams to visualize and better understand problems for some time. Feynman introduced his diagrams in 1949 and particle physicists have been using them ever since \cite{Fe,Ka}.  Before category theory was even introduced, electrical engineers were using circuit diagrams to study electrical circuit networks. In the 1940s Olson \cite{Ol} pointed out that analogies between electrical, mechanical, thermodynamic, hydraulic, and chemical networks allow circuit diagrams to be applied to a wide variety of fields. Further examples of diagrams include signal flow diagrams used in control theory \cite{F}, and Petri nets which originated in chemistry \cite{P} but are now widely used in computer science.

Networks, such as the aforementioned examples, are often studied from a ``compositional" perspective. To study networks in this way, we allow them to have inputs and outputs. We then build larger networks out of smaller ones by attaching inputs from one network to outputs of another network, so long as they are compatible.  For some types of networks this process results in something complicated and hard to understand, even if  the two smaller parts are completely understood. For example, the behaviors of two independent pendulums are not chaotic, but when attached to make a double pendulum the resulting behavior is chaotic. Given a network, we want to know how much of the behavior of the network is determined by the behavior of the parts.

This is where the connection between category theory and networks  comes into play. With some careful work, a diagram of a network can be viewed as a morphism in a category. The process of constructing larger diagrams from smaller pieces is then viewed as composition of morphisms. From here, behavior that is preserved when constructing larger networks corresponds to assigning behavior to the morphisms using a functor. This is formalized by functorial semantics, which was first introduced by Lawvere \cite{Law} in his thesis. For each type of network, the diagrams act as a syntax, and the behaviors act as a semantics.

Baez and Fong \cite{BF} showed how to describe circuits made of wires, resistors, capacitors, and inductors, as morphisms in a category by introducing the notion of ``decorated cospans." Baez and the author \cite{BC} have used also props (traditionally written as PROPs) to also study circuits from another perspective.  Erbele \cite{BE} and \cite{BSZ} separately considered signal flow diagrams as morphisms in some category. Baez, Fong, and Pollard \cite{BFP} looked at Markov processes, while Baez and Pollard \cite{BP} studied  reaction networks using decorated cospans. In all of these cases, the functorial semantics of the categories are studied as well. The goal of this paper is to construct and study the functorial semantics of the category $\BondGraph$ whose morphisms correspond to ``bond graphs."

On April 24, 1959, the engineer Paynter \cite{HP, HP2} woke up and invented the diagrammatic language of bond graphs to study generalized versions of voltage and current, which are implicit in the analogies found by Olson \cite{Ol}. Engineers call these generalized notions for voltage and current ``effort" and ``flow," respectively.  There are also generalizations of resistors, capacitors, and inductors. A well known example of effort and flow is in mechanical systems, where effort is force and flow is velocity.  In hydraulic systems effort is pressure and flow is velocity. A lesser known example is that in chemistry, chemical potential acts as effort and molar flow acts as flow. The idea is that in all of these cases the effort and flow behave in mathematically the same way.

The analogies between electrical, chemical, mechanical, and hydraulic and other systems are by now well-understood, and have been ably explained by Karnopp, Margolis, and Rosenberg \cite{KMR, KR1, KR3}and Brown \cite{Brown}. Bond graphs are now widely used in engineering \cite{JT, JD}. Their power lies in providing a common language for these various branches of engineering, which engineers call ``energy domains." Therefore it seems fitting to introduce bond graphs into category theory.

To more easily understand bond graphs we focus primarily on their use in electrical engineering. To begin, we associate a pair of real numbers called potential, $\phi$, and current, $I$, to any wire:
\vspace{-2ex}
\begin{figure}[H]
\centering
\begin{tikzpicture}
[circuit ee IEC, set resistor graphic=var resistor IEC
      graphic,scale=.5]
\scalebox{1}{
		\node [style=none] (0) at (-1, -0) {};
		\node [style=none] (1) at (1, -0) {};
		\node [style=none] (2) at (0, 0.5) {};
		\node [style=none] (3) at (0, -0.5) {};
		\draw[line width = 1.5pt] node[above] {$\phi,I$}  (1.center) to (0.center);
}
\end{tikzpicture} 
\end{figure}

\vspace{-2ex}

\noindent The end of a wire is called a ``terminal"  so we say that wires go between terminals. Later we shall associate to potential and current the terminals instead, but this is just a choice of preference since there is no real difference. Engineers often find it useful to work with wires in pairs.  For example, in electrical sockets and household appliances the wires come in pairs. Engineers want the current flowing along one wire to be the negative of the current flowing along the other wire. The difference in potential between the wires is also more important the the individual potentials. For such wires:

\begin{figure}[H]
\centering
\begin{tikzpicture}
[circuit ee IEC, set resistor graphic=var resistor IEC
      graphic,scale=.5]
\scalebox{1}{
		\node [style=none] (2) at (-1, -.25) {};
		\node [style=none] (3) at (1, -.25) {};
		\draw[line width = 1.5pt] node[below=.15] {$\phi_2,I_2$}  (2.center) to (3.center);
		\node [style=none] (4) at (-1, .25) {};
		\node [style=none] (5) at (1, .25) {};
		\draw[line width = 1.5pt] node[above=.15] {$\phi_2,I_2$}  (4.center) to (5.center);
}
\end{tikzpicture} 
\end{figure}
\vspace{-2ex}
 \noindent we call $V=\phi_2-\phi_1$ the ``voltage" and $I=I_1=-I_2$ the ``current." Thus current on a single wire and current on a pair of wires are slightly different, yet called the same thing. Engineers call a pair of wires like this a ``bond" and the pair of terminals a ``port." Thus bonds go between ports, and in a ``bond graph" we draw a bond as follows:  
\vspace{-2ex}
\begin{figure}[H]
\centering
\begin{tikzpicture}[circuit ee IEC, set resistor graphic=var resistor IEC
      graphic, scale=0.8, every node/.append style={transform shape}]
[
	node distance=1.5cm,
	mynewelement/.style={
		color=blue!50!black!75,
		thick
	},
	mybondmodpoint/.style={
	rectangle,
	minimum size=3mm,
	very thick,
	draw=red!50!black!50, 
	outer sep=2pt
	}
]		
	\node(J11) {};
	\node (R3) [right=1 of J11] {}
	edge [line width=3.5pt]    node [below]{$I$} (J11)
        edge  [line width=3.5pt]  node [above]{$V$} (J11);
\end{tikzpicture}
\end{figure}

\vspace{-2ex}

Voltage and current for bonds behave mathematically the same as effort and flow. A bond graph then consists bonds and ports connected using ``$1$-junctions,"  and ``$0$-junctions."   Each bond is assigned both an effort and a flow variable, while the junctions dictate their relationships.  Here is an example of a bond graph where we use the convention that $E$ stands for effort and $F$ stands for flow: 


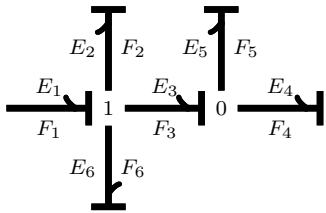
\begin{figure}[H] 
	\centering
\begin{tikzpicture}
[ 
	node distance=1.5cm,
	mynewelement/.style={
		color=blue!50!black!75,
		thick
	},
	mybondmodpoint/.style={
	rectangle,
	minimum size=3mm,
	very thick,
	draw=red!50!black!50, 
	outer sep=2pt
	}
]
	\node (S) {};
	\node (J11) [right of=S]{1}
	edge [inbonde, line width=2.5pt]  node [below]{$F_{1}$} (S)
        edge [inbonde, line width=2.5pt]  node [above]{$E_{1}$} (S);
	\node (R1) [above of=J11]{}
	edge  [inbonde, line width=2.5pt]  node [right]{$F_{2}$} (J11)
        edge  [inbonde, line width=2.5pt] node [left]{$E_{2}$} (J11);
	\node (J01) [right of=J11] {0}
	edge  [inbonde, line width=2.5pt] node  [color=black] [above]{$E_{3}$} (J11)
        edge  [inbonde, line width=2.5pt]  node  [color=black] [below]{$F_{3}$} (J11);
	\node (R2) [right of=J01] {}
	edge  [inbonde, line width=2.5pt] node [below]{$F_{4}$} (J01)
        edge  [inbonde, line width=2.5pt] node [above]{$E_{4}$}  (J01);
	\node (C1) [above of=J01] {}
	edge  [inbonde, line width=2.5pt]  node [right]{$F_{5}$} (J01)
	edge  [inbonde, line width=2.5pt]  node [left]{$E_{5}$} (J01);
	\node (C3) [below of=J11] {}
	edge  [inbonde, line width=2.5pt]  node [right]{$F_{6}$} (J11)
	edge  [inbonde, line width=2.5pt]  node [left]{$E_{6}$} (J11);
\end{tikzpicture}
\end{figure}

The arrow indicates which direction of current flow counts as positive, while the bar is called the `causal stroke'. These are unnecessary for our work, so we adopt a simplified notation without the arrow or bar. Additionally, one may also attach general circuit components, but we will also not consider these. The $1$-junction puts the associated bonds into a ``series" connection.  For $n\geq 3$ bonds this imposes the equations $$\sum_{i=1}^{n} E_i = 0$$ and $$F_1 = F_2=\cdots=F_n.$$ When there are $3$ bonds this junction is drawn as follows:

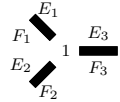
\begin{figure}[H] 
	\centering
\begin{tikzpicture}[circuit ee IEC, set resistor graphic=var resistor IEC
      graphic, scale=0.8, every node/.append style={transform shape}]
[
	node distance=1.5cm,
	mynewelement/.style={
		color=blue!50!black!75,
		thick
	},
	mybondmodpoint/.style={
	rectangle,
	minimum size=3mm,
	very thick,
	draw=red!50!black!50, 
	outer sep=2pt
	}
]		
	\node(J11) {$\mathrm{1}$};
	\node (R2) [ below left of=J11] {}
	edge  [line width=3.5pt]   node [right = .1, below]{$F_2$} (J11)
        edge  [line width=3.5pt]   node [above=.15, left]{$E_2$} (J11);
	\node (R1) [ above left of=J11] {}
	edge [line width=3.5pt]    node [below = .15, left ]{$F_1$} (J11)
        edge  [line width=3.5pt]   node [right=.1, above]{$E_1$} (J11);
	\node (C1) [right of=J11] {}
	edge [line width=3.5pt]    node [below]{$F_3$} (J11)
        edge [line width=3.5pt]    node [above]{$E_3$} (J11);
      \end{tikzpicture} 
\end{figure}

\vspace{-1ex}
\noindent Since there are $3$ bonds and thus $3$ ports this is sometimes called a $3$-port while a junction with $n$ bonds is sometimes called an $n$-port. We impose a notion of input and output for bond graphs where we think of effort and flow as going from left to right. Thus $(E_1,F_1)$ and $(E_2,F_2)$ are inputs, while $(E_3,F_3)$ is the output. Then the equations become $E_1+E_2=E_3$ and $F_1=F_2=F_3$. We then say that the $1$-junction adds efforts and ``coduplicates" flows. We may also turn this picture around to get another $1$-junction with one input and two outputs:

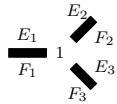
\begin{figure}[H] 
	\centering
\begin{tikzpicture}[circuit ee IEC, set resistor graphic=var resistor IEC
      graphic, scale=0.8, every node/.append style={transform shape}]
[
	node distance=1.5cm,
	mynewelement/.style={
		color=blue!50!black!75,
		thick
	},
	mybondmodpoint/.style={
	rectangle,
	minimum size=3mm,
	very thick,
	draw=red!50!black!50, 
	outer sep=2pt
	}
]		
	\node(J11) {$\mathrm{1}$};
	\node (R2) [ below right of=J11] {}
	edge  [line width=3.5pt]   node [ left = .1, below]{$F_3$} (J11)
        edge  [line width=3.5pt]   node [above=.15, right]{$E_3$} (J11);
	\node (R1) [ above right of=J11] {}
	edge [line width=3.5pt]    node [below=.15, right]{$F_2$} (J11)
        edge  [line width=3.5pt]   node [left=.1, above ]{$E_2$} (J11);
	\node (C1) [left of=J11] {}
	edge [line width=3.5pt]    node [below]{$F_1$} (J11)
        edge [line width=3.5pt]    node [above]{$E_1$} (J11);
      \end{tikzpicture} 
\end{figure}

\vspace{-1ex}

We say that this junction ``coadds" effort and duplicates flow since now $E_1=E_2+E_3$ and $F_1=F_2=F_3$. The other type of junction is the $0$-junction, which imposes the equations $$\sum_{i=1}^{n} F_i = 0$$ and $$E_1 = E_2=\cdots=E_n.$$ When there are $3$ bonds this junction is drawn as follows:
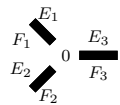
\begin{figure}[H] 
	\centering
\begin{tikzpicture}[circuit ee IEC, set resistor graphic=var resistor IEC
      graphic, scale=0.8, every node/.append style={transform shape}]
[
	node distance=1.5cm,
	mynewelement/.style={
		color=blue!50!black!75,
		thick
	},
	mybondmodpoint/.style={
	rectangle,
	minimum size=3mm,
	very thick,
	draw=red!50!black!50, 
	outer sep=2pt
	}
]		
	\node(J11) {$\mathrm{0}$};
	\node (R2) [ below left of=J11] {}
	edge  [line width=3.5pt]   node [ right = .1, below]{$F_2$} (J11)
        edge  [line width=3.5pt]   node [above=.15, left]{$E_2$} (J11);
	\node (R1) [ above left of=J11] {}
	edge [line width=3.5pt]    node [below = .15, left ]{$F_1$} (J11)
        edge  [line width=3.5pt]   node [right = .1, above]{$E_1$} (J11);
	\node (C1) [right of=J11] {}
	edge [line width=3.5pt]    node [below]{$F_3$} (J11)
        edge [line width=3.5pt]    node [above]{$E_3$} (J11);
      \end{tikzpicture} 
\end{figure}

\vspace{-1ex}

Now $F_1+F_2=F_3$ and $E_1=E_2=E_3$ so we say that this junction adds flow and coduplicates effort. Similarly, the reflected $0$-junction coadds flow and duplicates effort. This is drawn as:
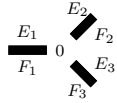
\begin{figure}[H] 
	\centering
\begin{tikzpicture}[circuit ee IEC, set resistor graphic=var resistor IEC
      graphic, scale=0.8, every node/.append style={transform shape}]
[
	node distance=1.5cm,
	mynewelement/.style={
		color=blue!50!black!75,
		thick
	},
	mybondmodpoint/.style={
	rectangle,
	minimum size=3mm,
	very thick,
	draw=red!50!black!50, 
	outer sep=2pt
	}
]		
	\node(J11) {$\mathrm{0}$};
	\node (R2) [ below right of=J11] {}
	edge  [line width=3.5pt]   node [left = .1, below]{$F_3$} (J11)
        edge  [line width=3.5pt]   node [above=.15, right]{$E_3$} (J11);
	\node (R1) [ above right of=J11] {}
	edge [line width=3.5pt]    node [below=.15, right]{$F_2$} (J11)
        edge  [line width=3.5pt]   node [ left=.1, above]{$E_2$} (J11);
	\node (C1) [left of=J11] {}
	edge [line width=3.5pt]    node [below]{$F_1$} (J11)
        edge [line width=3.5pt]    node [above]{$E_1$} (J11);
      \end{tikzpicture} 
\end{figure}

\vspace{-1ex}

Previous work on electrical circuits \cite{BC,BF} allows us to view a diagram of wires between terminals as a morphism between finite sets in some category. A single wire between two terminals acts as the identity for the one element set. By representing a bond as a pair of wires between four terminals we can think of a bond as the identity for the two element set. We drop the labels and draw this as follows:
\[
  \xymatrixrowsep{75pt}
  \xymatrixcolsep{8pt}
  \xymatrix@1{
 *+[u]{
\begin{tikzpicture}[circuit ee IEC, set resistor graphic=var resistor IEC
      graphic,scale=0.8, every node/.append style={transform shape}]
[
	node distance=1.5cm,
	mynewelement/.style={
		color=blue!50!black!75,
		thick
	},
	mybondmodpoint/.style={
	rectangle,
	minimum size=3mm,
	very thick,
	draw=red!50!black!50, 
	outer sep=2pt
	}
]		
	\node(J11) {};
	\node (R3) [right=1 of J11] {}
	edge [line width=3.5pt]    node [below]{} (J11)
        edge  [line width=3.5pt]  node [above]{} (J11);
\end{tikzpicture}
}
&  := &
\identitytwo{.1\textwidth}
}
\]

\noindent Then a bond graph is a diagram of wires with pairs of wires going between pairs of terminals, i.e.\ ports, and we can think of a bond graph as a morphism between copies of the two element set. To do so for any bond graph we need to rewrite the junctions in terms of pairs of wires. The $1$-junction connects three pairs of wires in the following way:
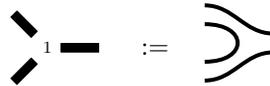
\[
  \xymatrixrowsep{5pt}
  \xymatrixcolsep{8pt}
  \xymatrix@1{
 *+[u]{
\begin{tikzpicture}[circuit ee IEC, set resistor graphic=var resistor IEC
      graphic, scale=0.8, every node/.append style={transform shape}]
[
	node distance=1.5cm,
	mynewelement/.style={
		color=blue!50!black!75,
		thick
	},
	mybondmodpoint/.style={
	rectangle,
	minimum size=3mm,
	very thick,
	draw=red!50!black!50, 
	outer sep=2pt
	}
]		
	\node(J11) {$\mathrm{1}$};
	\node (R2) [ below left of=J11] {}
	edge  [line width=3.5pt]   node [below]{} (J11)
        edge  [line width=3.5pt]   node [above]{} (J11);
	\node (R1) [ above left of=J11] {}
	edge [line width=3.5pt]    node [below]{} (J11)
        edge  [line width=3.5pt]   node [above]{} (J11);
	\node (C1) [right of=J11] {}
	edge [line width=3.5pt]    node [right]{} (J11)
        edge [line width=3.5pt]    node [left]{} (J11);
      \end{tikzpicture} 
}
 & := &  \monadmult{.1\textwidth}
  }
\]

\noindent If we flip the wire diagram we get the other $1$-junction. The $0$-junction is drawn in terms of pairs of wires as:
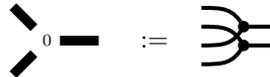
\[
  \xymatrixrowsep{5pt}
  \xymatrixcolsep{8pt}
  \xymatrix@1{
 *+[u]{
\begin{tikzpicture}[circuit ee IEC, set resistor graphic=var resistor IEC
      graphic,scale=0.8, every node/.append style={transform shape}]
[
	node distance=1.5cm,
	mynewelement/.style={
		color=blue!50!black!75,
		thick
	},
	mybondmodpoint/.style={
	rectangle,
	minimum size=3mm,
	very thick,
	draw=red!50!black!50, 
	outer sep=2pt
	}
]		
	\node(J11) {$\mathrm{0}$};
	\node (R2) [ below left of=J11] {}
	edge [line width=3.5pt]    node [below]{} (J11)
        edge   [line width=3.5pt]  node [above]{} (J11);
	\node (R1) [ above left of=J11] {}
	edge [line width=3.5pt]   node [below]{} (J11)
        edge  [line width=3.5pt]   node [above]{} (J11);
	\node (C1) [right of=J11] {}
	edge [line width=3.5pt]    node [right]{} (J11)
        edge  [line width=3.5pt]   node [left]{} (J11);
      \end{tikzpicture} 
}
 &  := & \parmult{.1\textwidth}
  }
\]

\vspace{-1ex}

\noindent We may turn the picture around as well to get another $0$-junction. These four junctions can be stuck together in various ways to make larger bond graphs. There are also rules for simplifying bond graphs that are derived from the equations above. However, some bond graphs are not permitted even though one could imagine drawing them. This is typically because the result is trivial or does not have a useful physical interpretation in electrical engineering. One simple example is the following:
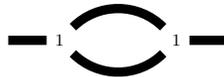
\begin{figure}[H]
\centering
\begin{tikzpicture}[circuit ee IEC, set resistor graphic=var resistor IEC
      graphic, scale=0.8, every node/.append style={transform shape}]
[
	node distance=1.5cm,
	mynewelement/.style={
		color=blue!50!black!75,
		thick
	},
	mybondmodpoint/.style={
	rectangle,
	minimum size=3mm,
	very thick,
	draw=red!50!black!50, 
	outer sep=2pt
	}
]		
	\node (J11) {1};
	\node (C1) [right of=J11] {}
        edge [line width=3.5pt]   node [left]{} (J11);
	\node (J12) [ left =1.5 of J11] {1}
    edge [line width=3.5pt, in =225, out=-45, looseness=1]   node [left]{} (J11)
    edge [line width=3.5pt, in =135, out=45, looseness=1]   node [left]{} (J11);
	\node (J13) [left of=J12] {}
        edge [line width=3.5pt]   node [left]{} (J12);
      \end{tikzpicture} 
\end{figure}
\vspace{-2ex}

Viewing this as the composite of some  $2$ output morphism and some $2$ input morphism presents a mild roadblock.  Since in a category we cannot arbitrarily stop compatible morphisms from composing, such a bond graph must be given meaning. Also, there are no junctions with less than $3$ total inputs and outputs, but with some experience using bond graphs one can imagine constructing such junctions. In fact, the above description of bond graphs using wires only works in the ``usual" cases. Thus we need to do better. To overcome these obstacles we construct a category $\BondGraph$ such that: 

\begin{enumerate}
\item $\BondGraph$ has generating morphisms corresponding to $1$-junctions, $0$-junctions, and unary versions of them.
\item The morphisms in $\BondGraph$ obey relations that can be derived from the equations governing the junctions.
\item There is a functor assigning to any morphism in $\BondGraph$ a vector space of efforts and flows consistent with equations governing junctions.
\item There is a functor assigning to any morphism in $\BondGraph$ a vector space of potentials and currents consistent with the laws governing potential and current along wires. 
\end{enumerate}

\noindent Finally, there is a natural transformation between the two functors which comes from the relationship between effort, flow, potential, and current given by the equations $V=\phi_2-\phi_1$ and $I=I_1=-I_2$.

\subsubsection*{Plan of the paper}

In Section \ref{sec:corelations} we begin by describing the category, $\Fin\Corel$, which has finite sets as objects and ``corelations" as morphisms. To summarize the work of  Baez, Fong, and the author \cite{BC, BF}, there is a correspondence between morphisms in $\Fin\Corel$ and electrical circuits made of only wire. We thus take $\Fin\Corel$ to be the category whose morphisms correspond to electrical circuits made of only wire.  The one element set, which we call $1\in \Fin\Corel$, plays a fundamental role: in terms of electrical engineering it is called a ``terminal" because it is the end of a wire. Then the morphisms in $\Fin\Corel$ correspond to various ways in which terminals can be connected.  We also discuss the notion of an ``extraspecial commutative Frobenius monoid." The first example of such a monoid is the object $1$ together with some morphisms in $\Fin\Corel.$ Many other examples of this type of monoid are examined in later sections. 
 
We need to understand the object $1$ since it represents the end of a wire and we want to understand bonds through their relationship to wires. In this analogy the end of a bond, or ``port,"  is the end of a pair of wires with opposite current, i.e., a pair of terminals with opposite current. Hence, bond graphs  correspond to specific ways of connecting ports just as electrical circuits correspond to specific ways of connecting terminals. Since $1$ is a terminal we think of the object $2$ as a port. In this approach bond graphs are morphisms in a subcategory of $\Fin\Corel$ with objects $2m$ and some morphisms in $\textrm{hom}(2m,2n)$ corresponding to ways in which ports can be connected. In Sections   \ref{sec:series}, \ref{sec:parallel}, and \ref{sec:bimonoids} we prove results that help us describe this subcategory. 

In Sections \ref{sec:series} and \ref{sec:parallel} we study two Frobenius monoid structures that can be attached to the object $2$. The first one is actually a monad that comes from an adjunction in the symmetric monoidal category $\Fin\Corel$. We show that the multiplication and comultiplication associated to the Frobenius monoid in Theorem \ref{thm:series} are closely related to $1$-junctions. Due to this we consider the unit and counit associated to this Frobenius monoid as corresponding to unary versions of $1$-junctions. 

 The second Frobenius monoid we consider, $1+1$, is constructed as the coproduct of the monoid $1$ with itself. The properties exhibited by the multiplication and comultiplication associated to this Frobenius monoid shown in Theorem \ref{thm:parallel} correspond to properties exhibited by $0$-junctions. Now we use the unit and counit associated to this Frobenius monoid as unary versions of $0$-junctions. In summary, both Frobenius monoids are extraspecial, but the first is also symmetric, while the second is commutative. This slight difference between the two Frobenius monoids is the first drawback of this approach.

In Section \ref{sec:bimonoids} we describe how $1$-junctions and $0$-junctions interact in terms of these two Frobenius monoids.  Surprisingly, the morphisms which are generated by these junctions are closely related to ``weak bimonoids," introduced by Pastro and Street in their work on quantum categories \cite{PS}. In Theorem \ref{thm:weakbimonoid}, we prove that the morphisms associated with our Frobenius monoids come together as a pair of weak bimonoids with a few additional properties. The weak bimonoid axioms are precisely the ones obeyed by the $1$- and $0$ junctions and their unary versions. However, the other flaw in this approach is that the additional properties do not match properties of bond graphs. We define $\Fin\Corel^{\circ}$ to be the subcategory of $\Fin\Corel$ whose morphisms can be attained from the morphisms making up the two Frobenius monoids. 

Section \ref{sec:props} introduces the framework of props, which are strict symmetric monoidal categories whose objects are the natural numbers where the tensor product is addition. This background allows us to present examples of these types of categories in terms of generators and relations.  In Conjecture \ref{con:presentation} we conjecture that Theorems \ref{thm:series}, \ref{thm:parallel}, and \ref{thm:weakbimonoid} give enough relations to present $\Fin\Corel^{\circ}$. Another benefit of the prop framework is that we can easily describe functors out of a prop by defining them on generators and checking relations.

In Section \ref{sec:Lagrangian} we look at the category $\Lag\Rel_k$, whose morphisms are ``Lagrangian relations" between symplectic vector spaces, in another attempt at constructing a category where the morphisms are bond graphs. This is because a bond graph determines a Lagrangian subspace of possible efforts and flows at the ports. There is a pair of extraspecial Frobenius commutative monoids in $\Lag\Rel_k$. One of the monoids has morphisms that behave like $1$-junctions, while the other has morphisms  that behave like $0$-junctions. We again use this to define unary junctions corresponding to the units and counits.  However, a problem arises when the morphisms from from one monoid interact with the morphisms from the other. Instead of two weak bimonoids they define two bimonoids. The axioms of a bimonoid are too strong. Namely, the interaction of the unit and comultiplication causes a problem. The same issue occurs with the counit and the multiplication. We call the subcategory generated by these two Frobenius monoids $\Lag\Rel_k^{\circ}$.

In Section \ref{sec:functors} we define the category $\BondGraph$ as a prop using characteristics of both $\Fin\Corel^{\circ}$ and $\Lag\Rel_k^{\circ}$. We then look at the functorial semantics of $\BondGraph$. The functor $K \maps \Fin\Corel \to \Lag\Rel_k$, which was introduced by Baez and Fong \cite{BF} and further characterized by Baez and the author \cite{BC},  picks out a ``Lagrangian subspace" of possible potentials and currents at the terminals of each electrical circuit.  In Proposition \ref{prop:blackbox} we show that the functor $Ki\maps \Fin\Corel^{\circ} \to \Lag\Rel_k$ can be defined on the generators of $\Fin\Corel^{\circ}$, where $i\maps \Fin\Corel^{\circ} \to \Fin\Corel$ is the inclusion functor.  Then in Proposition \ref{prop:functors} we define a functor $G\maps \BondGraph \to \Fin\Corel^{\circ}$ which gives a semantics for bond graphs by composing with $Ki\maps \Fin\Corel^{\circ} \to \Lag\Rel_k$. Essentially, the  functor $KiG$ picks out a Lagrangian subspace for a bond graph by assigning to each port two pairs of potential and current. 

We also define another functor $F \maps \BondGraph \to \Lag\Rel_k^{\circ}$  in  Proposition \ref{prop:functors}. Once again we compose with an inclusion functor $i'\maps \Lag\Rel_k^{\circ}\to \Lag\Rel_k$, which results in another semantics for bond graphs. The functor $i'F\maps \BondGraph \to \Lag\Rel_k$ assigns a space of possible efforts and flows to each bond graph by assigning to each port an effort and flow, subject to some laws. The two semantics are connected by the way that potential and current are related to effort and flow. This relationship corresponds to a natural transformation which we prove exists in in Theorem \ref{thm:natural}. This gives us the following diagram:


\[
\begin{tikzcd}[column sep=scriptsize]
& \Lag\Rel_k^{\circ}  \arrow[dd,Rightarrow, ""{ below}, "\alpha",xshift=5ex, shorten <= .75em, shorten >= .75em]{}  \arrow[r, ""{ below}, "i'"]{}
 & \Lag\Rel \\
\BondGraph  \arrow[ur, ""{ below}, "F"]{}  \arrow[dr, ""{above}, "G" ' ]{} \\
&\Fin\Corel^{\circ}  \arrow[r, ""{below}, "i"]{} & \Fin\Corel \arrow[uu, "K"'] 
\end{tikzcd}
\]

\subsubsection*{Acknowledgments}  

Most of all I would like to thank my adviser John Baez. His thoughts and ideas have significantly contributed to not only this paper, but also my overall success. Our conversations continue to greatly improve my understanding of category theory, my writing, and my professional growth. I thank Brendan Fong, Blake Pollard, Jason Erbele, and Nick Woods for their suggestions, helpfulness, and our many conversations.

\section{Perfectly Conductive Wires and Corelations}
\label{sec:corelations}

We summarize the work that is necessary to view circuits made of perfectly conductive wire as morphisms in a category.  This background, together with the correspondence between bond graphs and certain types of circuits, provides us with a clear approach for studying bond graphs as morphisms in a category. Baez and Fong \cite{BC, BF} have gone into significant detail about the relationship between circuits, including those with passive linear components, and categories. To begin, we think of circuits as graphs equipped with a pair of functions.

\begin{definition}
A \define{graph} is a finite set $E$ of \define{edges} and a finite set $N$ of \define{nodes} equipped with a pair of functions $s,t \maps E \to N$ assigning to 
each edge its \define{source} and \define{target}.  We say that $e \in E$ is an edge \define{from} $s(e)$ \define{to} $t(e)$.  
\end{definition}

\begin{definition}
Given finite sets $X$ and $Y$, a \define{circuit from $X$ to $Y$} is a cospan of finite sets:
\[
  \xymatrixrowsep{15pt}
  \xymatrixcolsep{5pt}
  \xymatrix{
    & N \\
    X \ar[ur]^-{i} && Y \ar[ul]_-{o}
  }
\]
together with a graph $G$ having $N$ as its set of vertices:
\[ \xymatrix{E \ar@<2.5pt>[r]^{s} \ar@<-2.5pt>[r]_{t} & N.} \] 
We call the sets $i(X)$, $o(Y)$, and $\partial N = i(X) \cup o(Y)$ the \define{inputs},  \define{outputs}, and \define{terminals} of the circuit, respectively.
\end{definition}

\noindent Here is an example of such a circuit:

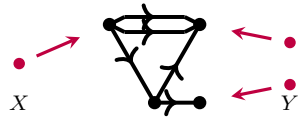
\begin{center}
    \begin{tikzpicture}[circuit ee IEC, set resistor graphic=var resistor IEC graphic, scale=.4]
\scalebox{1}{
      {\node[circle,draw,inner sep=1pt,fill=gray,color=purple]         (x) at
	(-3,-1.3) {};
	\node at (-3,-2.6) {$X$};}
      \node[contact]         (A) at (0,0) {};
      \node[contact]         (B) at (3,0) {};
      \node[contact]         (C) at (1.5,-2.6) {};
      \node[contact]         (D) at (3,-2.6) {};
      {\node[circle,draw,inner sep=1pt,fill=gray,color=purple]         (y1) at
	(6,-.6) {};
	  \node[circle,draw,inner sep=1pt,fill=gray,color=purple]         (y2) at
	  (6,-2) {};
	  \node at (6,-2.6) {$Y$};}
      \coordinate         (ua) at (.5,.25) {};
      \coordinate         (ub) at (2.5,.25) {};
      \coordinate         (la) at (.5,-.25) {};
      \coordinate         (lb) at (2.5,-.25) {};
      \path (A) edge (ua);
      \path (A) edge (la);
      \path (B) edge (ub);
      \path (B) edge (lb);
      \path (ua) edge  [->-=.5] node[label={[label distance=1pt]90:{}}] {} (ub);
      \path (la) edge [->-=.5]  node[label={[label distance=1pt]270:{}}] {} (lb);
      \path (A) edge [->-=.5]  node[label={[label distance=2pt]180:{}}] {} (C);
      \path (C) edge  [->-=.5]  node[label={[label distance=2pt]0:{}}] {} (B);
      \path (C) edge  [->-=.5]  node[label={[label distance=2pt]0:{}}] {} (D);
      {
	\path[color=purple, very thick, shorten >=10pt, shorten <=5pt, ->, >=stealth] (x) edge (A);
	\path[color=purple, very thick, shorten >=10pt, shorten <=5pt, ->, >=stealth] (y1) edge (B);
	\path[color=purple, very thick, shorten >=10pt, shorten <=5pt, ->, >=stealth] (y2)  edge (D);}
}
    \end{tikzpicture}
  \end{center}

Given two circuits, one from $X$ to $Y$, and the other from $Y$ to $Z$, we define their composite in the following way. An output of one circuit is ``glued" to an input of another circuit if they are the image of the same element in $Y$.  The precise details involve composing isomorphisms classes of cospans using pushouts, which were explained by Baez and the author  \cite{BC}. However, it suffices to understand an example.
These two circuits:
 
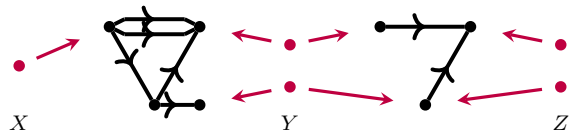
\begin{center}
    \begin{tikzpicture}[circuit ee IEC, set resistor graphic=var resistor IEC
      graphic,scale=.4]
      \node[circle,draw,inner sep=1pt,fill=gray,color=purple]         (x) at
      (-3,-1.3) {};
      \node at (-3,-3.2) {\footnotesize $X$};
      \node[circle,draw,inner sep=1pt,fill]         (A) at (0,0) {};
      \node[circle,draw,inner sep=1pt,fill]         (B) at (3,0) {};
      \node[circle,draw,inner sep=1pt,fill]         (C) at (1.5,-2.6) {};
      \node[circle,draw,inner sep=1pt,fill]         (D) at (3,-2.6) {};
      \node[circle,draw,inner sep=1pt,fill=gray,color=purple]         (y1) at
      (6,-.6) {};
      \node[circle,draw,inner sep=1pt,fill=gray,color=purple]         (y2) at
      (6,-2) {};
      \node at (6,-3.2) {\footnotesize $Y$};
      \coordinate         (ua) at (.5,.25) {};
      \coordinate         (ub) at (2.5,.25) {};
      \coordinate         (la) at (.5,-.25) {};
      \coordinate         (lb) at (2.5,-.25) {};
      \path (A) edge (ua);
      \path (A) edge (la);
      \path (B) edge (ub);
      \path (B) edge (lb);
      \path (ua) edge [->-=.5]  node[label={[label distance=1pt]90:{}}] {} (ub);
      \path (la) edge [->-=.5]   node[label={[label distance=1pt]270:{}}] {} (lb);
      \path (A) edge   [->-=.5]  node[label={[label distance=-1pt]180:{}}] {} (C);
      \path (C) edge  [->-=.5]   node[label={[label distance=-1pt]0:{}}] {} (B);
      \path (C) edge  [->-=.5]   node[label={[label distance=-1pt]0:{}}] {} (D);
      \path[color=purple, very thick, shorten >=10pt, shorten <=5pt, ->, >=stealth] (x) edge (A);
      \path[color=purple, very thick, shorten >=10pt, shorten <=5pt, ->, >=stealth] (y1) edge (B);
      \path[color=purple, very thick, shorten >=10pt, shorten <=5pt, ->, >=stealth] (y2)
      edge (D);

      \node[circle,draw,inner sep=1pt,fill]         (A') at (9,0) {};
      \node[circle,draw,inner sep=1pt,fill]         (B') at (12,0) {};
      \node[circle,draw,inner sep=1pt,fill]         (C') at (10.5,-2.6) {};
      \node[circle,draw,inner sep=1pt,fill=gray,color=purple]         (z1) at
      (15,-.6) {};
      \node[circle,draw,inner sep=1pt,fill=gray,color=purple]         (z2) at (15,-2) {};
      \node at (15,-3.2) {\footnotesize $Z$};
      \path (A') edge [->-=.5]  node[above] {} (B');
      \path (C') edge [->-=.5]  node[right] {} (B');
      \path[color=purple, very thick, shorten >=10pt, shorten <=5pt, ->, >=stealth] (y1) edge (A');
      \path[color=purple, very thick, shorten >=10pt, shorten <=5pt, ->, >=stealth] (y2)
      edge (C');
      \path[color=purple, very thick, shorten >=10pt, shorten <=5pt, ->, >=stealth] (z1) edge (B');
      \path[color=purple, very thick, shorten >=10pt, shorten <=5pt, ->, >=stealth]
      (z2) edge (C');
    \end{tikzpicture}
\end{center}

\noindent
are composed by gluing to obtain this one: 

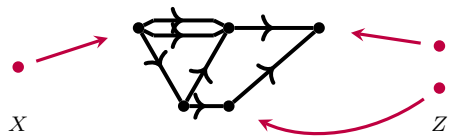
\begin{center}
    \begin{tikzpicture}[circuit ee IEC, set resistor graphic=var resistor IEC
      graphic,scale=0.4]
      \node[circle,draw,inner sep=1pt,fill=gray,color=purple]         (x) at (-4,-1.3) {};
      \node at (-4,-3.2) {\footnotesize $X$};
      \node[circle,draw,inner sep=1pt,fill]         (A) at (0,0) {};
      \node[circle,draw,inner sep=1pt,fill]         (B) at (3,0) {};
      \node[circle,draw,inner sep=1pt,fill]         (C) at (1.5,-2.6) {};
      \node[circle,draw,inner sep=1pt,fill]         (E) at (3,-2.6) {};
      \node[circle,draw,inner sep=1pt,fill]         (D) at (6,0) {};
      \coordinate         (ua) at (.5,.25) {};
      \coordinate         (ub) at (2.5,.25) {};
      \coordinate         (la) at (.5,-.25) {};
      \coordinate         (lb) at (2.5,-.25) {};
      \coordinate         (ub2) at (3.5,.25) {};
      \coordinate         (ud) at (5.5,.25) {};
      \coordinate         (lb2) at (3.5,-.25) {};
      \path (A) edge (ua);
      \path (A) edge (la);
      \path (B) edge (ub);
      \path (B) edge (lb);
      \node[circle,draw,inner sep=1pt,fill=gray,color=purple]         (z1) at
      (10,-.6) {};
      \node[circle,draw,inner sep=1pt,fill=gray,color=purple]         (z2) at (10,-2) {};
      \node at (10,-3.2) {\footnotesize $Z$};
      \path (ua) edge [->-=.5]  node[above] {} (ub);
      \path (la) edge [->-=.5] node[below] {} (lb);
      \path (A) edge [->-=.5]  node[left] {} (C);
      \path (C) edge [->-=.5]  node[right] {} (B);
      \path (C) edge [->-=.5]  node[right] {} (E);
      \path (E) edge [->-=.5]  node[right] {} (D);
      \path (B) edge [->-=.5]  node[right] {} (D); 
      \path[color=purple, very thick, shorten >=10pt, shorten <=5pt, ->, >=stealth] (x) edge (A);
      \path[color=purple, very thick, shorten >=10pt, shorten <=5pt, ->, >=stealth] (z1)   edge (D);
      \path[bend left, color=purple, very thick, shorten >=10pt, shorten <=5pt, ->, >=stealth] (z2) edge (E);
    \end{tikzpicture}
  \end{center}

Using this composition we have a category, which we call $\Circ$. This category can be equipped with a monoidal structure defined by stacking one graph above another graph to formally turn two graphs with inputs and outputs into a single one. Once again we skip the precise details. These two circuits:

 
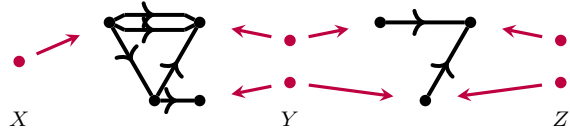
\begin{center}
    \begin{tikzpicture}[circuit ee IEC, set resistor graphic=var resistor IEC
      graphic,scale=.4]
      \node[circle,draw,inner sep=1pt,fill=gray,color=purple]         (x) at
      (-3,-1.3) {};
      \node at (-3,-3.2) {\footnotesize $X$};
      \node[circle,draw,inner sep=1pt,fill]         (A) at (0,0) {};
      \node[circle,draw,inner sep=1pt,fill]         (B) at (3,0) {};
      \node[circle,draw,inner sep=1pt,fill]         (C) at (1.5,-2.6) {};
      \node[circle,draw,inner sep=1pt,fill]         (D) at (3,-2.6) {};
      \node[circle,draw,inner sep=1pt,fill=gray,color=purple]         (y1) at
      (6,-.6) {};
      \node[circle,draw,inner sep=1pt,fill=gray,color=purple]         (y2) at
      (6,-2) {};
      \node at (6,-3.2) {\footnotesize $Y$};
      \coordinate         (ua) at (.5,.25) {};
      \coordinate         (ub) at (2.5,.25) {};
      \coordinate         (la) at (.5,-.25) {};
      \coordinate         (lb) at (2.5,-.25) {};
      \path (A) edge (ua);
      \path (A) edge (la);
      \path (B) edge (ub);
      \path (B) edge (lb);
      \path (ua) edge [->-=.5]  node[label={[label distance=1pt]90:{}}] {} (ub);
      \path (la) edge [->-=.5]   node[label={[label distance=1pt]270:{}}] {} (lb);
      \path (A) edge   [->-=.5]  node[label={[label distance=-1pt]180:{}}] {} (C);
      \path (C) edge  [->-=.5]   node[label={[label distance=-1pt]0:{}}] {} (B);
      \path (C) edge  [->-=.5]   node[label={[label distance=-1pt]0:{}}] {} (D);
      \path[color=purple, very thick, shorten >=10pt, shorten <=5pt, ->, >=stealth] (x) edge (A);
      \path[color=purple, very thick, shorten >=10pt, shorten <=5pt, ->, >=stealth] (y1) edge (B);
      \path[color=purple, very thick, shorten >=10pt, shorten <=5pt, ->, >=stealth] (y2)
      edge (D);

      \node[circle,draw,inner sep=1pt,fill]         (A') at (9,0) {};
      \node[circle,draw,inner sep=1pt,fill]         (B') at (12,0) {};
      \node[circle,draw,inner sep=1pt,fill]         (C') at (10.5,-2.6) {};
      \node[circle,draw,inner sep=1pt,fill=gray,color=purple]         (z1) at
      (15,-.6) {};
      \node[circle,draw,inner sep=1pt,fill=gray,color=purple]         (z2) at (15,-2) {};
      \node at (15,-3.2) {\footnotesize $Z$};
      \path (A') edge [->-=.5]  node[above] {} (B');
      \path (C') edge [->-=.5]  node[right] {} (B');
      \path[color=purple, very thick, shorten >=10pt, shorten <=5pt, ->, >=stealth] (y1) edge (A');
      \path[color=purple, very thick, shorten >=10pt, shorten <=5pt, ->, >=stealth] (y2)
      edge (C');
      \path[color=purple, very thick, shorten >=10pt, shorten <=5pt, ->, >=stealth] (z1) edge (B');
      \path[color=purple, very thick, shorten >=10pt, shorten <=5pt, ->, >=stealth]
      (z2) edge (C');
    \end{tikzpicture}
\end{center}
\noindent are tensored by stacking the first above the second, giving us a circuit from $X+Y$ to $Y+Z$:
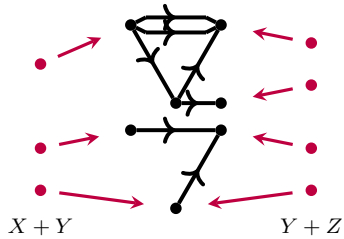
 \begin{center}
    \begin{tikzpicture}[circuit ee IEC, set resistor graphic=var resistor IEC
      graphic,scale=.4]
      \node[circle,draw,inner sep=1pt,fill=gray,color=purple]         (x) at
      (-3,-1.3) {};
      \node at (-3,-3.2) {\footnotesize $$};
      \node[circle,draw,inner sep=1pt,fill]         (A) at (0,0) {};
      \node[circle,draw,inner sep=1pt,fill]         (B) at (3,0) {};
      \node[circle,draw,inner sep=1pt,fill]         (C) at (1.5,-2.6) {};
      \node[circle,draw,inner sep=1pt,fill]         (D) at (3,-2.6) {};
      \node[circle,draw,inner sep=1pt,fill=gray,color=purple]         (y1) at
      (6,-.6) {};
      \node[circle,draw,inner sep=1pt,fill=gray,color=purple]         (y2) at
      (6,-2) {};
      \node at (6,-3.2) {\footnotesize $$};
      \coordinate         (ua) at (.5,.25) {};
      \coordinate         (ub) at (2.5,.25) {};
      \coordinate         (la) at (.5,-.25) {};
      \coordinate         (lb) at (2.5,-.25) {};
      \path (A) edge (ua);
      \path (A) edge (la);
      \path (B) edge (ub);
      \path (B) edge (lb);
      \path (ua) edge [->-=.5]  node[label={[label distance=1pt]90:{}}] {} (ub);
      \path (la) edge [->-=.5]   node[label={[label distance=1pt]270:{}}] {} (lb);
      \path (A) edge   [->-=.5]  node[label={[label distance=-1pt]180:{}}] {} (C);
      \path (C) edge  [->-=.5]   node[label={[label distance=-1pt]0:{}}] {} (B);
      \path (C) edge  [->-=.5]   node[label={[label distance=-1pt]0:{}}] {} (D);
      \path[color=purple, very thick, shorten >=10pt, shorten <=5pt, ->, >=stealth] (x) edge (A);
      \path[color=purple, very thick, shorten >=10pt, shorten <=5pt, ->, >=stealth] (y1) edge (B);
      \path[color=purple, very thick, shorten >=10pt, shorten <=5pt, ->, >=stealth] (y2)
      edge (D);

      \node[circle,draw,inner sep=1pt,fill]         (A') at (0,-3.5) {};
      \node[circle,draw,inner sep=1pt,fill]         (B') at (3,-3.5) {};
      \node[circle,draw,inner sep=1pt,fill]         (C') at (1.5,-6.1) {};
      \node[circle,draw,inner sep=1pt,fill=gray,color=purple]         (z1) at
      (6,-4.1) {};
      \node[circle,draw,inner sep=1pt,fill=gray,color=purple]         (z2) at (6,-5.5) {};
      \node at (-3,-6.7) {\footnotesize $X+Y$};
      \node[circle,draw,inner sep=1pt,fill=gray,color=purple]         (y1') at
      (-3,-4.1) {};
      \node[circle,draw,inner sep=1pt,fill=gray,color=purple]         (y2') at
      (-3,-5.5) {};
      \node at (6,-6.7) {\footnotesize $Y+Z$};
      \path (A') edge [->-=.5]  node[above] {} (B');
      \path (C') edge [->-=.5]  node[right] {} (B');
      \path[color=purple, very thick, shorten >=10pt, shorten <=5pt, ->, >=stealth] (y1') edge (A');
      \path[color=purple, very thick, shorten >=10pt, shorten <=5pt, ->, >=stealth] (y2')
      edge (C');
      \path[color=purple, very thick, shorten >=10pt, shorten <=5pt, ->, >=stealth] (z1) edge (B');
      \path[color=purple, very thick, shorten >=10pt, shorten <=5pt, ->, >=stealth]
      (z2) edge (C');
    \end{tikzpicture}
\end{center}

\noindent This results in a symmetric monoidal category.

\begin{proposition}
There is a symmetric monoidal category $\define{\Circ}$ where the objects are finite sets, the morphisms are isomorphism classes of circuits,  composition is gluing inputs to outputs, and the tensor product is stacking circuits.
\end{proposition}

We shall see that to any circuit we can associate a ``behavior" \cite{BF}. Roughly speaking, this consists of taking a circuit, and assigning to it a vector space tied to the inputs and outputs of the graph. Since current may only flow between two nodes which are connected by a path in a circuit, all that matter for any two nodes is whether or not there is a path between them.  That is, the defining property for a circuit made of perfectly conductive wire is the connectivity of the nodes. 

Thus  the behavior of any circuit in $\Circ$ is determined by the connectivity between input and output nodes of the underlying graph, but not how the nodes are connected. We extract the relevant property from a circuit in two steps. First replace $G$ with its set of connected components $\pi_0(G)$, which also defines a map $p_G\maps N\to \pi_0(G)$ that sends each node to the connected component that it lies in. This sends a circuit to a cospan of finite sets: 
\[
  \xymatrixrowsep{15pt}
  \xymatrixcolsep{5pt}
  \xymatrix{
    & \pi_0(G) \\
    X \ar[ur]^-{p_G i} && Y \ar[ul]_-{p_G o}
  }
\]
\noindent If we apply this process to an example, then the circuit:
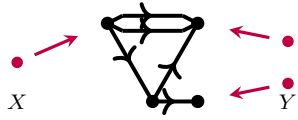
\begin{center}
    \begin{tikzpicture}[circuit ee IEC, set resistor graphic=var resistor IEC graphic,scale=.4]
\scalebox{1}{
      {\node[circle,draw,inner sep=1pt,fill=gray,color=purple]         (x) at
	(-3,-1.3) {};
	\node at (-3,-2.6) {$X$};}
      \node[contact]         (A) at (0,0) {};
      \node[contact]         (B) at (3,0) {};
      \node[contact]         (C) at (1.5,-2.6) {};
      \node[contact]         (D) at (3,-2.6) {};
      {\node[circle,draw,inner sep=1pt,fill=gray,color=purple]         (y1) at
	(6,-.6) {};
	  \node[circle,draw,inner sep=1pt,fill=gray,color=purple]         (y2) at
	  (6,-2) {};
	  \node at (6,-2.6) {$Y$};}
      \coordinate         (ua) at (.5,.25) {};
      \coordinate         (ub) at (2.5,.25) {};
      \coordinate         (la) at (.5,-.25) {};
      \coordinate         (lb) at (2.5,-.25) {};
      \path (A) edge (ua);
      \path (A) edge (la);
      \path (B) edge (ub);
      \path (B) edge (lb);
      \path (ua) edge  [->-=.5] node[label={[label distance=1pt]90:{}}] {} (ub);
      \path (la) edge [->-=.5]  node[label={[label distance=1pt]270:{}}] {} (lb);
      \path (A) edge [->-=.5]  node[label={[label distance=2pt]180:{}}] {} (C);
      \path (C) edge  [->-=.5]  node[label={[label distance=2pt]0:{}}] {} (B);
      \path (C) edge  [->-=.5]  node[label={[label distance=2pt]0:{}}] {} (D);
      {
	\path[color=purple, very thick, shorten >=10pt, shorten <=5pt, ->, >=stealth] (x) edge (A);
	\path[color=purple, very thick, shorten >=10pt, shorten <=5pt, ->, >=stealth] (y1) edge (B);
	\path[color=purple, very thick, shorten >=10pt, shorten <=5pt, ->, >=stealth] (y2)  edge (D);}
}
    \end{tikzpicture}
  \end{center}

\noindent becomes the cospan:

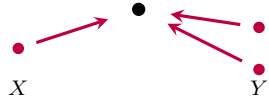
\begin{center}
    \begin{tikzpicture}[circuit ee IEC, set resistor graphic=var resistor IEC graphic,scale=.4]
\scalebox{1}{
      {\node[circle,draw,inner sep=1pt,fill=gray,color=purple]         (x) at
	(-2,-1.3) {};
	\node at (-2,-2.6) {$X$};}
      \node[contact]         (A) at (2,0) {};
      {\node[circle,draw,inner sep=1pt,fill=gray,color=purple]         (y1) at
	(6,-.6) {};
	  \node[circle,draw,inner sep=1pt,fill=gray,color=purple]         (y2) at
	  (6,-2) {};
	  \node at (6,-2.6) {$Y$};}
      {
	\path[color=purple, very thick, shorten >=10pt, shorten <=5pt, ->, >=stealth] (x) edge (A);
	\path[color=purple, very thick, shorten >=10pt, shorten <=5pt, ->, >=stealth] (y1) edge (A);
	\path[color=purple, very thick, shorten >=10pt, shorten <=5pt, ->, >=stealth] (y2)  edge (A);}
}
    \end{tikzpicture}
  \end{center}
\noindent In general, from a category $C$ with pushouts, one has a category $\Cospan(C)$ with objects the objects of $C$ and morphisms the isomorphism classes of cospans \cite{Be}. 

\begin{definition}
Let $\define{\Fin\Cospan}$ be the symmetric monoidal category whose objects are finite sets and morphisms are isomorphism classes of cospans. The tensor product is defined to be disjoint union, which we write as ``+."
\end{definition}

\noindent Determining an underlying isomorphism class of cospans for a circuit defines a functor $H\maps \Circ\to \Fin\Cospan$. 

Next, imagine a machine which had some part that is completely disconnected from the inputs and outputs. One would not consider such a part to actually be a part of the machine and would prefer to simply discard it. We do the same for any connected components of a circuit which have no paths to either inputs or outputs. In fact this can be done without affecting the behavior of the circuit.  For us this means looking at a cospan of finite sets:

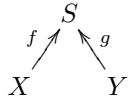
\[
  \xymatrixrowsep{15pt}
  \xymatrixcolsep{5pt}
  \xymatrix{
    & S \\
    X \ar[ur]^-{f} && Y \ar[ul]_-{g}
  }
\]
and replacing $S$ by $f(X)\cup g(Y)$. The result is another cospan where every element in the apex is in the image of at least one function.

\begin{definition}
A cospan $X \stackrel{f}{\rightarrow} S \stackrel{g}{\leftarrow} Y$
in the category of finite sets is \define{jointly epic} if $f(X) \cup g(Y) = S$.   
\end{definition}

A jointly epic cospan determines a partition of $X+Y$ where two points are in the same part of the partition if and only if they map to the same point of $S$. Partitions $P$ of $X+Y$ are dual to the notion of relations, $R\subseteq X\times Y$ so we make the following definition.

\begin{definition}
Given sets $X$ and $Y,$ a \define{corelation} from $X$ to $Y$ is a partition of $X+Y$, or equivalently, an isomorphism class of jointly epic cospans $X \stackrel{f}{\rightarrow} S \stackrel{g}{\leftarrow} Y$.
\end{definition}

This brings us to a subcategory of $\Fin\Cospan$, which is our main focus of attention.

\begin{definition}
Let $\define{\Fin\Corel}$ be the symmetric monoidal category with finite sets as objects and corelations as morphisms. The tensor product is again disjoint union.
\end{definition}

The process of forcing a cospan to become jointly epic defines a unique functor \cite{CF} \[      H' \maps \Fin\Cospan \to \Fin\Corel  .\]
 By taking the composite \[      H'H \maps \Circ \to \Fin\Corel  \] we associate any circuit with an underlying isomorphism class of corelations.  Since it turn out that the underlying corelation completely determines the behavior of a circuit anyway, we take $\Fin\Corel$  to be the category whose morphisms are circuits. 

With this choice in mind, we draw corelations as string diagrams where connected inputs and outputs have paths between them, and disconnected inputs and outputs do not.  For example:

 
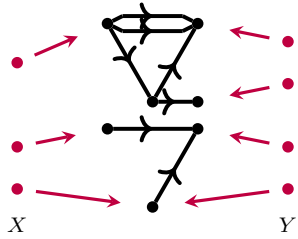
\begin{center}
    \begin{tikzpicture}[circuit ee IEC, set resistor graphic=var resistor IEC
      graphic,scale=.4]
      \node[circle,draw,inner sep=1pt,fill=gray,color=purple]         (x) at
      (-3,-1.3) {};
      \node at (-3,-3.2) {\footnotesize $$};
      \node[circle,draw,inner sep=1pt,fill]         (A) at (0,0) {};
      \node[circle,draw,inner sep=1pt,fill]         (B) at (3,0) {};
      \node[circle,draw,inner sep=1pt,fill]         (C) at (1.5,-2.6) {};
      \node[circle,draw,inner sep=1pt,fill]         (D) at (3,-2.6) {};
      \node[circle,draw,inner sep=1pt,fill=gray,color=purple]         (y1) at
      (6,-.6) {};
      \node[circle,draw,inner sep=1pt,fill=gray,color=purple]         (y2) at
      (6,-2) {};
      \node at (6,-3.2) {\footnotesize $$};
      \coordinate         (ua) at (.5,.25) {};
      \coordinate         (ub) at (2.5,.25) {};
      \coordinate         (la) at (.5,-.25) {};
      \coordinate         (lb) at (2.5,-.25) {};
      \path (A) edge (ua);
      \path (A) edge (la);
      \path (B) edge (ub);
      \path (B) edge (lb);
      \path (ua) edge [->-=.5]  node[label={[label distance=1pt]90:{}}] {} (ub);
      \path (la) edge [->-=.5]   node[label={[label distance=1pt]270:{}}] {} (lb);
      \path (A) edge   [->-=.5]  node[label={[label distance=-1pt]180:{}}] {} (C);
      \path (C) edge  [->-=.5]   node[label={[label distance=-1pt]0:{}}] {} (B);
      \path (C) edge  [->-=.5]   node[label={[label distance=-1pt]0:{}}] {} (D);
      \path[color=purple, very thick, shorten >=10pt, shorten <=5pt, ->, >=stealth] (x) edge (A);
      \path[color=purple, very thick, shorten >=10pt, shorten <=5pt, ->, >=stealth] (y1) edge (B);
      \path[color=purple, very thick, shorten >=10pt, shorten <=5pt, ->, >=stealth] (y2)
      edge (D);

      \node[circle,draw,inner sep=1pt,fill]         (A') at (0,-3.5) {};
      \node[circle,draw,inner sep=1pt,fill]         (B') at (3,-3.5) {};
      \node[circle,draw,inner sep=1pt,fill]         (C') at (1.5,-6.1) {};
      \node[circle,draw,inner sep=1pt,fill=gray,color=purple]         (z1) at
      (6,-4.1) {};
      \node[circle,draw,inner sep=1pt,fill=gray,color=purple]         (z2) at (6,-5.5) {};
      \node at (-3,-6.7) {\footnotesize $X$};
      \node[circle,draw,inner sep=1pt,fill=gray,color=purple]         (y1') at
      (-3,-4.1) {};
      \node[circle,draw,inner sep=1pt,fill=gray,color=purple]         (y2') at
      (-3,-5.5) {};
      \node at (6,-6.7) {\footnotesize $Y$};
      \path (A') edge [->-=.5]  node[above] {} (B');
      \path (C') edge [->-=.5]  node[right] {} (B');
      \path[color=purple, very thick, shorten >=10pt, shorten <=5pt, ->, >=stealth] (y1') edge (A');
      \path[color=purple, very thick, shorten >=10pt, shorten <=5pt, ->, >=stealth] (y2')
      edge (C');
      \path[color=purple, very thick, shorten >=10pt, shorten <=5pt, ->, >=stealth] (z1) edge (B');
      \path[color=purple, very thick, shorten >=10pt, shorten <=5pt, ->, >=stealth]
      (z2) edge (C');
    \end{tikzpicture}
\end{center}

\noindent has corresponding corelation:

\[
  \xymatrixrowsep{1pt}
  \xymatrixcolsep{5pt}
  
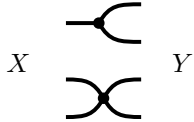
\xymatrix{
   X & \corelation{.075\textwidth} & Y \\
  }
\]
\noindent Additionally, since we are dealing with isomorphism classes anyway, we could have denoted $X$ as $3$ and $Y$ as $4$, so that this is a corelation $f\maps 3\to 4$.

In circuits, corelations, and also bond graphs, the distinction between inputs and outputs is arbitrary. This property is encapsulated by dagger compact categories. The concept of dagger categories were first discussed by Abramsky and Coecke \cite {AC} and further detailed by Selinger \cite{S}. As expected, $\Fin\Corel$ is such a category.

\begin{definition}
A \define{dagger category} is a category $\C$ equipped with an involutive contravariant endofunctor $(-)^{\dagger}\maps \C \to \C$ that is identity on objects. This associates to any morphism $f\maps A\to B,$ a morphism $f^{\dagger} \maps B\to A$, such that for all $f\maps A\to B$ and $g\maps B\to C$ we have $\mathrm{id}_A = \mathrm{id}_A^{\dagger}$, $(gf)^{\dagger}= f^{\dagger}g^{\dagger}$, and $f^{\dagger\dagger} = f$. A morphism is \define{unitary} if its dagger is also its inverse.
\end{definition}
\begin{definition} A \define{symmetric monoidal dagger category} is a symmetric monoidal category $\C$ with a dagger structure $(-)^{\dagger}\maps \C \to \C$ that is symmetric monoidal and where the associator, unitors, and braiding of $\C$ are unitary. 
\end{definition}

To see why $\Fin\Corel$ is a dagger category, notice that for any corelation $f\maps m \to n$ we get a corelation $f^{\dagger} \maps n \to m $ defined to give the same partition, but with the input and output formally exchanged.

\begin{definition}
\label{def:zigzag}
In a symmetric monoidal category $\C$ an object $A^{*}$ is a \define{dual} of $A$ if it is equipped with two morphisms called the \define{unit} $\eta_A\maps I\to A \otimes A^{*}$ and the \define{counit} $\epsilon_A \maps A^{*} \otimes A \to I$ such that the following diagrams commute:

\[
\xymatrix{
I\otimes A \ar[rr]^-{\eta_A \otimes \mathrm{id}_A} \ar[d]_-{r^{-1}_A \circ \ell_A} && (A\otimes A^{*}) \otimes A \ar[d]^-{\alpha_{A,A^{*}, A}} \\
A\otimes I && A\otimes ( A^{*} \otimes A ) \ar[ll]^-{\mathrm{id}_A \otimes \epsilon_A}
}
\]

\[
\xymatrix{
A^{*}\otimes I \ar[rr]^-{ \mathrm{id}_{A^{*}} \otimes \eta_A} \ar[d]_-{\ell^{-1}_A\circ r_A} && A^{*}\otimes (A \otimes A^{*})	\ar[d]^-{\alpha^{-1}_{A^{*}, A,A^{*}}} \\
I\otimes A^{*} && (A^{*}\otimes A) \otimes A^{*} \ar[ll]^-{\epsilon_A \otimes \mathrm{id}_A}
}
\]
A symmetric monoidal category $\C$ is \define{compact closed} if every object $A$ has a dual object $A^{*}$.
\end{definition}

One can show that any two duals of an object are canonically isomorphic, so one may speak of ``the'' dual.   In $\Fin\Corel$ the objects are self-dual. This is most easily seen by looking at the object $1$ where the following morphisms act as the unit and counit respectively:

\[
  \xymatrixrowsep{1pt}
  \xymatrixcolsep{75pt}
  \xymatrix{
    \captwo{.07\textwidth}  &  \cuptwo{.07\textwidth}  \\
     0\to 2 & 2\to 0
  }
\]
The first corelation depicts the partition from $0$ to $2$ where both elements are in the same part, while the second is the same partition, but from $2$ to $0$. The identities in Definition \ref{def:zigzag} are sometimes called the ``zig-zag identities." This choice of name is clear when we express the laws in terms of string diagrams:

\[
  \xymatrixrowsep{1pt}
  \xymatrixcolsep{75pt}
  \xymatrix{
    \zigzaglaw{.07\textwidth}  =  \idone{.07\textwidth} =  \zigzaglawother{.07\textwidth}  \\
  }
\]

\begin{definition}
A \define{dagger compact category} is a dagger symmetric monoidal category that  is compact closed and such that the following diagram commutes for any $A$:

\[
  \xymatrix{
    I \ar[r]^-{\epsilon^{\dagger}_A} \ar[dr]_-{\eta_A} & A\otimes A^{*} \ar[d]^-{\Large \swap{1em}_{A,A^{*}}} \\
    & A^{*} \otimes A 
  }
\]
 where $\swap{1em}_{A,A^{*}}$ is the braiding on $A\otimes A^{*}$.
\end{definition}

\noindent Given a set $n$, the unit corelation $\eta_n \maps 0\to 2n$ is defined as having $n$ parts arranged in the following pattern:

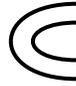
\[
 \xymatrixrowsep{1pt}
 \xymatrixcolsep{75pt}
 \xymatrix{
    \capn{.07\textwidth} 
 }
\]

The dagger of this corelation is precisely the same as the counit $\epsilon_n \maps 2n\to 0$. No new connections are made when composing with the braiding morphism, so the corelation does not change. Thus we say that $\Fin\Corel$ is a dagger compact category.   Similar work can be found for $\Fin\Cospan$  in Fong \cite{BF}.

Note that any function $f \maps X \to Y$ between finite sets gives rise to two corelations $X \stackrel{f}{\rightarrow} Y \stackrel{1}{\leftarrow} Y$ and $Y \stackrel{1}{\rightarrow} Y \stackrel{f}{\leftarrow} X$, which are daggers of each other. We denote $f\maps X\to Y$ to mean $X \stackrel{f}{\rightarrow} Y \stackrel{1}{\leftarrow} Y$, where context determines if we are looking at the corelation or the function. Thus $f^{\dagger}\maps Y\to X$ is the corelation  $Y \stackrel{1}{\rightarrow} Y \stackrel{f}{\leftarrow} X$.

 Now consider the two unique functions $m \maps 2 \to 1$ and $i \maps 0 \to 1$. They induce corelations $m \maps 2 \to 1$, $i \maps 0 \to 1$, $m^{\dagger} \maps 1 \to 2$, $i^{\dagger} \maps 1 \to 0$.  We make some minor notational changes by calling $m^{\dagger} = d$ and $i^{\dagger}=e$. Next we represent these corelations with the following string diagrams:
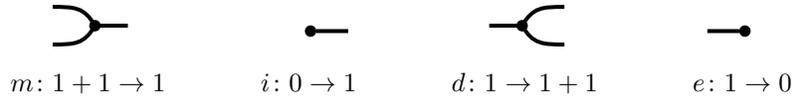
\[
  \xymatrixrowsep{1pt}
  \xymatrixcolsep{30pt}
  \xymatrix{
    \mult{.1\textwidth}  & \unit{.1\textwidth}  & \comult{.4\textwidth} & \counit{.1\textwidth}  \\
    m\maps 1 + 1 \to 1 & i\maps 0 \to 1 & d\maps 1 \to 1 + 1 & e\maps 1 \to 0
  }
\]

\noindent These four corelations obey some interesting equations, which we briefly review.

\begin{definition}
Given a symmetric monoidal category $\C$, a \define{monoid} $(X,\mu, \iota)$ is an object $X\in \C$ together with a \define{multiplication} $\mu \maps X \otimes X \to X $ and \define{unit} $\iota \maps I \to X$ obeying the associative and unit laws. We can draw these laws using string diagrams as follows:
\[
  \xymatrixrowsep{.1pt}
  \xymatrixcolsep{1pt}
  \xymatrix{
    \assocl{.1\textwidth} &=& \assocr{.1\textwidth} &&&&&& \unitl{.1\textwidth} &=&
    \idone{.1\textwidth} &=&  \unitr{.1\textwidth} \\
     \mu (\mu\otimes \mathrm{id}_X) &=&  \mu(\mathrm{id}_X\otimes \mu) &&&&&& \mu (\iota \otimes \mathrm{id}_X) &=& \mathrm{id}_X  &=&  \mu (\mathrm{id}|_X\otimes \iota)
  }
\]
\end{definition}
\begin{definition}
A \define{comonoid} $(X,\delta,\epsilon)$ in $\C$ is an object $X \in \C$  together with a \define{comultiplication} $\delta \maps X \to X\otimes X $ and \define{counit} $\epsilon \maps X \to I$ obeying the coassociative and counit laws:
\[
  \xymatrixrowsep{.1pt}
  \xymatrixcolsep{1pt}
  \xymatrix{
    \coassocl{.1\textwidth} &=& \coassocr{.1\textwidth} &&&&&& \counitl{.1\textwidth} &=&
    \idone{.1\textwidth} &=&\counitr{.1\textwidth}   \\
     ( \mathrm{id}_X \otimes \delta) \delta &=& (\delta \otimes \mathrm{id}_X) \delta &&&&&& (\epsilon \otimes  \mathrm{id}_X) \delta &=&  \mathrm{id}_X  &=&  ( \mathrm{id}_X\otimes \epsilon)\delta
  }
\]
\end{definition}

\begin{definition}
A \define{Frobenius monoid} in a symmetric monoidal category $\C$ is a monoid $(X,\mu, \iota)$ together with a comonoid $(X,\delta,\epsilon)$ obeying the Frobenius laws:

\[
  \xymatrixrowsep{.1pt}
  \xymatrixcolsep{1pt}
  \xymatrix{
         \frobs{.1\textwidth} &=& \frobx{.1\textwidth} &=& \frobz{.1\textwidth} \\
     (\mathrm{id}_X \otimes \mu) (\delta \otimes \mathrm{id}_X) &=& \delta \mu  &=&  (\mu \otimes \mathrm{id}_X) (\mathrm{id}_X \otimes \delta)
  }
\]

\noindent A Frobenius monoid is
\begin{itemize}
\item \define{commutative} if:
\[
  \xymatrixrowsep{.1pt}
  \xymatrixcolsep{1pt}
  \xymatrix{
    \commute{.1\textwidth} &=& \mult{.07\textwidth} \\
     \mu \swap{1em} &=& \mu
  }
\]
where $\swap{1em}$ is the braiding on $X \otimes X$

\item \define{cocommutative} if:
\[
  \xymatrixrowsep{.1pt}
  \xymatrixcolsep{1pt}
  \xymatrix{
    \cocommute{.1\textwidth} &=& \comult{.07\textwidth} \\
  \swap{1em} \delta &=& \delta
  }
\]
Note that a Frobenius monoid is cocommutative if and only if it is commutative.

\item \define{symmetric} if:
\[
  \xymatrixrowsep{.1pt}
  \xymatrixcolsep{1pt}
  \xymatrix{
    \symmetric{.1\textwidth} &=& \multunit{.07\textwidth} \\
 \epsilon \mu \swap{1em} &=& \epsilon \mu
  }
\]

\item \define{cosymmetric} if:
\[
  \xymatrixrowsep{.1pt}
  \xymatrixcolsep{1pt}
  \xymatrix{
    \cosymmetric{.1\textwidth} &=& \comultcounit{.07\textwidth} \\
 \swap{1em} \delta \iota &=& \delta \iota
  }
\]
Note that a Frobenius monoid is cosymmetric if and only if it is symmetric.

\item \define{special} if:
\[
  \xymatrixrowsep{.1pt}
  \xymatrixcolsep{1pt}
  \xymatrix{
    \spec{.1\textwidth} &=&  \idone{.1\textwidth} \\
   \mu \delta &=& \mathrm{id}_X
  }
\]

\item \define{extra} if:
\[
  \xymatrixrowsep{.1pt}
  \xymatrixcolsep{1pt}
  \xymatrix{
\extral{.1\textwidth} &=& \idonezero{.1\textwidth} \\
\epsilon \iota &=& \mathrm{id}_I
  }
\]
\end{itemize}

\end{definition}

It can be shown that the object $(1,m,i,d,e)$ in $\Fin\Corel$ is an extraspecial commutative Frobenius monoid. The precise connection between such objects and corelations was worked out by Fong and the author \cite{CF}. To summarize, strict monoidal functors from $\Fin\Corel$ to symmetric monoidal categories correspond to extraspecial commutative Frobenius monoids. Additionally, the morphisms $m,i,d$, and $e$ ``generate" the morphisms of $\Fin\Corel$ in a sense that we discuss in Section \ref{sec:props}.

With the above background in mind we now turn our attention to ports. In electrical engineering a port is a pair of terminals at the end of a pair of wires with opposing current. The opposite current data is assigned to a pair of wires later when we look at semantics assigning functors in Section \ref{sec:functors}. So we define a port to be the following.

\begin{definition}
We say that a \define{port} is a copy of the object $2\in \Fin\Corel.$
\end{definition}


In terms of circuits, if we are given three ports, then there are two important ways to connect them. One way is with a  ``series junction" and the other is with a  ``parallel junction." These are also referred to as $1$-junctions and $0$-junctions, respectively, in bond graph literature \cite{HP}. We now study the morphisms in $\Fin\Corel$ which act as $1$-junctions. These morphisms correspond to a multiplication and comultiplication structure associated to the object $2$, which are a part of an overall larger Frobenius monoid structure associated to $2$.



\section{Series Junctions}
\label{sec:series}

The first monoid structure we equip to the object $2$ arises naturally as a monad constructed from an adjunction. We see that the morphisms associated to this monoid are comparable to $1$-junctions because of the relations obeyed. Later, in Proposition \ref{prop:blackbox}, where we prove that the morphisms are mapped to the behaviors corresponding to $1$-junctions, we see a more definitive reason for the association. After all, there is another very similar monoid in $\Fin\Corel$, which we look at in Section \ref{sec:parallel}, that  also obeys similar relations. 

The two morphisms,  $ d \circ i \maps 0 \to 2$ and $e\circ m \maps 2\to 0,$ form an adjunction in $\Fin\Corel$, viewed as a one object bicategory. To see that these morphisms form an adjunction, first draw them using string diagrams:

\[
  \xymatrixrowsep{1pt}
  \xymatrixcolsep{8pt}
  \xymatrix{
    \captwo{.07\textwidth}  :=  & \comultcounit{.07\textwidth} & = & d\circ i \maps 0\to 2 \\
    \cuptwo{.07\textwidth}  :=  & \multunit{.07\textwidth} & = & e\circ m\maps 0\to 2
  }
\]
\noindent We may then check that the zig-zag identities hold from $(1,m,i,d,e)$  being an extraspecial commutative Frobenius monoid. Note this also follows from $1$ being self-dual:

\[
  \xymatrixrowsep{1pt}
  \xymatrixcolsep{10pt}
  \xymatrix{
    \zigzag   && =  & \idone{.1\textwidth}\\
  }
\]

Now recall that from any adjunction one can construct a monad. For us this makes $2$ in $\Fin\Corel$ into a monoid with the morphisms $\mathrm{id}_1 + (e\circ m) + \mathrm{id}_1\maps 4 \to 2$ and $d\circ i\maps 0 \to 2$, acting as multiplication and the unit respectively. These are drawn as follows:

\[
  \xymatrixrowsep{1pt}
  \xymatrixcolsep{5pt}
  \xymatrix{
    m_2 &:= & \monadmult{.1\textwidth}  &= &   \mathrm{id}_1 + (e\circ m) + \mathrm{id}_1\maps 4 \to 2  \\
    i_2 &:= & \captwo{.1\textwidth}  &= & d\circ i \maps 0\to 2  \phantom{hadalssdedddi} \\
  }
\]

Although the way in which $(2,m_2,i_2)$ forms a monoid is well known \cite{RS}, we give the diagrammatic proof using our string diagrams in $\Fin\Corel$ for completeness.

Associativity follows from only the monoidal structure of $\Fin\Corel$, while the left and right unit laws follow from the zig-zag identities governing adjunctions:

\[
  \xymatrixrowsep{1pt}
  \xymatrixcolsep{5pt}
  \xymatrix{
    \monadassocl{.1\textwidth}  &= &\monadassocm{.1\textwidth}  &= &\monadassocr{.1\textwidth} \\
  }
\]
\[
  \xymatrixrowsep{1pt}
  \xymatrixcolsep{5pt}
  \xymatrix{
    \monadunitl{.1\textwidth}  &= & \identitytwo{.1\textwidth} &= & \monadunitr{.1\textwidth} \\
  }
\]

Since $1$ is an extraspecial commutative Frobenius monoid and since $\Fin\Corel$ is dagger compact, $2$ can be equipped with additional morphisms, which obey far more equations than just those of a monoid.  First, we turn $m_2$ and $i_2$ around to get two more corelations which we call $d_2$ and $e_2.$ These are drawn in the following way:

\[
  \xymatrixrowsep{1pt}
  \xymatrixcolsep{5pt}
  \xymatrix{
    d_2 & = & \monadcomult{.07\textwidth}   &= &   \mathrm{id}_1 + (d\circ i) + \mathrm{id}_1\maps 2 \to 4  \\
   e_2 & = &  \cuptwo{.07\textwidth}  &= & e\circ m \maps 0\to 2  \phantom{hadalsedddi} \\
  }
\]

\noindent and they give us a comonoid $(2,d_2,e_2)$:

\[
  \xymatrixrowsep{1pt}
  \xymatrixcolsep{5pt}
  \xymatrix{
    \monadcoassocl{.1\textwidth}  &= &\monadcoassocm{.1\textwidth}  &= &\monadcoassocr{.1\textwidth} \\
  }
\]
\[
  \xymatrixrowsep{1pt}
  \xymatrixcolsep{5pt}
  \xymatrix{
    \monadcounitl{.1\textwidth}  &= & \identitytwo{.1\textwidth} &= & \monadcounitr{.1\textwidth} \\
  }
\]
Series junctions between ports impose equations which allow one to simplify and redraw circuit diagrams. 
These simplifications correspond to how the monoid and comonoid structure interact.

\begin{theorem}\label{thm:series}
The object $(2,m_2,i_2,d_2,e_2)$ in $\Fin\Corel$ is an extraspecial symmetric Frobenius monoid.
\end{theorem}

\begin{proof}

We already have the monoid and comonoid structure so we begin with the Frobenius laws. These follow due only to the monoidal structure of $\Fin\Corel$: 

\[
  \xymatrixrowsep{1pt}
  \xymatrixcolsep{5pt}
  \xymatrix{
    \monadfrobl{.1\textwidth}  &= & \monadfrobm{.1\textwidth} &= & \monadfrob{.1\textwidth} \\
  }
\]

\noindent The other Frobenius law is proven in a similar manner. For the extraspecial structure, it is necessary to first show the extra property, $e_2\circ i_2 = \mathrm{id}_0.$ This comes from first using the special property $m\circ d=\mathrm{id}_1$ and then the extra property $i\circ e=\mathrm{id}_0$:

\[
  \xymatrixrowsep{1pt}
  \xymatrixcolsep{5pt}
  \xymatrix{
    \monadextra{.1\textwidth}  &= & \monadextram{.1\textwidth} &= & \extral{.1\textwidth} \phantom{hf}=  \\
  }
\]

\noindent Using this fact we can show the special property, $d_2\circ m_2 = \mathrm{id}_2$:
\[
  \xymatrixrowsep{1pt}
  \xymatrixcolsep{5pt}
  \xymatrix{
    \monadspec{.1\textwidth}  &= & \identitytwo{.1\textwidth}   \\
  }
\]

\noindent Since the multiplication for $1$ is commutative, it is also symmetric. We use this symmetric  property to show the multiplication for $2$ is also symmetric. This gives us the following:
\[
  \xymatrixrowsep{1pt}
  \xymatrixcolsep{5pt}
  \xymatrix{
    \monadsymml{.1\textwidth}  &= & \monadsymmm{.1\textwidth}  &= & \monadsymmr{.1\textwidth}  &= & \monadsymmrr{.1\textwidth}   \\
  }
\]

\noindent We can further simplify this because of the naturality of the braiding, which allows us to move the cup past the braiding. Then we can use that we are in a symmetric monoidal category to cancel the two braidings:
\[
  \xymatrixrowsep{1pt}
  \xymatrixcolsep{5pt}
  \xymatrix{
    \monadsymmrr{.1\textwidth}   &= & \monadsymmcups{.13\textwidth}  &= & \monadsymmend{.13\textwidth} 
  }
\qedhere 
\]
\end{proof}

The morphisms $m_2$ and $d_2$ together with the equations between them correspond exactly to $1$-junctions and their equations. Due to this it is natural to try to define unary $1$-junctions using the morphisms $i_2$ and $e_2$.  In fact all of the properties obeyed involving the morphisms, $m_2,d_2,i_2$ and $,e_2$ do translate to relations which come from the equations associated to $1$-junctions. However, the braiding morphism does not fit perfectly into the analogy. If  braiding of bonds is allowed then the junctions would obey the commutative property, but we have only symmetry.


\section{Parallel Junctions}
\label{sec:parallel}

Next we equip another Frobenius monoid structure to the object $2$ through another very common construction.  The morphisms associated to this monoid act similarly to parallel junctions. The necessary relations are obeyed and also the semantics assigning functor in Proposition \ref{prop:blackbox} assigns this Frobenius monoid to the behaviors exhibited by $0$-junctions.

 From a pair of monoids $(X,m_X,i_X)$ and $(Y,m_Y,i_Y)$ in a braided monoidal category, there is a standard way to make $X\otimes Y$ into a monoid with $$ ( m_ X\otimes m_Y)\circ(\mathrm{id}_X \otimes \swap{1em} \otimes \mathrm{id}_Y)\maps (X\otimes Y) \otimes (X\otimes Y)\to (X\otimes Y)$$ as multiplication and $$i_X\otimes i_Y\maps I \to (X\otimes Y)$$ as the unit. 
\noindent Consider the object $1+1 =2 \in \Fin\Corel$ with the following morphisms:

\[
  \xymatrixrowsep{1pt}
  \xymatrixcolsep{8pt}
  \xymatrix{
    \mu_2 & :=& \parmult{.07\textwidth}  & =  & ( m + m)\circ(\mathrm{id}_1 + \swap{1em} + \mathrm{id}_1)\maps 4\to 2  \\
    \iota_2 & := & \parunit{.07\textwidth}  &  =  & i + i\maps 0 \to 2  \phantom{hadasffaasfaldfai} 
  }
\]

Then $(2,\mu_2,\iota_2)$ is a monoid in $\Fin\Corel$. Since both monoids in this construction are $(1,m,i)$, our new monoid inherits the same properties that $(1,m,i)$ has. We also have corelations $\delta_2$ and  $\epsilon_2$, which when drawn as string diagrams are reflections of $\mu_2$ and $\iota_2$ respectively: 

\[
  \xymatrixrowsep{1pt}
  \xymatrixcolsep{8pt}
  \xymatrix{
    \delta_2 & :=& \parcomult{.07\textwidth}  & =  & (\mathrm{id}_1 + \swap{1em} + \mathrm{id}_1) \circ ( d + d)\maps 2\to 4  \\
    \epsilon_2 &:= & \parcounit{.07\textwidth}  &  =  & e + e\maps 2 \to 0  \phantom{hadasffsfaddldfai} 
  }
\]

\begin{theorem}\label{thm:parallel}
The object $(2, \mu_2,\iota_2, \delta_2, \epsilon_2) $  in $\Fin\Corel$ is an extraspecial commutative Frobenius monoid.
\end{theorem}

\begin{proof}

Each law is inherited from the extraspecial commutative Frobenius monoid $1$ due to the naturality of the braiding. We start with the monoid laws:
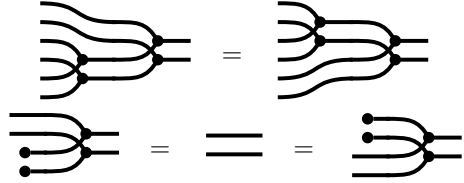
\[
  \xymatrixrowsep{1pt}
  \xymatrixcolsep{5pt}
  \xymatrix{
    \parassocl{.12\textwidth}  &= & \parassocr{.12\textwidth}   
  }
\]

\vspace{-3ex}

\[
  \xymatrixrowsep{1pt}
  \xymatrixcolsep{5pt}
  \xymatrix{
    \parunitl{.1\textwidth}  &= & \identitytwo{.12\textwidth} &= & \parunitr{.11\textwidth} 
  }
\]

\noindent Similarly we get a comonoid structure:
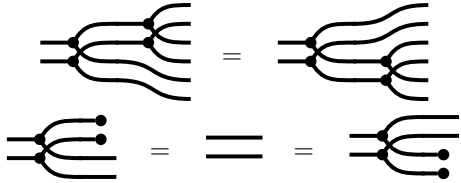
\[
  \xymatrixrowsep{1pt}
  \xymatrixcolsep{5pt}
  \xymatrix{
    \parcoassocl{.12\textwidth}  &= & \parcoassocr{.12\textwidth}   
  }
\]

\vspace{-3ex}

\[
  \xymatrixrowsep{1pt}
  \xymatrixcolsep{5pt}
  \xymatrix{
    \parcounitl{.1\textwidth}  &= & \identitytwo{.12\textwidth} &= & \parcounitr{.1\textwidth} 
  }
\]

 \noindent We next have the Frobenius laws:
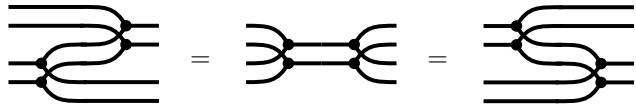
\[
  \xymatrixrowsep{1pt}
  \xymatrixcolsep{5pt}
  \xymatrix{
    \parfrobl{.12\textwidth}  &= & \parfrobm{.12\textwidth} &= & \parfrobr{.12\textwidth} 
  }
\]

\noindent Then we have the extra and special laws:
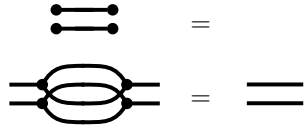
\[
  \xymatrixrowsep{1pt}
  \xymatrixcolsep{5pt}
  \xymatrix{
	\parextra{.11\textwidth} & = & \\
    \parspec{.12\textwidth}  &= & \identitytwo{.12\textwidth}  
  }
\]

\noindent Finally, we have commutativity:
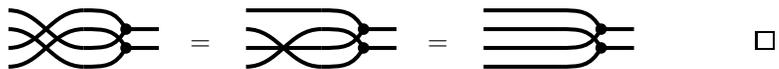
\[
  \xymatrixrowsep{1pt}
  \xymatrixcolsep{5pt}
  \xymatrix{
    \parcoml{.13\textwidth}  &= &  \parcomm{.13\textwidth} &= &  \parcomr{.13\textwidth}   
  }
\qedhere 
\]
\end{proof}

In terms of circuits this says that the connectivity of parallel junctions is all that matters. Unlike the Frobenius monoid associated to $1$-junctions, this Frobenius monoid perfectly captures the properties of the $0$-junctions, their unary versions, and their connection to a braiding.  Next we show that the way in which these two Frobenius monoids corresponds nicely to the way in which series and parallel junctions interect. However, we shall see some issues that arise, which force us to consider a second approach.


\section{Weak Bimonoids}
\label{sec:bimonoids}

We have found two Frobenius monoid structures on the object $2 \in \Fin\Corel$, one related to series junctions and one related to parallel junctions.  We now describe how these two structures $(2,m_2,i_2,d_2,e_2)$ and $(2, \mu_2, \iota_2, \delta_2, \epsilon_2)$ interact.

\begin{definition} A \define{weak bimonoid} in a braided monoidal bicategory $\C$ is a pair consisting of a  monoid $(X,\mu, \iota)$ in $\C$ together with a comonoid $(X,\delta,\epsilon)$ in $\C$ obeying:

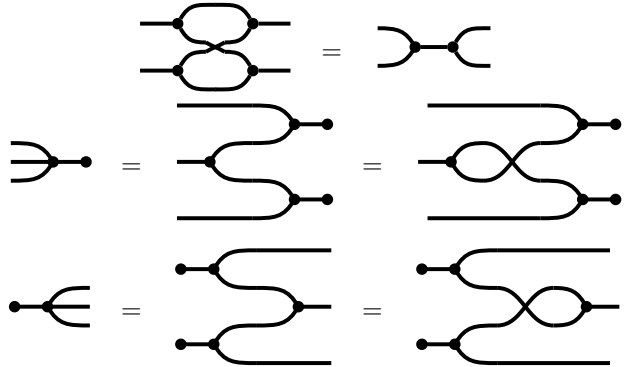
\[
  \xymatrixrowsep{1pt}
  \xymatrixcolsep{5pt}
  \xymatrix{
    \bi{.13\textwidth}  &= & \frobx{.11\textwidth} 
  }
\]

\vspace{-3ex}

\[
  \xymatrixrowsep{1pt}
  \xymatrixcolsep{5pt}
  \xymatrix{
    \weakbil{.07\textwidth}  &= & \weakbim{.13\textwidth} &= & \weakbir{.16\textwidth}  \\
    \weakcobil{.07\textwidth}  &= & \weakcobim{.14\textwidth} &= & \weakcobir{.16\textwidth}
  }
\]

\end{definition}

This differs from a \define{bimonoid}, where the second and third sets of equations are replaced by:

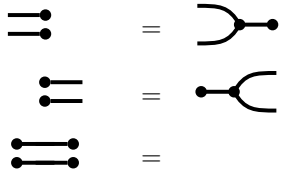
\[
  \xymatrixrowsep{.1pt}
  \xymatrixcolsep{5pt}
  \xymatrix{
     \parcounit{.06\textwidth}  &= & \multunit{.07\textwidth}  \\
     \parunit{.06\textwidth}   &=  & \comultcounit{.07\textwidth} \\
	\parextra{.09\textwidth} & = &
}
\]
It is a fun exercise to show that every bimonoid is a weak bimonoid using string diagrams. Frobenius monoids and bimonoids are the more commonly studied structures that combine monoids and comonoids, while weak bimonoids are much newer. These were introduced by Pastro and Street in their work on quantum categories \cite{PS}. They also proved the following fact, which we use to show how our monoids interact.

\begin{theorem}[Pastro, Street, \cite{PS}]\label{thm:RS}
If $X$ is a special commutative Frobenius monoid in a braided monoidal category then $X\otimes X$, equipped with the following morphisms, is a weak bimonoid:
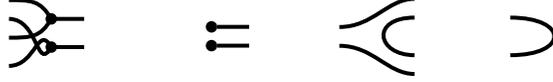
\[
  \xymatrixrowsep{1pt}
  \xymatrixcolsep{25pt}
  \xymatrix{
 	\weirdmult{.06\textwidth}  &  \parunit{.06\textwidth}    &  \monadcomult{.07\textwidth}   &\cuptwo{.07\textwidth}  \\
  }
\]
\end{theorem}

Since the Frobenius monoid $(1,m,i,d,e)$ is also commutative, we can replace the multiplication morphism with the following:
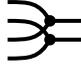
\[
  \xymatrixrowsep{1pt}
  \xymatrixcolsep{25pt}
  \xymatrix{
 	\parmult{.07\textwidth} \\
  }
\]
\vspace{-3ex}

\noindent and the theorem still holds. This gives us a weak bimonoid, but since  $(1,m,i,d,e)$ is an extraspecial commutative Frobenius monoid our weak bimonoid will obey some more relations.  In fact,  because $\Fin\Corel$ is dagger compact there is a second weak bimonoid.

\begin{theorem}\label{thm:weakbimonoid}
$(2,\mu_2,\iota_2, d_2,e_2)$ and $(2,m_2,i_2,\delta_2,\epsilon_2)$ are weak bimonoids. Additionally:
\begin{itemize}
\item  the extra law for both weak bimonoids, $\epsilon_2\circ i_2 =\mathrm{id}_0= e_2 \circ \iota_2$
\item  $m_2\circ \delta_2 = \mu_2 \circ d_2$
\item $(m_2\circ \delta_2)^2 = m_2 \circ \delta_2$.
\end{itemize}
\end{theorem}
\begin{proof}
 We first prove that $(2,\mu_2,\iota_2, d_2,e_2)$ is a weak bimonoid. Note that proving $(2,m_2,i_2,\delta_2,\epsilon_2)$ is a weak bimonoid amounts to the same thing. Our diagrammatic proofs are similar to those used by Pastro and Street \cite{PS}.

\noindent The first weak bimonoid law to show is: $$(\mu_2 + \mu_2)\circ (\mathrm{id}_2 + \swap{3ex}_{2,2} + \mathrm{id}_2) \circ (d_2 + d_2) = d_2\circ \mu_2. $$ We begin by using the naturality of the braiding and then we use the cocommutativity of the comonoid $(1,d,e)$.  Following this we need the Frobenius and special laws for $(1,m,i,d,e)$:

\vspace{-1ex}

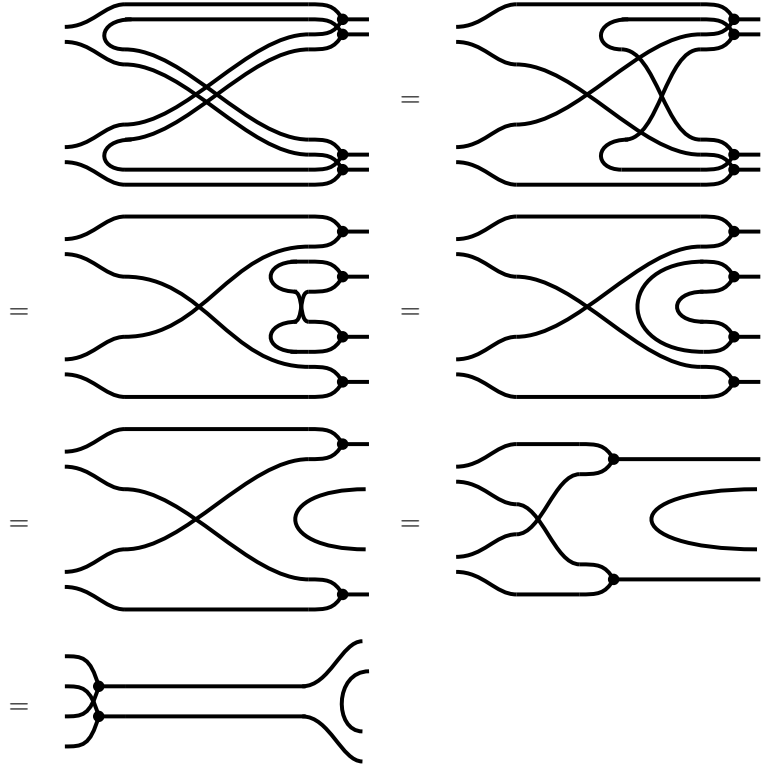
\[
  \xymatrixrowsep{1pt}
  \xymatrixcolsep{5pt}
  \xymatrix{
  &   & \bigbimonoidone{.35\textwidth}  &= & \bigbimonoidtwo{.35\textwidth}  \\
     &= & \bigbimonoidthree{.35\textwidth}  &= &   \bigbimonoidfour{.35\textwidth}\\
	&= &     \bigbimonoidfive{.35\textwidth} &= &  \bigbimonoidsix{.35\textwidth} \\
 	&=& \bigbimonoidseven{.35\textwidth}
  }
\]

\vspace{-1ex}

\noindent Next we have the other weak bimonoid laws:
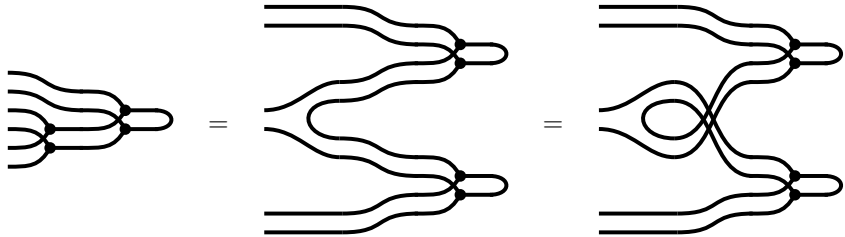
\[
  \xymatrixrowsep{1pt}
  \xymatrixcolsep{5pt}
  \xymatrix{
    \weakbimonoidtwol{.12\textwidth}  &= & \weakbimonoidtwom{.12\textwidth} &= & \weakbimonoidtwor{.12\textwidth}   \\
  }
\]
\noindent The easiest way to see that this is true is to note that these are both corelations from $6$ inputs to $0$ outputs where all of the inputs are path connected to each other. Thus they are both the partition with all inputs in the same part, and no outputs, so they are the same corelation. 

The last weak bimonoid law  is the following:
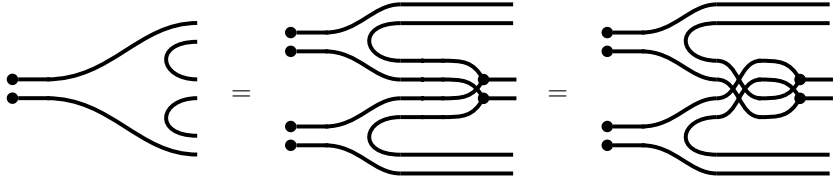
\[
  \xymatrixrowsep{1pt}
  \xymatrixcolsep{5pt}
  \xymatrix{
    \weakbimonoidcotwol{.1\textwidth}  &= &\weakbimonoidcotwom{.1\textwidth}  &= &\weakbimonoidcotwor{.12\textwidth} \\
  }
\]
\noindent This comes from passing the two inner unit morphisms through the braiding and then using the unit laws to delete them. Then we merely have two caps in between two units as we want. This shows that $(2,\mu_2,\iota_2, d_2,e_2)$ is a weak bimonoid. Next we prove one of the extra laws since the other is similar:
\[
  \xymatrixrowsep{1pt}
  \xymatrixcolsep{5pt}
  \xymatrix{
   \bimonoidextral{.07\textwidth} &= &  \bimonoidextram{.13\textwidth} \enspace  &= & \extral{.12\textwidth} 	\phantom{hf}=  \\
}
\]

\noindent We also have $m_2\circ \delta_2 =d\circ m= \mu_2 \circ d_2$:
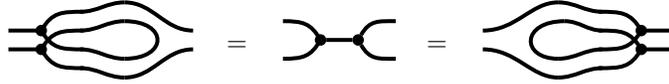
\[
  \xymatrixrowsep{1pt}
  \xymatrixcolsep{5pt}
  \xymatrix{
   \strangelawl{.07\textwidth} &= &  \frobx{.1\textwidth} &= &  \strangelawr{.07\textwidth} \\
}
\]
\noindent One way to see this is that these all describe the corelation from $2$ to $2$ where all of the inputs and outputs are in the same part. We may also understand this by noticing that commutativity allows the inner ``loop" to be drawn as a cup or cap.
\noindent Finally, $(m_2\circ \delta_2)^2 = m_2 \circ \delta_2$, comes from $d\circ m$ being idempotent, since $d\circ m \circ d\circ m = d\circ id_2\circ m = d\circ m$. It follows that $ (\mu_2 \circ d_2)^2 =  \mu_2 \circ d_2$.
\end{proof}

The weak bimonoid laws encapsulate the same properties obeyed by interacting $1$- and $0$-junctions, including their unary versions. Unfortunately the law: $$m_2\circ \delta_2 = \mu_2 \circ d_2$$ is not a property that holds for bond graphs. It would indicate that a $0$-junction followed by a $1$-junction is equivalent to the reverse, however they are opposite to eachother, not the same. 

 We next want to define a symmetric monoidal subcategory of $\Fin\Corel$ using the morphisms $m_2,i_2,\delta_2,\epsilon_2,\mu_2,\iota_2, d_2$, and $e_2$ as generators. In order to describe this subcategory in a quick and simple way we first introduce the framework of ``props."  These types of categories can be described with generators and equations. Another use for props is that we can use them to easily define functors. The payoff comes in Section \ref{sec:functors} when we discuss the functorial semantics, first introduced by Lawvere \cite{Law}, of bond graphs.


\section{Props}
\label{sec:props}


Props were first introduced by Mac Lane \cite{Ma65}, where they were called PROPs. This acronym stands for ``product and permutation categories," and we choose to write ``prop" instead of ``PROP." Originally they were used to extend ``Lawvere theories" \cite{Law}, which are categories with finite cartesian products where every object is of the form $X^n$ for some distinguished object $X$. Lawvere theories act as theories for structures which are sets equipped with $n$-ary operations.  Props are more general because they use tensor powers in a monoidal category and because they allow for operations with any finite number of inputs and outputs. 

\begin{definition}
A \define{prop} is a strict symmetric monoidal category having the natural
numbers as objects, with the tensor product of objects given by addition.  We define a morphism of props to be a strict symmetric monoidal functor that is the identity on objects.  Let $\define{\PROP}$ be the category of props.
\end{definition}

\begin{definition}
If $\T$ is a prop and $\C$ is a symmetric monoidal category, an 
\define{algebra of} $\T$ \define{in} $\C$ is a strict symmetric monoidal functor 
$F \maps \T \to \C$.   We define a morphism of algebras of $\T$ in  
$\C$ to be a monoidal natural transformation between such functors. 
\end{definition}

The next result implies that $\Fin\Corel$ and $\Fin\Corel^{\circ}$ are equivalent to props. 

\begin{proposition} \cite{BC}
\label{prop:strictification_1}
A symmetric monoidal category $\C$ is equivalent to a prop if and only if there is an object $x \in \C$ such that every object of $\C$ is isomorphic to $x^{\otimes n}$ for some $n \in \N$.  
\end{proposition}

\noindent Given a prop, one attempts to figure out its algebras. We provide some examples where this has been done.

\begin{example}
\label{ex:finset}
The category of finite sets and functions can be made into a symmetric monoidal category where the tensor product of sets is their disjoint union.  Proposition \ref{prop:strictification_1}  implies that this symmetric monoidal category is equivalent to a prop.  We write $\define{\Fin\Set}$ to mean this prop.  We identify this prop with a skeleton of the category of finite sets and functions, having finite ordinals $0, 1, 2, \dots$ as objects.

The algebras of $\Fin\Set$ are commutative monoids \cite{Pi}. We can see why by examining the unique functions $m \maps 2 \to 1$ and $i\maps 0 \to 1$. Given a symmetric monoidal functor $F\maps \Fin\Set \to \C$, the object $F(1)$ is a commutative monoid in $\C$ with maps $F(m) \maps F(1) \otimes F(1) \to F(1)$ and $F(i) \maps F(0) \to F(1)$. On the other hand, any commutative monoid in $\C$ comes from a unique choice of $F$, which is harder to show. 
\end{example}

\begin{example}
\label{ex:fincorel}
Proposition \ref{prop:strictification_1} implies that $\Fin\Corel$ is equivalent to a prop.  We call this prop $\define{\Fin\Corel}$, and we identify this prop with a skeleton of the category of finite sets and corelations having finite ordinals as objects.  From now on we use $\Fin\Corel$ to mean the equivalent prop.

Fong and the author \cite{CF} showed that the algebras of $\Fin\Corel$ are extraspecial commutative Frobenius monoids. To see why, recall that the corelations $m \maps 2 \to 1$, $i \maps 0 \to 1$, $d \maps 1 \to 2$, $e \maps 1 \to 0$  make the object $1$ into an extraspecial Frobenius commutative monoid in $\Fin\Corel$. Then it follows from Lack \cite{La}, who showed that the algebras of $\Fin\Cospan$ are special commutative Frobenius monoids.

\end{example}

Additionally, Fong and the author \cite{CF} summarize the relationship between cospans, corelations, spans, relations, and their algebras in terms of props. This is done by presenting props in terms of signatures and equations, which act like generators and relations do for groups. 

\begin{definition} Define the category of \define{signatures} to be the functor category $\Set^{\N \times \N}$.
\end{definition}

A signature $\Sigma\maps \N \times \N \to \Set$ in $\Set^{\N \times \N}$ is thus a functor which picks out a set $\Sigma(m,n)$ for each pair $(m,n) \in \N\times \N$. The idea is that each set, $\Sigma(m,n)$, acts like a set of generating morphisms for a prop. Then for each pair, $(m,n)$, we get a set of generators from $m$ to $n$ for a prop. We summarize the precise framework necessary to generate props in this way, details of which were done by Baez and the author \cite{BC}.

\begin{proposition}\label{prop:signature}
There is a functor
\[          F \maps \Set^{\N \times \N} \to \PROP  \]
 which takes any signature $\Sigma$ to $F\Sigma$, the \define{free prop} on $\Sigma$, having morphisms corresponding to the pairs $(m,n)$ picked out by $\Sigma$ and those generated by these morphisms by composition and tensoring.
\end{proposition}

\begin{corollary}\label{cor:signature}
The category $\PROP$ is cocomplete, and any prop $\T$ is the coequalizer of some diagram:
\[
    \xymatrix{
      FE \ar@<-.5ex>[r]_{\rho} \ar@<.5ex>[r]^{\lambda} & F\Sigma.
    }
\]
\end{corollary}
 
We say that $(\Sigma,E)$ \define{presents} the prop $T$,  that $\Sigma$ is the set of \define{generators}, and that $E$ is the set of \define{equations}. A prop which is presented in this way is also precisely the same as the prop whose algebras correspond to the object described by signature and equations. In other words the algebras of a prop $T$ presented by $(\Sigma,E)$ are objects $(X,f_1,\ldots, f_n)$ obeying equations $(e_1,\ldots,e_m)$ where $f_i$ correspond to elements of the signature $\Sigma$ and $e_j$ correspond elements of the set of equations $E$.

 \begin{example}
Since the algebras of $\Fin\Corel$ are extraspecial commutative Frobenius monoids $\Fin\Corel$ is equivalent to the coequalizer of:

\[
    \xymatrix{
      FE \ar@<-.5ex>[r]_{\rho} \ar@<.5ex>[r]^{\lambda} & F\Sigma
    }
\]
with $\Sigma$ having generators $\mu\maps 2\to 1, \iota \maps 0\to 1, \delta \maps 1\to 2, \epsilon \maps 1\to 0$ and $E$ corresponding to the equations of an extraspecial commutative Frobenius monoid on morphisms built from these generators.

\end{example}

\begin{definition}
Let $\define{\Fin\Corel^{\circ}}$ be the symmetric monoidal subcategory of $\Fin\Corel$ whose objects are $0,2,4,...$ and whose morphisms are generated by composing and tensoring 
$\{m_2,i_2,\delta_2,\epsilon_2,\mu_2,\iota_2, d_2,e_2\}$ in $\Fin\Corel.$ 
\end{definition}

\noindent $\Fin\Corel^{\circ}$ is equivalent to a prop since the objects are generated by $2$.

\begin{conjecture}\label{con:presentation}
The prop $\Fin\Corel^{\circ}$ can be presented with $8$ generators corresponding to $\{m_2,i_2,\delta_2,\epsilon_2,\mu_2,\iota_2, d_2,e_2\}$ and equations corresponding to the equations shown in Theorem \ref{thm:series}, Theorem \ref{thm:parallel}, and Theorem \ref{thm:weakbimonoid}.
\end{conjecture}

By definition, given a symmetric monoidal functor $F\maps \Fin\Corel^{\circ} \to \C$, the object $F(2)$ together with the associated morphisms obey the required equations. However, the converse is once again much more difficult to show.

We shall see that this prop is useful for transforming typical bond graph diagrams into circuit diagrams. This is due to the fact that when restricting only to the multiplication and comultiplication morphisms, together with the allowable ways of combining bond graphs, all of the desired properties are obeyed. It is only when we throw in the other morphisms and the other ways of composing that there are issues.

Another consequence of Corollary \ref{cor:signature} is that morphisms of props can be defined on generators as long the the equations hold in the image. That is-- to get a functor $F\maps T\to T'$ for props presented  by $(E,\Sigma)$ and $(E',\Sigma'$) we need only define $F(\Sigma)$ and check that the equations from $E$ hold in $T'$ between the morphisms built from those in $F(\Sigma)$.

Similarly, since an algebra of a prop $T$ presented by $(\Sigma,E)$ corresponds to some type of object, if we have an example of such an object in symmetric monoidal category, then this also defines a functor. In other words, given an object $c\in C$ equipped with morphisms obeying equations corresponding to those in $E$, then there is a unique functor $F\maps T\to C$ such that $F(1) = c$. 

The props we have looked at and the props to come are also dagger compact categories. Thus we need to define morphisms of props which also preserve the dagger structure.

\begin{definition}
A \define{symmetric monoidal dagger functor} is a symmetric monoidal functor $F\maps \C \to \D$ between dagger compact categories such that $F((-)^{\dagger}) = (F(-))^{\dagger}$, $\mu_{x,y}\maps F(x) \otimes F(y) \to F(x\otimes y)$ is unitary for all $x,y\in \C$, and $\eta \maps I \to F(I)$ is unitary. A \define{dagger morphism of props} is a morphism of props $F$ between props such that  $F((-)^{\dagger}) = (F(-))^{\dagger}$.
\end{definition}


\section{Lagrangian Relations}
\label{sec:Lagrangian}

Any circuit can be associated with a space of possible potential and current pairs, which is then called the ``behavior" of the circuit. When we say that a circuit is determined by the underlying corelation we mean that the corelation completely determines the behavior. Similarly, any bond graph can be associated with a space of possible effort and potential pairs, which is called the behavior of the bond graph. These are ``Lagrangian subspaces" of a symplectic vector space.

\begin{definition}
A \define{symplectic vector space} $V$ over a field $k$ is a finite-dimensional vector
space equipped with a \define{symplectic structure} $\omega$, meaning a map $\omega \maps V \times  V \to k$ that is:
\begin{itemize}
\item bilinear,
\item alternating: $\omega(v,v) = 0$ for all $v \in V$,
\item nondegenerate: if $\omega(u,v) = 0 $ for all $u \in V$ then $v = 0$.
\end{itemize}
\end{definition}

\noindent
There is a standard way to make $k\oplus k$ into a symplectic vector space, namely:
\[   \omega((u, v), (u',v')) =  u' v -u v'. \]
Given two symplectic vector spaces $(V_1,\omega_1)$ and $(V_2, \omega_2)$,
we give their direct sum the symplectic structure: 
\[   (\omega_1 \oplus \omega_2)((u_1,u_2), (v_1, v_2)) = \omega_1(u_1, v_1) + 
\omega_2(u_2, v_2)  .\]

\begin{definition}
For any subspace $W \subseteq V$, its \define{$\omega$-orthogonal subspace}  is $W^{\perp} = \{v\in V| \forall u \in W \; \omega(v,u) = 0 \}$. 
\end{definition}

This subspace has the property that $(W^{\perp})^{\perp} = W$ and $\dim(W) + \dim(W^{\perp}) = \dim(V)$. It is different from the usual orthogonal space in that $W\cap W^{\perp}$ need not be trivial.

\begin{definition}

 We say that a subspace $W$ is:

\begin{itemize}
\item \define{isotropic} if $W\subseteq W^{\perp}$,

\item \define{coisotropic} if $W^{\perp} \subseteq W$, 

\item \define{Lagrangian} if $W=W^{\perp}$. 

\end{itemize}
\end{definition}

\begin{definition}
Given a symplectic structure $\omega$ on a vector space $V$, we define its \define{conjugate} to be the symplectic structure $\overline\omega = -\omega$, and write the conjugate symplectic vector space $(V,\overline\omega)$ as
$\overline V.$
\end{definition}

Recall that relations $R\subseteq X\times Y$ are generalizations of functions $f\maps X\to Y$. Extending this idea we can think of Lagrangian subspaces as a type of arrow. 

\begin{definition}
A \define{Lagrangian relation} $L\maps V_1 \to V_2$ is a Lagrangian subspace $L \subseteq \overline{V_1} \oplus V_2$.
\end{definition}

Composition of Lagrangian relations is done by composing the underlying relations.  We need to take the conjugate of the first space for the composition of Lagrangian relations to be Lagrangian. This fact is well-known, though the proof is nontrivial \cite{BF}. This is enough to define a category.

\begin{definition}
Let $\define{\Lag\Rel_k}$ be the category with symplectic vector spaces $(k^{2n},\omega)$ for $n\in \mathbb{N}$ as objects and Lagrangian relations $L\maps k^{2n} \to k^{2m}$ as morphisms. 
\end{definition}

\begin{remark}
 $\Lag\Rel_k$ is also dagger compact and is equivalent to a prop by Proposition\ \ref{prop:strictification_1} with generating object $k\oplus k$. From now on we use $\Lag\Rel_k$ to mean the equivalent prop. If $L\maps V_1\to V_2$ then $f^\dagger\maps V_2\to V_1$ is the same relation with the inputs and outputs exchanged.
\end{remark}

Baez and Erbele \cite{Be} showed that $k$ can be equipped with two different extraspecial commutative Frobenius monoid structures in the category $\Fin\Rel_k$, which has finite dimensional vector spaces as objects and linear relations as morphisms. We call these the \define{duplicative Frobenius structure} on $k$ and the \define{additive Frobenius structure} on $k$.  The duplicative Frobenius structure is $k$ equipped with these linear relations:
\begin{itemize}
\item 
as comultiplication, the linear map called \define{duplication}:
\[   \begin{array}{cccl}
\Delta \maps & k &\to & k \oplus k \\
                   & \phi &\mapsto & (\phi, \phi)   
\end{array}
\]
\item 
as counit, the linear map called \define{deletion}:
\[   \begin{array}{cccl}
! \maps & k &\to & \{0\} \\
                   & \phi &\mapsto & 0   
\end{array}
\]
\item
as multiplication, the linear relation called \define{coduplication}:
\[    \begin{array}{cccl}
\Delta^\dagger \maps & k &\asrelto& k \oplus k \\
        &  \Delta^\dagger &=& \{(I, I, I) : \;  I \in k \} \subseteq k \oplus (k \oplus k)   
\end{array}
\]
\item 
as unit, the linear relation called \define{codeletion}:
\[    \begin{array}{cccl}
!^\dagger \maps & k &\asrelto& \{0\} \\
          &!^\dagger &=& \{(\phi,0)\} \subseteq k \oplus \{0\} 
\end{array}
\]
\end{itemize}

 \noindent The additive Frobenius structure is $k$ equipped with these linear relations:
\begin{itemize}
\item as multiplication, the linear map called \define{addition}:

\[    \begin{array}{cccl}
+\maps & k\oplus k &\to& k \\
            & (I_1,I_2) &\mapsto& I_1 + I_2 
\end{array}
\]
\item as unit, the linear map called \define{zero}:
\[    \begin{array}{cccl}
0\maps & \{0\} &\to& k \\
            & 0  &\mapsto & 0 
\end{array}
\]
\item as comultiplication, the linear relation called \define{coaddition}:
\[    \begin{array}{cccl}
+^\dagger \maps & k &\asrelto& k \oplus k \\
        &  +^\dagger &=& \{(I_1+I_2, I_1, I_2) : \; I_1, I_2 \in k \} \subseteq k \oplus (k \oplus k)   
\end{array}
\]
\item as counit, the linear relation called \define{cozero}:
\[    \begin{array}{cccl}
0^\dagger \maps & k &\asrelto& \{0\} \\
          &0^\dagger &=& \{(0,0)\} \subseteq k \oplus \{0\} 
\end{array}
\]
\end{itemize}

\noindent By tensoring the two extraspecial commutative Frobenius monoids in $\Fin\Rel_k$: $$(k,\Delta^\dagger, !^\dagger, \Delta, !)$$ and  $$(k,+,0,+^\dagger,0^\dagger)$$ we construct one in $\Lag\Rel_k$:
 $$(k\oplus k,\Delta^{\dagger} \oplus +, !^{\dagger} \oplus 0, \Delta \oplus +^{\dagger}, ! \oplus 0^{\dagger}).$$

 \noindent Baez and the author \cite{BC} proved there is a unique morphism of props
 $$K \maps \Fin\Corel \to \Lag\Rel_k$$ defined by $K(1) = k\oplus k$ such that $K(m) = \Delta^{\dagger} \oplus +$, $K(i) = !^{\dagger} \oplus 0$, $K(d) = \Delta \oplus +^{\dagger}$, and $K(e) =  ! \oplus 0^{\dagger}$. In other words the Frobenius monoid $(1,m,i,d,e)$ is sent to the Frobenius monoid $(k\oplus k,\Delta^{\dagger} \oplus +, !^{\dagger} \oplus 0, \Delta \oplus +^{\dagger}, ! \oplus 0^{\dagger})$. Further, $K$ is a dagger morphism of props because the dagger structure for both categories is the act of formally exchanging inputs and outputs. More explicitly, the functor does the following on generators of $\Fin\Corel$:

\[
\begin{array}{lcl}
  K(m) &=& \{(\phi_1,I_1,\phi_2,I_2,\phi_3,I_3) :\;  \phi_1= \phi_2 = \phi_3, \; I_1+ I_2=I_3\}  \\
&=&  \Delta^{\dagger} \oplus + \maps k^4\to k^2 \\ \\
  K(i) &=& \{(\phi_2, I_2) : \; I_2 = 0\}   \\
&=&  !^{\dagger} \oplus 0 \maps \{0\} \to k^2 \\ \\
  K(d) &=& \{(\phi_1,I_1,\phi_2,I_2,\phi_3,I_3) :\; \phi_1= \phi_2 = \phi_3, \;  I_1= I_2 + I_3\} \\
&=&   \Delta \oplus +^{\dagger} \maps k^2 \to k^4 \\ \\
  K(e) &=& \{(\phi_1, I_1) :\; I_1= 0\} \\
&=&  ! \oplus 0^{\dagger} \maps k^2 \to \{0\}
\end{array}
\]
This says that the potential on connected wires in a circuit must be equal, while the sum of the input current is equal to the sum of the output current for connected wires. Given a circuit in $\Circ$ and the composite $$KH'H \maps \Circ \to \Lag\Rel_k$$ we get the behavior of the circuit, which is completely determined by the underlying corelation of the circuit. In Section \ref{sec:functors} we shall see how this functor acts on $\Fin\Corel^{\circ}$, but for now we continue to study $\Lag\Rel_k$ in order to define another category with morphisms acting like bond graphs.

 A bond graph has an associated Lagrangian subspace consisting of effort and flow pairs. The junctions impose equations that determine Lagrangian subspaces, which can all be written in terms of the additive and duplicative structures. 

The junction:
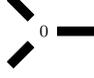
\vspace{-2ex}
\begin{figure}[H] 
	\centering
\begin{tikzpicture}[circuit ee IEC, set resistor graphic=var resistor IEC
      graphic, scale=0.8, every node/.append style={transform shape}]
[
	node distance=1.5cm,
	mynewelement/.style={
		color=blue!50!black!75,
		thick
	},
	mybondmodpoint/.style={
	rectangle,
	minimum size=3mm,
	very thick,
	draw=red!50!black!50, 
	outer sep=2pt
	}
]		
	\node (J11) {0};
	\node (R2) [ below left of=J11] {}
	edge[line width=3.5pt]    node [below]{} (J11)
        edge [line width=3.5pt]   node [above]{} (J11);
	\node (R1) [ above left of=J11] {}
	edge [line width=3.5pt]   node [below]{} (J11)
        edge [line width=3.5pt]   node [above]{} (J11);
	\node (C1) [right of=J11] {}
	edge [line width=3.5pt]   node [right]{} (J11)
        edge  [line width=3.5pt]  node [left]{} (J11);
      \end{tikzpicture} 
\end{figure}
\vspace{-1ex}
\noindent is associated with the Lagrangian subspace: $$ \Delta^{\dagger} \oplus + = \{(E_1,\ldots, F_3): \;  E_1 = E_2 = E_3, F_1 + F_2 = F_3,\} \maps k^4\to k^2 $$

\noindent The junction:
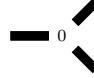
\vspace{-2ex}
\begin{figure}[H] 
	\centering
\begin{tikzpicture}[circuit ee IEC, set resistor graphic=var resistor IEC
      graphic, scale=0.8, every node/.append style={transform shape}]
[
	node distance=1.5cm,
	mynewelement/.style={
		color=blue!50!black!75,
		thick
	},
	mybondmodpoint/.style={
	rectangle,
	minimum size=3mm,
	very thick,
	draw=red!50!black!50, 
	outer sep=2pt
	}
]		
	\node (J11) {0};
	\node (R2) [ below right of=J11] {}
	edge [line width=3.5pt]   node [below]{} (J11)
        edge [line width=3.5pt]   node [above]{} (J11);
	\node (R1) [ above right of=J11] {}
	edge  [line width=3.5pt]  node [below]{} (J11)
        edge [line width=3.5pt]   node [above]{} (J11);
	\node (C1) [left of=J11] {}
	edge[line width=3.5pt]    node [above]{} (J11)
        edge  [line width=3.5pt]  node [below]{} (J11);
      \end{tikzpicture} 
\end{figure}

\vspace{-3ex}

\noindent is associated with the Lagrangian subspace: $$ \Delta \oplus +^{\dagger}= \{(E_1,\ldots, F_3): \;  E_1 = E_2 = E_3,  F_1 = F_2 + F_3\} \maps k^2\to k^4 $$

\noindent The junction:
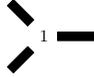
\vspace{-2ex}
\begin{figure}[H] 
	\centering
\begin{tikzpicture}[circuit ee IEC, set resistor graphic=var resistor IEC
      graphic, scale=0.8, every node/.append style={transform shape}]
[
	node distance=1.5cm,
	mynewelement/.style={
		color=blue!50!black!75,
		thick
	},
	mybondmodpoint/.style={
	rectangle,
	minimum size=3mm,
	very thick,
	draw=red!50!black!50, 
	outer sep=2pt
	}
]		
	\node (J11) {1};
	\node (R2) [ below left of=J11] {}
	edge [line width=3.5pt]   node [below]{} (J11)
        edge [line width=3.5pt]   node [above]{} (J11);
	\node (R1) [ above left of=J11] {}
	edge [line width=3.5pt]   node [below]{} (J11)
        edge [line width=3.5pt]  node [above]{} (J11);
	\node (C1) [right of=J11] {}
	edge [line width=3.5pt]   node [right]{} (J11)
        edge [line width=3.5pt]   node [left]{} (J11);
      \end{tikzpicture} 
\end{figure}

\vspace{-3ex}

\noindent is associated with the Lagrangian subspace: $$ + \oplus \Delta^{\dagger}=\{(E_1,\ldots, F_3): \; E_1 + E_2 = E_3, F_1 = F_2 = F_3\} \maps k^4\to k^2 $$

\noindent The junction:
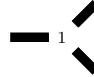
\vspace{-2ex}
\begin{figure}[H] 
	\centering
\begin{tikzpicture}[circuit ee IEC, set resistor graphic=var resistor IEC
      graphic, scale=0.8, every node/.append style={transform shape}]
[
	node distance=1.5cm,
	mynewelement/.style={
		color=blue!50!black!75,
		thick
	},
	mybondmodpoint/.style={
	rectangle,
	minimum size=3mm,
	very thick,
	draw=red!50!black!50, 
	outer sep=2pt
	}
]		
	\node (J11) {1};
	\node (R2) [ below right of=J11] {}
	edge[line width=3.5pt]    node [below]{} (J11)
        edge  [line width=3.5pt]  node [above]{} (J11);
	\node (R1) [ above right of=J11] {}
	edge [line width=3.5pt]   node [below]{} (J11)
        edge[line width=3.5pt]    node [above]{} (J11);
	\node (C1) [left of=J11] {}
	edge [line width=3.5pt]   node [above]{} (J11)
        edge [line width=3.5pt]   node [below]{} (J11);
      \end{tikzpicture} 
\end{figure}

\vspace{-3ex}

\noindent is associated with the Lagrangian subspace: $$+^{\dagger} \oplus \Delta=\{(E_1,\ldots, F_3): \; E_1 = E_2 + E_3, F_1 = F_2 = F_3\} \maps k^2\to k^4 $$

 Now notice that $(k\oplus k,\Delta^{\dagger} \oplus +, !^{\dagger} \oplus 0,  \Delta \oplus +^{\dagger}, ! \oplus 0^{\dagger})$ is an extraspecial commutative Frobenius monoid and its multiplication and comultiplication correspond to the two types of $0$-junctions. This hints at a way to define unary $0$-junctions corresponding to the unit and counit. 

We associate the unit: $$!^{\dagger} \oplus 0 = \{(E, F) : \; F= 0\} \maps \{0\} \to k^2 $$
\noindent to the bond graph:
\vspace{-2ex}
\begin{figure}[H] 
	\centering
\begin{tikzpicture}[circuit ee IEC, set resistor graphic=var resistor IEC
      graphic, scale=0.8, every node/.append style={transform shape}]
[
	node distance=1.5cm,
	mynewelement/.style={
		color=blue!50!black!75,
		thick
	},
	mybondmodpoint/.style={
	rectangle,
	minimum size=3mm,
	very thick,
	draw=red!50!black!50, 
	outer sep=2pt
	}
]		
	\node (J11) {0};
	\node (R1) [right of=J11] {}
	edge  [line width=3.5pt]  node [below]{} (J11)
        edge  [line width=3.5pt]  node [above]{} (J11);
\end{tikzpicture}
\end{figure}

\vspace{-1ex}

\noindent The counit: $$! \oplus 0^{\dagger}= \{(E, F) : \; F= 0\} \maps k^2\to \{0\} $$
\noindent is associated with:
\vspace{-2ex}
\begin{figure}[H] 
	\centering
\begin{tikzpicture}[circuit ee IEC, set resistor graphic=var resistor IEC
      graphic, scale=0.8, every node/.append style={transform shape}]
[
	node distance=1.5cm,
	mynewelement/.style={
		color=blue!50!black!75,
		thick
	},
	mybondmodpoint/.style={
	rectangle,
	minimum size=3mm,
	very thick,
	draw=red!50!black!50, 
	outer sep=2pt
	}
]		
	\node (J11) {0};
	\node (R2) [left of=J11] {}
	edge [line width=3.5pt]  node [below]{} (J11)
        edge [line width=3.5pt]   node [above]{} (J11);
\end{tikzpicture}
\end{figure}

\vspace{-1ex}

 By tensoring we construct another extraspecial commutative Frobenius monoid $(k\oplus k,+ \oplus \Delta^{\dagger},  0 \oplus !^{\dagger}, +^{\dagger} \oplus \Delta,  0^{\dagger} \oplus !)$. We use this one to define unary $1$-junctions.

The unit: $$0 \; \oplus \; !^{\dagger} = \{(E, F) : \; E= 0\} \maps \{0\} \to k^2 $$
\noindent is drawn as:
\vspace{-2ex}
\begin{figure}[H] 
	\centering
\begin{tikzpicture}[circuit ee IEC, set resistor graphic=var resistor IEC
      graphic, scale=0.8, every node/.append style={transform shape}]
[
	node distance=1.5cm,
	mynewelement/.style={
		color=blue!50!black!75,
		thick
	},
	mybondmodpoint/.style={
	rectangle,
	minimum size=3mm,
	very thick,
	draw=red!50!black!50, 
	outer sep=2pt
	}
]		
	\node (J11) {1};
	\node (R1) [right of=J11] {}
	edge [line width=3.5pt]   node [below]{} (J11)
        edge [line width=3.5pt]   node [above]{} (J11);
\end{tikzpicture}
\end{figure}

\vspace{-1ex}

\noindent The counit: $$0^{\dagger} \; \oplus \; ! = \{(E, F) : \; E= 0\} \maps k^2\to \{0\} $$
\noindent is drawn as:
\vspace{-2ex}
\begin{figure}[H] 
	\centering
\begin{tikzpicture}[circuit ee IEC, set resistor graphic=var resistor IEC
      graphic, scale=0.8, every node/.append style={transform shape}]
[
	node distance=1.5cm,
	mynewelement/.style={
		color=blue!50!black!75,
		thick
	},
	mybondmodpoint/.style={
	rectangle,
	minimum size=3mm,
	very thick,
	draw=red!50!black!50, 
	outer sep=2pt
	}
]		
	\node (J11) {1};
	\node (R2) [left of=J11] {}
	edge [line width=3.5pt]  node [below]{} (J11)
        edge [line width=3.5pt]    node [above]{} (J11);
\end{tikzpicture}
\end{figure}

\vspace{-1ex}

Since these Frobenius monoids are built using the equations that define $0$- and $1$-junctions it is no surprise that they correspond to these junctions. However, when the morphisms from the two Frobenius monoids interact something unexepcted occurs. They obey relations which are too strong. 

\begin{theorem} 
\label{thm:lagrel}
We have two extraspecial commutative Frobenius monoids
\begin{itemize}
 \item $(k\oplus k,\Delta^{\dagger} \oplus +, !^{\dagger} \oplus 0, \Delta \oplus +^{\dagger}, ! \oplus 0^{\dagger})$ 
 \item $(k\oplus k,+ \oplus \Delta^{\dagger}, 0 \oplus !^{\dagger}, +^{\dagger} \oplus \Delta,  0^{\dagger} \oplus !)$ 
\end{itemize}
\noindent which come together as two bimonoids
\begin{itemize}
\item $(k\oplus k,\Delta^{\dagger} \oplus +, !^{\dagger} \oplus 0, +^{\dagger} \oplus \Delta, 0^{\dagger} \oplus !)$
\item $(k\oplus k,+ \oplus \Delta^{\dagger}, 0 \oplus !^{\dagger}, \Delta \oplus +^{\dagger}, ! \oplus 0^{\dagger}).$
\end{itemize}
\noindent Additionally: $$(\Delta^{\dagger} \oplus +)\circ (+^{\dagger} \oplus \Delta) = ((+ \oplus \Delta^{\dagger})\circ (\Delta \oplus +^{\dagger}))^{-1}$$
\end{theorem}
\begin{proof}
These facts follow from relations shown by Bonchi, Soboci\'nski, and Zanasi \cite{Be}.
\end{proof}

In a bimonoid there are laws governing both the interaction of the unit and comulitiplication and  the interaction of the counit and multiplication.  These two laws do not describe the way that a unary $1$-junction would interact with a $0$-junction or vice-versa. Regardless, we define another category using the above ideas.

\begin{definition}
Let $\Lag\Rel_k^{\circ}$ be the subcategory of $\Lag\Rel_k$ generated by the $8$ morphisms:
$$(\Delta^{\dagger} \oplus +, !^{\dagger} \oplus 0, \Delta \oplus +^{\dagger},  ! \oplus 0^{\dagger}, + \oplus \Delta^{\dagger},  0 \oplus !^{\dagger}, +^{\dagger} \oplus \Delta,  0^{\dagger} \oplus !)$$ Note that $\Lag\Rel_k^{\circ}$ is equivalent to a prop.  Henceforth we refer to this prop as $LagRel^\circ_k$.
\end{definition}

\begin{conjecture}
We conjecture that $\Lag\Rel_k^{\circ}$ is presented by generators and equations corresponding to Theorem \ref{thm:lagrel}.
\end{conjecture}

To summarize the two approaches we say that $\Fin\Corel^{\circ}$ correctly captures the interaction between $1$- and $0$-junctions, while $\Lag\Rel_k^{\circ}$ correctly captures the behavior of each junction separately. 


\section{Functors for Bond Graphs}
\label{sec:functors}

Using the framework of props we can easily describe a category having only the relations that appear in both categories. 
\begin{definition}
Let $\BondGraph$ be the prop generated by the $8$ morphisms:
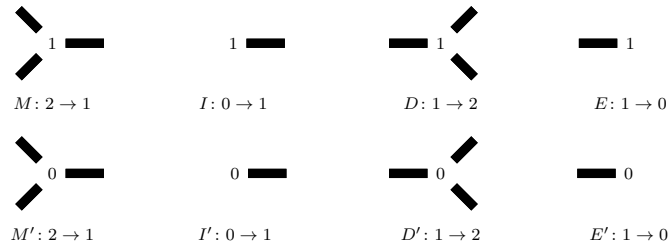
\[
  \xymatrixrowsep{1pt}
  \xymatrixcolsep{25pt}
  \xymatrix@1{

\\
}
\]

such that:
\begin{itemize}
\item $(1,M,I,D,E)$ is an extraspecial symmetric Frobenius monoid.
\end{itemize}

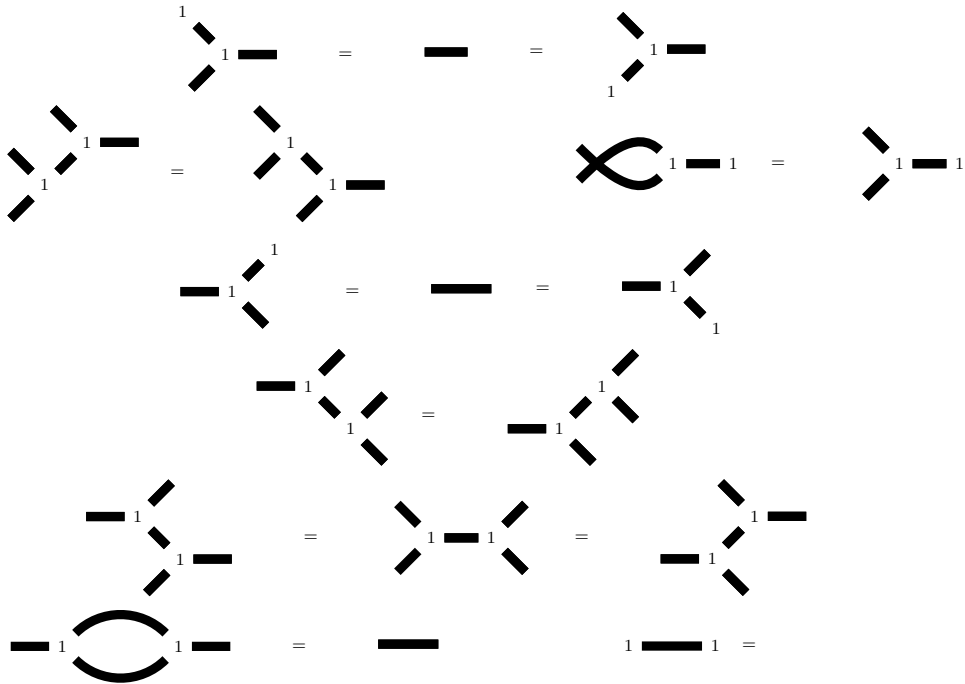
\vspace{-3ex}


\[
  \xymatrixrowsep{1pt}
  \xymatrixcolsep{15pt}
  \xymatrix@1{
\begin{aligned}
%
 
\end{aligned}
  }
\]

\begin{itemize}
\item $(1,M',I',D', E')$ is an extraspecial symmetric Frobenius monoid.
\end{itemize}

\vspace{-3ex}

\[
  \xymatrixrowsep{5pt}
  \xymatrixcolsep{15pt}
  \xymatrix@1{
\begin{aligned}
%
 
  }
\]

\noindent is idempotent.
\end{definition}
Further, we equip $\BondGraph$ with a dagger structure where if $f\maps m\to n$ is a morphism then $f^\dagger\maps n\to m$ is the vertical reflection of $f$. This is the unique way to make this prop into a symmetric monoidal dagger category such that  \[   M^\dagger = D, \quad I^\dagger = E, \quad M'^\dagger = D', \quad I'^\dagger = E'.\] 

We take the composite of $K\maps \Fin\Corel \to \Lag\Rel_k$ and the inclusion functor $i \maps \Fin\Corel^{\circ} \to \Fin\Corel$ which results in

 \[    Ki \maps \Fin\Corel^{\circ} \to \Lag\Rel_k  .\]

\noindent We may describe this functor in a simple way using props.

\begin{proposition}\label{prop:blackbox}
$K i \maps \Fin\Corel^{\circ} \to \Lag\Rel_k $ is a symmetric monoidal dagger functor that is determined by the following:

\[\arraycolsep=1pt
\begin{array}{lcl}
  Ki(m_2) &=& \{(\phi_1,\ldots\thinspace , I_6): \; \phi_1= \phi_5, I_1 = I_5,  \phi_4 = \phi_6, I_4 = I_6, \phi_2 = \phi_3, I_2 + I_3 = 0\}  \\
 & & \maps k^8\to k^4 \\ \\
  Ki(i_2) &=& \{(\phi_1, I_1, \phi_2,I_2) : \; I_1 +I_2 = 0, \phi_1=\phi_2\}   \\
 & & \maps \{0\}\to k^4 \\ \\
  Ki(\mu_2) &=& \{(\phi_1,\ldots\thinspace , I_6): \; \phi_1= \phi_3 =\phi_5,\phi_2= \phi_4 =\phi_6, I_1+I_3=I_5, I_2+I_4=I_6\}  \\
 & & \maps k^8\to k^4 \\ \\
  Ki(\iota_2) &=& \{(\phi_1, I_1,\phi_2,I_2) : \; I_1 = 0,I_2 =0\}   \\
 & & \maps \{0\}\to k^4 \\ 
\end{array}
\]

\end{proposition}

\begin{proof}
$K i$ is a symmetric monoidal dagger functor since $K$ is. It is tedious, but not difficult, to check that $K i$ does act on the morphisms $m_2,i_2,\mu_2,\iota_2$ in the above way. We just write them in terms of $m,d,i,e$ and then compose the results in $\Lag\Rel_k$. Since $Ki$ is a dagger functor the following is a consequence:

\[\arraycolsep=1pt
\begin{array}{lcl}
  Ki(d_2) &=& \{(\phi_1,\ldots\thinspace , I_6): \; \phi_1= \phi_3, I_1 = I_3,  \phi_2 = \phi_6, I_2 = I_6, \phi_4 = \phi_4, I_4 + I_5 = 0\} \\
 & & \maps k^4\to k^8 \\ \\
  Ki(e_2) &=& \{(\phi_1, I_1, \phi_2,I_2) :\; I_1 +I_2= 0, \phi_1 = \phi_2\} \\
 & & \maps k^4\to \{0\} \\ \\
  Ki(\delta_2) &=& \{(\phi_1,\ldots\thinspace , I_6): \; \phi_1= \phi_3 =\phi_5,\phi_2= \phi_4 =\phi_5, I_1=I_3+ I_5, I_2=I_4+ I_6\} \\
 & & \maps k^4\to k^8 \\ \\
  Ki(\epsilon_2) &=& \{(\phi_1, I_1, \phi_2,I_2) :\; I_1= 0, I_2 = 0\} \\
 & & \maps k^4\to \{0\}. \\
\end{array}
\]

\noindent The only thing to show is that the above $8$ morphisms are enough to define $Ki\maps \Fin\Corel^{\circ} \to \Lag\Rel_k $. We have that $(k\oplus k,\Delta^{\dagger} \oplus +, \Delta \oplus +^{\dagger}, !^{\dagger} \oplus 0,  ! \oplus 0^{\dagger})$ is an extraspecial commutative Frobenius monoid in $\Lag\Rel_k$. Just as we built two monoids on $2$ from $1\in \Fin\Corel$, we use the same constructions for $k\oplus k$, to give us two monoid structures on $(k\oplus k)^2$. The subspaces associated to these monoids are the same subspaces which our morphisms are sent to.

In other words the two monoids are: $$((k\oplus k)^2,Ki(m_2), Ki(i_2),Ki(d_2),Ki(e_2))$$ and $$ ((k\oplus k)^2,Ki(\mu_2), Ki(\iota_2),Ki(\delta_2),Ki(\epsilon_2)).$$

In the same way we figured out the additional equations for the morphisms associated to the monoids in $\Fin\Corel$ one can show that these subspaces obey at least the same equations. There may be more equations, but this does not matter. Similarly, we have the weak bimonoid laws and the additional laws given in \ref{thm:weakbimonoid}. Since all of the equations hold, this defines a functor using the equivalent props.
\end{proof}

So far we have the following diagram of symmetric monoidal dagger functors between props:
\[
\begin{tikzcd}[column sep=scriptsize]
 \Lag\Rel_k^{\circ} \arrow[r, ""{name=U, below}, "i'"]{}
 & \Lag\Rel \\
\Fin\Corel^{\circ}  \arrow[r, ""{name=U, below}, "i"]{} & \Fin\Corel  \arrow[u, "K"'] 
\end{tikzcd}
\]
By using $\BondGraph$ we can extend this to a diagram that commutes up to a natural transformation.

\begin{proposition}\label{prop:functors}
There exist unique dagger morphisms of props $G\maps \BondGraph \to \Fin\Corel^{\circ}$ and $F\maps \BondGraph \to \Lag\Rel_k^{\circ}$ defined by:
\begin{itemize}
\item $G(M) = m_2$
\end{itemize}
\vspace{-3ex}
\[
  \xymatrixrowsep{5pt}
  \xymatrix@1{
 *+[u]{
\begin{tikzpicture}[circuit ee IEC, set resistor graphic=var resistor IEC
      graphic, scale=0.8, every node/.append style={transform shape}]
[
	node distance=1.5cm,
	mynewelement/.style={
		color=blue!50!black!75,
		thick
	},
	mybondmodpoint/.style={
	rectangle,
	minimum size=3mm,
	very thick,
	draw=red!50!black!50, 
	outer sep=2pt
	}
]		
\node(G) at +(-1,0) {\large $G\maps$};
	\node(J11) {$\mathrm{1}$};
	\node (R2) [ below left of=J11] {}
	edge  [line width=3.5pt]   node [below]{} (J11)
        edge  [line width=3.5pt]   node [above]{} (J11);
	\node (R1) [ above left of=J11] {}
	edge [line width=3.5pt]    node [below]{} (J11)
        edge  [line width=3.5pt]   node [above]{} (J11);
	\node (C1) [right of=J11] {}
	edge [line width=3.5pt]    node [right]{} (J11)
        edge [line width=3.5pt]    node [left]{} (J11);
      \end{tikzpicture} 
}
 \ar@{|->}@<.25ex>[r] & \quad \monadmult{.1\textwidth}
  }
\]
\vspace{-4.5ex}
\begin{itemize}
\item $G(I) =  i_2$
\end{itemize}
\vspace{-3ex}
\[
  \xymatrixrowsep{5pt}
  \xymatrix@1{
 *+[u]{
\begin{tikzpicture}[circuit ee IEC, set resistor graphic=var resistor IEC
      graphic, scale=0.8, every node/.append style={transform shape}]
[
	node distance=1.5cm,
	mynewelement/.style={
		color=blue!50!black!75,
		thick
	},
	mybondmodpoint/.style={
	rectangle,
	minimum size=3mm,
	very thick,
	draw=red!50!black!50, 
	outer sep=2pt
	}
]		
\node(G) at +(-1,0) {\large $G\maps$};
	\node (J11) {$\mathrm{1}$};
	\node (R1) [right of=J11] {}
	edge  [line width=3.5pt]   node [below]{} (J11)
        edge  [line width=3.5pt]   node [above]{} (J11);
	\node (D1) [ above left of=J11] {};
	\node (D2) [ below left of=J11] {};
      \end{tikzpicture} 
	}
  \ar@{|->}@<.25ex>[r] & \quad \captwo{.1\textwidth}
  }
\]
\vspace{-4.5ex}
\begin{itemize}
\item $G(M') = \mu_2$
\end{itemize}
\vspace{-3ex}
\[
  \xymatrixrowsep{5pt}
  \xymatrix@1{
 *+[u]{
\begin{tikzpicture}[circuit ee IEC, set resistor graphic=var resistor IEC
      graphic, scale=0.8, every node/.append style={transform shape}]
[
	node distance=1.5cm,
	mynewelement/.style={
		color=blue!50!black!75,
		thick
	},
	mybondmodpoint/.style={
	rectangle,
	minimum size=3mm,
	very thick,
	draw=red!50!black!50, 
	outer sep=2pt
	}
]		
\node(G) at +(-1,0) {\large $G\maps$};
	\node(J11) {$\mathrm{0}$};
	\node (R2) [ below left of=J11] {}
	edge [line width=3.5pt]    node [below]{} (J11)
        edge   [line width=3.5pt]  node [above]{} (J11);
	\node (R1) [ above left of=J11] {}
	edge [line width=3.5pt]   node [below]{} (J11)
        edge  [line width=3.5pt]   node [above]{} (J11);
	\node (C1) [right of=J11] {}
	edge [line width=3.5pt]    node [right]{} (J11)
        edge  [line width=3.5pt]   node [left]{} (J11);
      \end{tikzpicture} 
}
  \ar@{|->}@<.25ex>[r] &\quad \parmult{.1\textwidth}
  }
\]
\vspace{-4.5ex}
\begin{itemize}
\item $G(I') =   \iota_2$
\end{itemize}
\vspace{-3ex}
\[
  \xymatrixrowsep{5pt}
  \xymatrix@1{
 *+[u]{
\begin{tikzpicture}[circuit ee IEC, set resistor graphic=var resistor IEC
      graphic, scale=0.8, every node/.append style={transform shape}]
[
	node distance=1.5cm,
	mynewelement/.style={
		color=blue!50!black!75,
		thick
	},
	mybondmodpoint/.style={
	rectangle,
	minimum size=3mm,
	very thick,
	draw=red!50!black!50, 
	outer sep=2pt
	}
]		
\node(G) at +(-1,0) {\large $G\maps$};
	\node (J11) {$\mathrm{0}$};
	\node (R1) [right of=J11] {}
	edge [line width=3.5pt]    node [below]{} (J11)
        edge [line width=3.5pt]    node [above]{} (J11);
	\node (D1) [ above left of=J11] {};
	\node (D2) [ below left of=J11] {};
      \end{tikzpicture} 
	}
  \ar@{|->}@<.25ex>[r] & \quad \parunit{.1\textwidth}
  }
\]
\vspace{-4.5ex}
\begin{itemize}
\item $F(M) = + \oplus \Delta^{\dagger}$
\end{itemize}
\vspace{-3ex}
\[
  \xymatrixrowsep{5pt}
  \xymatrix@1{
 *+[u]{
\begin{tikzpicture}[circuit ee IEC, set resistor graphic=var resistor IEC
      graphic, scale=0.8, every node/.append style={transform shape}]
[
	node distance=1.5cm,
	mynewelement/.style={
		color=blue!50!black!75,
		thick
	},
	mybondmodpoint/.style={
	rectangle,
	minimum size=3mm,
	very thick,
	draw=red!50!black!50, 
	outer sep=2pt
	}
]		
\node(F) at +(-1,0) {\large $F\maps$};
	\node(J11) {$\mathrm{1}$};
	\node (R2) [ below left of=J11] {}
	edge [line width=3.5pt]    node [below]{} (J11)
        edge  [line width=3.5pt]   node [above]{} (J11);
	\node (R1) [ above left of=J11] {}
	edge [line width=3.5pt]    node [below]{} (J11)
        edge [line width=3.5pt]    node [above]{} (J11);
	\node (C1) [right of=J11] {}
	edge  [line width=3.5pt]   node [right]{} (J11)
        edge [line width=3.5pt]    node [left]{} (J11);
      \end{tikzpicture} 
}
  \ar@{|->}@<.25ex>[r] &  \quad  + \oplus \Delta^{\dagger}
  }
\]
\vspace{-4.5ex}
\begin{itemize}
\item $F(I) =  0 \; \oplus \; !^{\dagger} $
\end{itemize}
\vspace{-3ex}
\[
  \xymatrixrowsep{5pt}
  \xymatrix@1{
 *+[u]{
\begin{tikzpicture}[circuit ee IEC, set resistor graphic=var resistor IEC
      graphic, scale=0.8, every node/.append style={transform shape}]
[
	node distance=1.5cm,
	mynewelement/.style={
		color=blue!50!black!75,
		thick
	},
	mybondmodpoint/.style={
	rectangle,
	minimum size=3mm,
	very thick,
	draw=red!50!black!50, 
	outer sep=2pt
	}
]		
\node(F) at +(-1,0) {\large $F\maps$};
	\node (J11) {$\mathrm{1}$};
	\node (R1) [right of=J11] {}
	edge  [line width=3.5pt]   node [below]{} (J11)
        edge [line width=3.5pt]    node [above]{} (J11);
	\node (D1) [ above left of=J11] {};
	\node (D2) [ below left of=J11] {};
      \end{tikzpicture} 
	}
  \ar@{|->}@<.25ex>[r] & \quad 0 \; \oplus \; !^{\dagger} 
  }
\]
\vspace{-4.5ex}
\begin{itemize}
\item $F(M') = \Delta^{\dagger} \oplus + $
\end{itemize}
\vspace{-3ex}
\[
  \xymatrixrowsep{5pt}
  \xymatrix@1{
 *+[u]{
\begin{tikzpicture}[circuit ee IEC, set resistor graphic=var resistor IEC
      graphic, scale=0.8, every node/.append style={transform shape}]
[
	node distance=1.5cm,
	mynewelement/.style={
		color=blue!50!black!75,
		thick
	},
	mybondmodpoint/.style={
	rectangle,
	minimum size=3mm,
	very thick,
	draw=red!50!black!50, 
	outer sep=2pt
	}
]		
\node(F) at +(-1,0) {\large $F\maps$};
	\node(J11) {$\mathrm{0}$};
	\node (R2) [ below left of=J11] {}
	edge  [line width=3.5pt]   node [below]{} (J11)
        edge [line width=3.5pt]   node [above]{} (J11);
	\node (R1) [ above left of=J11] {}
	edge [line width=3.5pt]    node [below]{} (J11)
        edge [line width=3.5pt]    node [above]{} (J11);
	\node (C1) [right of=J11] {}
	edge [line width=3.5pt]    node [right]{} (J11)
        edge [line width=3.5pt]   node [left]{} (J11);
      \end{tikzpicture} 
}
  \ar@{|->}@<.25ex>[r] & \quad \Delta^{\dagger} \oplus +
  }
\]
\vspace{-4.5ex}
\begin{itemize}
\item $F(I') =   !^{\dagger} \oplus 0$
\end{itemize}
\vspace{-3ex}
\[
  \xymatrixrowsep{5pt}
  \xymatrix@1{
 *+[u]{
\begin{tikzpicture}[circuit ee IEC, set resistor graphic=var resistor IEC
      graphic, scale=0.8, every node/.append style={transform shape}]
[
	node distance=1.5cm,
	mynewelement/.style={
		color=blue!50!black!75,
		thick
	},
	mybondmodpoint/.style={
	rectangle,
	minimum size=3mm,
	very thick,
	draw=red!50!black!50, 
	outer sep=2pt
	}
]		
\node(F) at +(-1,0) {\large $F\maps$};
	\node (J11) {$\mathrm{0}$};
	\node (R1) [right of=J11] {}
	edge  [line width=3.5pt]   node [below]{} (J11)
        edge [line width=3.5pt]    node [above]{} (J11);
	\node (D1) [ above left of=J11] {};
	\node (D2) [ below left of=J11] {};
      \end{tikzpicture} 
	}
 \ar@{|->}@<.25ex>[r] & \quad !^{\dagger} \oplus 0
  }
\]
\end{proposition}
\label{prop:functor}
\vspace{-4.5ex}

\begin{proof}
We first prove uniqueness of $G$. Suppose that $G$ is a dagger morphism of props such that $G(M) = m_2$,  $G(I) =  i_2$,  $G(M') = \mu_2$, and $G(I') =   \iota_2$ as pictured above. Note that since $G$ is a dagger functor,  $G$ must do the following to the other $4$ generators:

\begin{itemize}
\item $G(D) = d_2$
\end{itemize}
\vspace{-3ex}
\[
  \xymatrixrowsep{5pt}
  \xymatrix@1{
 *+[u]{
\begin{tikzpicture}[circuit ee IEC, set resistor graphic=var resistor IEC
      graphic, scale=0.8, every node/.append style={transform shape}]
[
	node distance=1.5cm,
	mynewelement/.style={
		color=blue!50!black!75,
		thick
	},
	mybondmodpoint/.style={
	rectangle,
	minimum size=3mm,
	very thick,
	draw=red!50!black!50, 
	outer sep=2pt
	}
]		
\node(G) at +(-1.5,0) {\large $G\maps$};
	\node(J11) {$\mathrm{1}$};
	\node (R2) [ below right of=J11] {}
	edge  [line width=3.5pt]   node [below]{} (J11)
        edge  [line width=3.5pt]   node [above]{} (J11);
	\node (R1) [ above right of=J11] {}
	edge [line width=3.5pt]    node [below]{} (J11)
        edge  [line width=3.5pt]   node [above]{} (J11);
	\node (C1) [left of=J11] {}
	edge [line width=3.5pt]    node [right]{} (J11)
        edge [line width=3.5pt]    node [left]{} (J11);
      \end{tikzpicture} 
}
 \ar@{|->}@<.25ex>[r] & \quad \monadcomult{.1\textwidth}
  }
\]
\vspace{-4.5ex}
\begin{itemize}
\item $G(E) =  e_2$
\end{itemize}
\vspace{-3ex}
\[
  \xymatrixrowsep{5pt}
  \xymatrix@1{
 *+[u]{
\begin{tikzpicture}[circuit ee IEC, set resistor graphic=var resistor IEC
      graphic, scale=0.8, every node/.append style={transform shape}]
[
	node distance=1.5cm,
	mynewelement/.style={
		color=blue!50!black!75,
		thick
	},
	mybondmodpoint/.style={
	rectangle,
	minimum size=3mm,
	very thick,
	draw=red!50!black!50, 
	outer sep=2pt
	}
]		
\node(G) at +(-1.5,0) {\large $G\maps$};
	\node (J11) {$\mathrm{1}$};
	\node (R1) [left of=J11] {}
	edge  [line width=3.5pt]   node [below]{} (J11)
        edge  [line width=3.5pt]   node [above]{} (J11);
	\node (D1) [ above right of=J11] {};
	\node (D2) [ below right of=J11] {};
      \end{tikzpicture} 
	}
  \ar@{|->}@<.25ex>[r] & \quad \cuptwo{.1\textwidth}
  }
\]
\vspace{-4.5ex}
\begin{itemize}
\item $G(D') = \delta_2$
\end{itemize}
\vspace{-3ex}
\[
  \xymatrixrowsep{5pt}
  \xymatrix@1{
 *+[u]{
\begin{tikzpicture}[circuit ee IEC, set resistor graphic=var resistor IEC
      graphic, scale=0.8, every node/.append style={transform shape}]
[
	node distance=1.5cm,
	mynewelement/.style={
		color=blue!50!black!75,
		thick
	},
	mybondmodpoint/.style={
	rectangle,
	minimum size=3mm,
	very thick,
	draw=red!50!black!50, 
	outer sep=2pt
	}
]		
\node(G) at +(-1.5,0) {\large $G\maps$};
	\node(J11) {$\mathrm{0}$};
	\node (R2) [ below right of=J11] {}
	edge [line width=3.5pt]    node [below]{} (J11)
        edge   [line width=3.5pt]  node [above]{} (J11);
	\node (R1) [ above right of=J11] {}
	edge [line width=3.5pt]   node [below]{} (J11)
        edge  [line width=3.5pt]   node [above]{} (J11);
	\node (C1) [left of=J11] {}
	edge [line width=3.5pt]    node [right]{} (J11)
        edge  [line width=3.5pt]   node [left]{} (J11);
      \end{tikzpicture} 
}
  \ar@{|->}@<.25ex>[r] &\quad \parcomult{.1\textwidth}
  }
\]
\vspace{-4.5ex}
\begin{itemize}
\item $G(E') = \epsilon_2$
\end{itemize}
\vspace{-3ex}
\[
  \xymatrixrowsep{5pt}
  \xymatrix@1{
 *+[u]{
\begin{tikzpicture}[circuit ee IEC, set resistor graphic=var resistor IEC
      graphic, scale=0.8, every node/.append style={transform shape}]
[
	node distance=1.5cm,
	mynewelement/.style={
		color=blue!50!black!75,
		thick
	},
	mybondmodpoint/.style={
	rectangle,
	minimum size=3mm,
	very thick,
	draw=red!50!black!50, 
	outer sep=2pt
	}
]		
\node(G) at +(-1.5,0) {\large $G\maps$};
	\node (J11) {$\mathrm{0}$};
	\node (R1) [left of=J11] {}
	edge [line width=3.5pt]    node [below]{} (J11)
        edge [line width=3.5pt]    node [above]{} (J11);
	\node (D1) [ above right of=J11] {};
	\node (D2) [ below right of=J11] {};
      \end{tikzpicture} 
	}
 \ar@{|->}@<.25ex>[r] & \quad \parcounit{.1\textwidth}
  }
\]
\vspace{-3ex}

We know what $G$ does to all $8$ generators, so $G$ is the unique morphism of props acting in this way on those generators.  Similarly $F$ is the unique morphism of props such that $F(M) = + \oplus \Delta^{\dagger}$, $F(I) =   0 \oplus !^{\dagger}$, $F(M') = \Delta^{\dagger} \oplus +$, and $F(I') =  !^{\dagger} \oplus 0$. On the other generators, $F$ acts as follows:

\begin{itemize}
\item  $F(D) = \Delta \oplus +^{\dagger}$
\end{itemize}
\vspace{-3ex}
\[
  \xymatrixrowsep{5pt}
  \xymatrix@1{
 *+[u]{
\begin{tikzpicture}[circuit ee IEC, set resistor graphic=var resistor IEC
      graphic, scale=0.8, every node/.append style={transform shape}]
[
	node distance=1.5cm,
	mynewelement/.style={
		color=blue!50!black!75,
		thick
	},
	mybondmodpoint/.style={
	rectangle,
	minimum size=3mm,
	very thick,
	draw=red!50!black!50, 
	outer sep=2pt
	}
]		
\node(F) at +(-1.5,0) {\large $F\maps$};
	\node(J11) {$\mathrm{1}$};
	\node (R2) [ below right of=J11] {}
	edge [line width=3.5pt]    node [below]{} (J11)
        edge  [line width=3.5pt]   node [above]{} (J11);
	\node (R1) [ above right of=J11] {}
	edge [line width=3.5pt]    node [below]{} (J11)
        edge [line width=3.5pt]    node [above]{} (J11);
	\node (C1) [left of=J11] {}
	edge  [line width=3.5pt]   node [right]{} (J11)
        edge [line width=3.5pt]    node [left]{} (J11);
      \end{tikzpicture} 
}
  \ar@{|->}@<.25ex>[r] &  \quad  \Delta \oplus +^{\dagger}
  }
\]
\vspace{-4.5ex}
\begin{itemize}
\item $F(E) =   ! \oplus 0^{\dagger}$
\end{itemize}
\vspace{-3ex}
\[
  \xymatrixrowsep{5pt}
  \xymatrix@1{
 *+[u]{
\begin{tikzpicture}[circuit ee IEC, set resistor graphic=var resistor IEC
      graphic, scale=0.8, every node/.append style={transform shape}]
[
	node distance=1.5cm,
	mynewelement/.style={
		color=blue!50!black!75,
		thick
	},
	mybondmodpoint/.style={
	rectangle,
	minimum size=3mm,
	very thick,
	draw=red!50!black!50, 
	outer sep=2pt
	}
]		
\node(F) at +(-1.5,0) {\large $F\maps$};
	\node (J11) {$\mathrm{1}$};
	\node (R1) [left of=J11] {}
	edge  [line width=3.5pt]   node [below]{} (J11)
        edge [line width=3.5pt]    node [above]{} (J11);
	\node (D1) [ above right of=J11] {};
	\node (D2) [ below right of=J11] {};
      \end{tikzpicture} 
	}
  \ar@{|->}@<.25ex>[r] &  \quad ! \oplus 0^{\dagger}
  }
\]
\vspace{-4.5ex}
\begin{itemize}
\item $F(D') = +^{\dagger} \oplus \Delta$
\end{itemize}
\vspace{-3ex}
\[
  \xymatrixrowsep{5pt}
  \xymatrix@1{
 *+[u]{
\begin{tikzpicture}[circuit ee IEC, set resistor graphic=var resistor IEC
      graphic, scale=0.8, every node/.append style={transform shape}]
[
	node distance=1.5cm,
	mynewelement/.style={
		color=blue!50!black!75,
		thick
	},
	mybondmodpoint/.style={
	rectangle,
	minimum size=3mm,
	very thick,
	draw=red!50!black!50, 
	outer sep=2pt
	}
]		
\node(F) at +(-1.5,0) {\large $F\maps$};
	\node(J11) {$\mathrm{0}$};
	\node (R2) [ below right of=J11] {}
	edge  [line width=3.5pt]   node [below]{} (J11)
        edge [line width=3.5pt]   node [above]{} (J11);
	\node (R1) [ above right of=J11] {}
	edge [line width=3.5pt]    node [below]{} (J11)
        edge [line width=3.5pt]    node [above]{} (J11);
	\node (C1) [left of=J11] {}
	edge [line width=3.5pt]    node [right]{} (J11)
        edge [line width=3.5pt]   node [left]{} (J11);
      \end{tikzpicture} 
}
 \ar@{|->}@<.25ex>[r] &  \quad +^{\dagger} \oplus \Delta
  }
\]
\vspace{-4.5ex}
\begin{itemize}
\item $F(E') =  0^{\dagger} \oplus !$
\end{itemize}
\vspace{-3ex}
\[
  \xymatrixrowsep{5pt}
  \xymatrix@1{
 *+[u]{
\begin{tikzpicture}[circuit ee IEC, set resistor graphic=var resistor IEC
      graphic, scale=0.8, every node/.append style={transform shape}]
[
	node distance=1.5cm,
	mynewelement/.style={
		color=blue!50!black!75,
		thick
	},
	mybondmodpoint/.style={
	rectangle,
	minimum size=3mm,
	very thick,
	draw=red!50!black!50, 
	outer sep=2pt
	}
]		
\node(F) at +(-1.5,0) {\large $F\maps$};
	\node (J11) {$\mathrm{0}$};
	\node (R1) [left of=J11] {}
	edge  [line width=3.5pt]   node [below]{} (J11)
        edge [line width=3.5pt]    node [above]{} (J11);
	\node (D1) [ above right of=J11] {};
	\node (D2) [ below right of=J11] {};
      \end{tikzpicture} 
	}
  \ar@{|->}@<.25ex>[r] & \quad 0^{\dagger} \; \oplus \; !
  }
\]
\vspace{-3ex}

Next we show that these dagger morphisms of props exist. Let $G$ be defined on the generators of $\BondGraph$ as described above. We must show that the necessary relations hold in $\Fin\Corel^{\circ}$. Most of the required relations have been shown or follow immediately from what has been shown. The only ones to check are the idempotency laws:
\begin{itemize}
\item $(G(M') \circ G(D)\circ G(M)\circ G(D'))^2 = G(M') \circ G(D)\circ G(M)\circ G(D')$
\item $(G(M) \circ G(D')\circ G(M')\circ G(D))^2 = G(M) \circ G(D')\circ G(M')\circ G(D)$ 
\end{itemize}

\noindent We have that :
\begin{align*}
   (G(M') \circ G(D)\circ G(M)\circ G(D'))^2 &= (\mu_2 \circ d_2 \circ m_2\circ \delta_2)\circ (\mu_2 \circ d_2 \circ m_2\circ \delta_2) \\
    &= (\mu_2 \circ d_2 \circ \mu_2 \circ d_2)\circ ( m_2\circ \delta_2 \circ m_2\circ \delta_2)\\
    &= (\mu_2 \circ d_2 \circ m_2\circ \delta_2) \\
    &= G(M') \circ G(D)\circ G(M)\circ G(D')
\end{align*}

\noindent One can similarly show $(G(M) \circ G(D')\circ G(M')\circ G(D))^2 = G(M) \circ G(D')\circ G(M')\circ G(D)$. Thus such a morphism of props exists. Next we show that $F$ exists. Once again we define it on the generators as above and check the necessary relations. The only ones left to check are again the idempotency laws:

\begin{itemize}
\item $(F(M') \circ F(D)\circ F(M)\circ F(D'))^2 = F(M') \circ F(D)\circ F(M)\circ F(D')$
\item $(F(M) \circ F(D')\circ F(M')\circ F(D))^2 = F(M) \circ F(D')\circ F(M')\circ F(D)$ 
\end{itemize}

\noindent We have that:
\begin{align*}
  F(M') \circ F(D)\circ F(M)\circ F(D') &=( + \oplus \Delta^{\dagger})\circ ( \Delta \oplus +^{\dagger}) \circ (\Delta^{\dagger} \oplus +) \circ ( +^{\dagger} \oplus \Delta) \\
    &= \mathrm{id}_2
\end{align*}

\vspace{-2ex}

\noindent Since $\mathrm{id}_2\circ \mathrm{id}_2 = \mathrm{id}_2$, we get:
\begin{align*}
  (F(M') \circ F(D)\circ F(M)\circ F(D'))^2 &=F(M') \circ F(D)\circ F(M)\circ F(D'). 
\end{align*}

\noindent Similarly one can show that:
\begin{align*}
  (F(M) \circ F(D')\circ F(M')\circ F(D))^2 &= F(M) \circ F(D')\circ F(M')\circ F(D). 
\end{align*} 
Thus $F$ is a morphism of props. Note that if $f$ is a morphism in $\BondGraph$ then $f^{\dagger}$ is the vertical reflection of $f$. The dagger structure associated to both $\Fin\Corel^{\circ}$ and $\Lag\Rel_k^{\circ}$ are also vertical reflections. Thus both functors preserve the dagger structure of $\BondGraph$ and thus they are both dagger morphisms of props.
\end{proof}

We have now assembled the following diagram where the arrows are symmetric monoidal dagger functors:

\[
\begin{tikzcd}[column sep=scriptsize]
& \Lag\Rel_k^{\circ} \arrow[r, ""{name=U, below}, "i'"]{}
 & \Lag\Rel_k \\
\BondGraph  \arrow[ur, ""{name=U, below}, "F"]{}  \arrow[dr, ""{name=U, above}, "G" ' ]{} \\
&\Fin\Corel^{\circ}  \arrow[r, ""{name=U, below}, "i"]{} & \Fin\Corel  \arrow[uu, "K"'] 
\end{tikzcd}
\]
This does not commute. However, we shall prove that it commutes up to a natural transformation. The idea is that $\Fin\Corel^{\circ}$ is the category that allows us to understand bond graphs in terms of potential and current, while $\Lag\Rel_k^{\circ}$ is the category that allows us to understand bond graphs in terms of effort and flow. Then the Lagrangian relation relating potential and current to effort and flow:

$$\{(V,I,\phi_1,I_1,\phi_2,I_2) | V = \phi_2-\phi_1, I = I_1 =- I_2\}$$ 

\noindent defines a natural transformation between the functors. Recall that this comes from voltage, the difference in potential, being a type of effort. Meanwhile, flow at the end of a bond is defined when current on the two terminals are equal and opposite, i.e.\ we have a port.

\begin{theorem}\label{thm:natural}
There is a natural transformation $\alpha \maps  i'F \Rightarrow KiG$ such that  $$\alpha_1 = \{(V,I,\phi_1,I_1,\phi_2,I_2) | V = \phi_2-\phi_1, I = I_1 =- I_2\}$$ and   $\alpha_n = \bigoplus_1^n \alpha_1$ for any $n\in \BondGraph$.
\end{theorem}

\begin{proof}
For $1 \in \BondGraph$ let $\alpha_1 \maps k^{2}\to k^4$ be defined by: $$\alpha_1 = \{(V,I,\phi_1,I_1,\phi_2,I_2) | V = \phi_2-\phi_1, I = I_1 =- I_2\}.$$ 
Then we define $\alpha_n = \bigoplus_1^n \alpha_1$. Note since $\Lag\Rel_k$ is a dagger category we have that $$\alpha_1^\dagger = \{(\phi_1,I_1,\phi_2,I_2, V,I,) | V = \phi_2-\phi_1, I = I_1 = -I_2\}$$ 
and that  $\alpha_n^\dagger = \bigoplus_1^n \alpha_1^\dagger$. It can be shown that $\alpha_n^\dagger \circ \alpha_n = \mathrm{id}_{k^{2n}}$ so $\alpha_n^\dagger$ is a left inverse of $\alpha_n$. We want to show the following square commutes for any $f\maps m\to n$ in $\BondGraph$:

\[ 
 \xymatrix{
k^{2m} \ar[d]_{\alpha_m} \ar[r]^{i'F (f)} & k^{2n} \ar[d]^{\alpha_n} \\
 k^{4n} \ar[r]_{KiG(f) } & k^{4m}
} 
\] 
\noindent
Thus we need to show that $$i'F (f) \circ \alpha_m = \alpha_n \circ KiG (f)$$ which is equivalent to showing  $$\alpha_n^\dagger \circ i'F (f) \circ \alpha_m = KiG (f).$$ We proceed by showing that both sides determine monoidal functors that are equal on identities, generators, and the braiding. Since all of the morphisms in $\BondGraph$ are built up by starting with generators and repeatedly composing, tensoring, and braiding, this proves that the two functors are equal. 

We first show that the left side, $\alpha_n^\dagger \circ KiG (f) \circ \alpha_m$, defines a functor $$\alpha^\dagger \circ KiG \circ \alpha \maps \BondGraph \to \Lag\Rel_k$$ with $(\alpha^\dagger \circ KiG  \circ \alpha)m=k^{2m}$ and $$(\alpha^\dagger \circ KiG \circ \alpha)f =\alpha_n^\dagger \circ KiG (f) \circ \alpha_m$$ for a morphism $f\maps m \to n$. It is clear that identities are preserved so to prove it is a functor we need only show that $$\alpha_n^\dagger \circ KiG (fg) \circ \alpha_m= \alpha_k^\dagger \circ KiG (f) \circ \alpha_n\circ \alpha_n^\dagger \circ KiG (g) \circ \alpha_m$$ for any $f\maps n\to k$ and  $g\maps m\to n$.
We prove this inductively. First, we need that $$KiG(f) \alpha_m = \alpha_n\alpha_n^\dagger KiG(f)\alpha_m$$ for any generator $f\maps m\to n$ and for the braiding $\swap{1em}\maps 2\to 2$. These   calculations are left until the end of the proof.  Next we show that if 
\begin{align*}
   KiG(f) \alpha_a & = \alpha_b\alpha_b^\dagger KiG(f)\alpha_a\\	
   KiG(g) \alpha_c & = \alpha_d\alpha_d^\dagger KiG(g)\alpha_c\\	
\end{align*} 

\vspace{-4ex}

\noindent for two morphisms, $f$ and $g$, then $$KiG (f\oplus g)\alpha_m = \alpha_n \alpha_n^\dagger KiG(f\oplus  g)\alpha_m.$$ Let $f\maps a \to b$ and $g \maps c \to d$, where $a+c=m$ and $b+d=n$. Then we get
\begin{align*}
   KiG(f\oplus g)\alpha_m &= KiG(f\oplus g)\circ (\alpha_a \oplus \alpha_c) \\
    & = (KiG(f)\oplus\ KiG(g))\circ(\alpha_a\oplus \alpha_c )\\
    & = KiG(f)\alpha_a \oplus KiG(g)\alpha_c \\
    &= (\alpha_b\alpha_b^\dagger KiG(f)\alpha_a)\oplus 		
	(\alpha_d\alpha_d^\dagger KiG(g)\alpha_c)  \\
    &=(\alpha_b\alpha_d)(\alpha_b^\dagger\alpha_d^\dagger)
	(KiG(f)\oplus KiG(g))\alpha_a\alpha_c\\
    &=\alpha_n\alpha_n^\dagger(KiG(f)\oplus KiG(g))\alpha_m\\
  &=\alpha_n\alpha_n^\dagger KiG (f\oplus g)\alpha_m.\\
\end{align*}

\vspace{-4ex}

\noindent Finally, we show that if 
\begin{align*}
   KiG(f) \alpha_k & = \alpha_n\alpha_n^\dagger KiG(f)\alpha_k\\	
   KiG(g) \alpha_m & = \alpha_k\alpha_k^\dagger KiG(g)\alpha_m\\	
\end{align*} 

\vspace{-4ex}

\noindent for two morphisms, $f$ and $g$, then $$KiG(fg)\alpha_m=\alpha_n\alpha_n^\dagger KiG(fg)\alpha_m.$$ Let $f\maps k\to n$ and $ g\maps m\to k$.
\begin{align*}
   KiG(fg)\alpha_m & = KiG(f)KiG(g)\alpha_m \\		             
    &= KiG(f)\alpha_k\alpha_k^\dagger KiG(g)\alpha_m \\
    &= \alpha_n\alpha_n^\dagger KiG(f)\alpha_k\alpha_k^\dagger KiG (g)\alpha_m \\
    &= \alpha_n\alpha_n^\dagger KiG(f) KiG(g)\alpha_m \\
    &= \alpha_n\alpha_n^\dagger KiG(fg)\alpha_m. \\
\end{align*} 

\vspace{-4ex}

\noindent Since any morphism in $\BondGraph$ can be built from generators and braiding morphisms by repeatedly tensoring and composing, this shows that $$KiG(f)\alpha_m=\alpha_n\alpha_n^\dagger KiG (f)\alpha_m$$ for any morphism $f\maps m\to n$. Then finally we can say that 
\begin{align*}
    \alpha_n^\dagger \circ KiG (fg) \circ \alpha_m &= \alpha_n^\dagger \circ KiG (f) KiG(g) \circ \alpha_m \\
    &=\alpha_n^\dagger \circ KiG (f) \alpha_k\alpha_k^\dagger KiG (g) \circ \alpha_m 
\end{align*}

\vspace{-1.5ex}

\noindent so that we have a functor. Next we show that both functors are monoidal. It is clear that $i'F$ is a monoidal functor since it is the composite of two monoidal functors. We show that $\alpha^\dagger \circ KiG \circ \alpha$ is monoidal.

For an object $m$ in $\BondGraph$ we have $(\alpha^\dagger \circ KiG \circ \alpha)(m) = k^{2m} $ so that $$(\alpha^\dagger \circ KiG \circ \alpha)(m) \oplus (\alpha^\dagger \circ KiG \circ \alpha)(n)  = k^{2m} \oplus k^{2n}$$
and 
$$k^{2m} \oplus k^{2n} \cong k^{2m+2n} = (\alpha^\dagger \circ KiG \circ \alpha)(m+n). $$ Thus for a triple of objects $m,n,l$ the square for associativity commutes. The unit object in $\Lag\Rel_k$ is $k^0$ and is equal to $(\alpha^\dagger \circ KiG \circ \alpha)(0)$ so that unique morphism $id_{k^0}\maps k^0\to k^0$  makes the unitality squares also commute.

Since both functors are monoidal and because of the way that morphisms in $\BondGraph$ are built, the two functors are equal if they are equal on the generators, the braiding, and the identities. We give the proof for the morphism $M$.
\begin{align*}
\alpha_1^{\dagger} \circ KiG (M) &= \{(\phi_5,I_5,\phi_6,I_6,V,I) | V = \phi_6-\phi_5, I = I_5 =- I_6\}\\[-3em]
& \qquad \circ \{(\phi_1,\ldots, I_6): 
	\begin{aligned}
                            \\ 
                             \\
         \phi_1= \phi_5 \phantom{hh}\\ 
          I_1 = I_5 \phantom{hh} \\
          \phi_4 = \phi_6 \phantom{hh}
        \end{aligned}
	\begin{aligned} \\
                              \\
             I_4 = I_6 \phantom{hhhh}\\
          \phi_2 = \phi_3 \phantom{hlhh} \\\hspace{2ex}
         \phantom{hh} I_2 + I_3 = 0\}. \\
	\end{aligned}
\end{align*}

 \noindent From relation composition this results in:
\begin{align*}
\alpha_1^{\dagger} \circ KiG(M) &= \{(\phi_1,\ldots, I_4,V,I) | V = \phi_4-\phi_1, I = I_1 =- I_4 \\[-3em]
&\qquad	 \qquad \qquad  \qquad \enspace \thinspace\thinspace\thinspace \begin{aligned} \\
                            \\ 
          I_1 = I_5,   \\
        \end{aligned}
	\begin{aligned} \\
                              \\
  \phi_2 = \phi_3,  \\
	\end{aligned}
	\begin{aligned} \\
\\
     I_2 + I_3 = 0\}. \\
	\end{aligned}
\end{align*}

\noindent Then we compose with
\begin{align*}
\alpha_2 = \{(V,'I',V'',I'',\phi_1,\ldots ,I_4) | V' = \phi_2-\phi_1, I' = I_1 =- I_2 \phantom{hhhhh} \\[-3em]
\quad \begin{aligned} \\
                             \\
         \phantom{hhh} V'' = \phi_4-\phi_3,\\ 
        \end{aligned}
	\begin{aligned} \\
                              \\
         I''= I_3 = -I_4 \} \phantom{hhh} \\
	\end{aligned}
\end{align*}

\noindent which gives us  $$  \alpha_1^{\dagger} \circ KiG(M) \circ \alpha_2 = \{(V',I',V'', I'',V,I): \; V'+V''=V,  I = I' = I''\}.$$

\noindent This is the same as the Lagragian subspace $$ i'F(M) = \{(E_1,F_1,E_2,F_2,E_3 , F_3): \; E_1+E_2=E_3,  F_1 = F_2 = F_3\}.$$

Thus the two functors are equal on all morphisms in $\BondGraph$, which shows that  $$\alpha_n^\dagger \circ i'F (f) \circ \alpha_m = KiG (f)$$  for any morphism $f\in \BondGraph$. Thus $\alpha$ is a natural transformation, as desired.

Finally, we check that $KiG(f) \alpha_m = \alpha_n \alpha^\dagger_n KiG(f) \alpha_m$ when $f$ is a generator or the braiding $\swap{1em}$. Consider again the generator $M\maps 2\to 1$. We need that $$KiG(M) \alpha_2 = \alpha_1\alpha_1^\dagger KiG(M)\alpha_2.$$ \vspace{1ex} We have $KiG(M) \alpha_2  = \{(V,I,V',I',\phi_5,I_5, \phi_6,I_6)\}$ such that 
\begin{align*}
   V+V'=\phi_6-\phi_5 &&& I=I'=I_5=-I_6. \\		
\end{align*} 

\vspace{-4.5ex}

\noindent By composing with  $ \alpha_1^\dagger $ we get $ \{(V,I,V',I', V'',I'')\}$ such that 
\begin{align*}
   V+V'=\phi_6-\phi_5 &&& I=I'=I_5=-I_6 \\	
     V''=\phi_6-\phi_5  &&&      I''=I_5=-I_6.  &&&  \\
\end{align*} 

\vspace{-4.5ex}

\noindent These equations simplify to become
\begin{align*}
   V+V'=V'' &&&     I=I'=I''. &&& \\	
\end{align*}

\vspace{-4.5ex}

\noindent Now if we compose with $\alpha_1$ we get $ \{(V,I,V',I',\phi_7,I_7, \phi_8,I_8)\}$ such that
\begin{align*}
   V+V'=V'' &&&     I=I'=I'' &&& \\	
    V''=\phi_8-\phi_7  &&& I''=I_7=-I_8    \\
\end{align*} 

\vspace{-4.5ex}

\noindent which can be reduced to
\begin{align*}
   V+V'=\phi_8-\phi_7 &&&     I=I_7=-I_8.\\	
\end{align*} 

\vspace{-4.5ex}

This  is $KiG(M) \alpha_2 $ with a change of names. We can see the general pattern with just this generator. Composing with $\alpha_n\alpha_n^\dagger$ simply renames the potentials and currents without changing any of the relations. For potentials, the relations  defining the subspaces deal only with the differences in potentials. Thus by writing the difference in potential as voltage, we get relations that define the subspace in terms of voltage. However, this does nothing because voltage is the difference in potential. So when we convert back to potential, the same relations hold. A similar process ensures that the current can be renamed through the same process without changing the subspace. With similar reasoning we conclude that the equation holds for the other generators as well as the braiding.
\end{proof}

\bibliographystyle{plain}


\end{document}